\documentclass[12pt]{article}

\usepackage[paperwidth=230mm, paperheight=310mm]{geometry}
\usepackage{graphicx}
\usepackage{amssymb}
\usepackage{amsmath}
\usepackage[inline]{enumitem}
\usepackage{amsthm}
\usepackage{amsmath}
\usepackage{physics}
\usepackage{float}
\usepackage{color}   
\usepackage{longtable}
\usepackage{comment}
\usepackage{lineno}

\makeatletter
\renewenvironment{proof}[1][\proofname]{%
\par\pushQED{\qed}\normalfont%
\topsep6\p@\@plus6\p@\relax
\trivlist\item[\hskip\labelsep\bfseries#1\@addpunct{.}]%
\ignorespaces
}{%
\popQED\endtrivlist\@endpefalse
}
\makeatother

\usepackage[pdftex]{hyperref}

\newcommand{\arxivversion}[1]{}

\newtheorem{theorem}{Theorem}
\newtheorem{proposition}[theorem]{Proposition}
\numberwithin{theorem}{section}
\newtheorem{corollary}[theorem]{Corollary}
\newtheorem{lemma}[theorem]{Lemma}

\newtheorem{remark}[theorem]{Remark}

\newtheoremstyle{def}
{\topsep}
{\topsep}
{}
{0pt}
{\bfseries}
{.}
{ }
{\thmname{#1}\thmnumber{ #2}\textnormal{\thmnote{ (#3)}}}
\theoremstyle{def}
\newtheorem*{definition}{Definition}
\newtheorem*{notation}{Notation}

\numberwithin{equation}{section}

\renewcommand{\P}{\mathbb{P}}
\newcommand{\E}{\mathbb{E}}
\newcommand{\R}{\mathbb{R}}

\newcommand{\pmR}{\partial\mathcal{R}}
\newcommand{\Z}{\mathbb{Z}}
\newcommand{\Q}{\mathbb{Q}}
\newcommand{\N}{\mathbb{N}}

\newcommand{\cB}{\mathcal B}
\newcommand{\cC}{\mathcal C}
\newcommand{\cF}{\mathcal F}
\newcommand{\cE}{\mathcal E}
\newcommand{\cG}{\mathcal G}
\newcommand{\cH}{\mathcal H}
\newcommand{\cL}{\mathcal{L}}
\newcommand{\cS}{\mathcal{S}}

\newcommand{\veps}{\varepsilon}

\newcommand{\supp}{\mathrm{supp}}
\newcommand{\diam}{\mathrm{diam}}
\newcommand{\pf}[1]{{\color{black}#1}}
\renewcommand{\pf}[1]{}
\newcommand{\SUBMIT}[1]{}
\newcommand{\ARXIV}[1]{#1}

\begin{document}

\author{
	Edwin Perkins\footnote{Department of Mathematics, The University of British Columbia.  {\tt perkins@math.ubc.ca}}
	\and 
	Delphin S\'enizergues\footnote{MODAL'X, Universit\'e Paris Nanterre. {\tt dsenizer@parisnanterre.fr}}
}

\title{Total disconnectedness and percolation for the supports of super-tree random measures}

\maketitle


\begin{abstract}
Super-tree random measures (STRMs) were introduced by Allouba, Durrett, Hawkes and Perkins as a simple stochastic 
model which emulates a superprocess at a fixed time. A STRM $\nu$ arises as the a.s.\ limit of a sequence of empirical measures for a discrete time particle system  which undergoes independent supercritical branching and independent random displacement (spatial motion) of children from their parents. 
We study the connectedness properties of the closed support of a STRM ($\supp(\nu$)) for a particular choice of random displacement. 
Our main results are distinct sufficient conditions for the a.s.\ total disconnectedness (TD) of $\supp(\nu)$, and for percolation on $\supp(\nu)$ which will imply a.s.\ existence of a non-trivial connected component in $\supp(\nu)$.  
We illustrate a close connection between a subclass of these STRMs and super-Brownian motion (SBM). 
For these particular STRMs the above results give a.s.\ TD of the support in three and higher dimensions and the existence of a non-trivial connected component in two dimensions, with the three-dimensional case being critical. 
The latter two-dimensional result assumes that $p_c(\Z^2)$, the critical probability for site percolation on $\Z^2$, is less than $1-e^{-1}$. 
(There is strong numerical evidence supporting this condition although the known rigorous bounds fall just short.)   
This gives evidence that the same connectedness properties should hold for SBM. 
The latter remains an interesting open problem in dimensions $2$ and $3$ ever since it was first posed by Don Dawson over $30$ years ago.
\end{abstract}

\setcounter{tocdepth}{1}
\tableofcontents
\section{Introduction and Description of the Main Results}\label{sec:intro}

Let $M_F(E)$ denote the space of finite measures on a metric space $E$ equipped with the topology of weak convergence. 
Super-Brownian motion (SBM), $(X_t,\, t\ge 0)$, is a continuous $M_F(\R^d)$-valued random process arising as the scaling limit of critical branching random walk. 
It also is the scaling limit of a number of other models in population genetics and statistical physics, the latter for $d$ above the critical dimension (see, e.g. \cite{S99,HS03}).  
Even for the case $d=2\text{ or }3$ which will motivate this work, it arises as the limit of the rescaled voter model \cite{CDP02}, rescaled Lotka-Volterra models \cite{CMP10}, or spatial Lambda-Fleming-Viot processes (\cite{CE18,CP20}). 
Recall that $A\subset\R^d$ is totally disconnected (TD) iff it contains no connected subset with more than one element. 
We let $\supp(M)$ denote the closed support of a measure, $M$,  defined on a metric space with its Borel $\sigma$-field. 
For $d=1$, the measure $X_1$ has a continuous density (e.g., Theorem~III.4.2 of \cite{Per02}) and so $\supp(X_1)$ will contain intervals of positive length if $X_1\neq 0$. 
An elementary cluster decomposition will show $\supp(X_1)$ is a.s.\ totally disconnected if $d\ge 4$ (see Section~III.6 of \cite{Per02}).
 Whether or not $\supp(X_1)$ is a.s.\ TD in $2$ or $3$ dimensions has been an interesting open problem for over $30$ years (see, e.g., Sec. 4 of \cite{P95}).
  (The answer will be the same for $\supp(X_t)$ for any $t>0$ and any initial condition $X_0$ by general absolute continuity results (e.g. Theorem~III.2.2 in \cite{Per02}).)
  One suspects that there is a critical dimension and so the problem cannot be ``difficult'' in both unresolved dimensions.  
  We partially resolve this disconnectedness problem in dimensions $2$ and $3$ for a closely related class of models from \cite{ADHP} called super-tree random measures.
   We believe this clarifies the state of affairs for SBM itself, as conjectured below.  

To describe the models from \cite{ADHP} we start with an offspring law on $\Z_+$ given by $p=(p_k,k\in\Z_+)$, where
\begin{equation}\label{pk}
\mu=\sum_{k=0}^\infty kp_k>1\quad \text{ and } \quad \sum_{k=0}^\infty k^2p_k<\infty.
\end{equation}
Introduce an $\R^d$-valued random vector $X$ with finite mean, a parameter $\beta>0$, and set $\rho=\mu^{-\beta/2}\in(0,1)$.  
We set $\Z_+=\{0,1,\dots\}$ (unlike \cite{ADHP}). 
Define an $M_F(\R^d)$-valued discrete time process $(\nu_n, \, n\in\Z_+)$ as follows:
\begin{enumerate}[label=\emph{(\alph*)}, ref=(\alph*)]
\item At $n=0$, we set $\nu_0=\delta_0$ (i.e., a single particle starts at the origin).
\item In generation $n\in\N=\{1,2,\dots\}$ each particle splits into $k$ particles with probability $p_k$ and each ``child'' is displaced from its parent by an amount equal in law to $\rho^n X$.  
The displacements and the numbers of children are independent for each particle and each generation, and are also independent from each other.
\item \label{it:nuprops:iii} The random measure $\nu_n$ assigns mass $\mu^{-n}$ to the location of each individual in generation~$n$.
\end{enumerate}
Clearly $\mu^n\nu_n(\R^d)$ is a supercritical Galton-Watson (GW) process with offspring law $p$ and so 
$\nu_n(\R^d)\to W$ a.s.\ as $n\rightarrow \infty$, where $\P(W=0)$ is the probability that the GW process becomes extinct in finite time.
In fact Theorem~11 of \cite{ADHP} (or see Theorem~\ref{nuprops}(d) below) implies that $\nu_n$ converges a.s.\ in $M_F(\R^d)$ as $n\to\infty$ to a limit $\nu$ which we call a {\it classical super-tree random measure} (CSTRM).
In this work we will also assume a mild polynomial decay on the tail probability of $X$ (see \eqref{TCH}) which will ensure that $\supp(\nu)$ is a.s.\ compact (Theorem~\ref{nuprops}(b)).
 Unlike \cite{ADHP} we will allow $p_0>0$, which affects little aside from the fact that now $\nu=0$ may occur with positive probability. 
 We reserve the term {\it super-tree random measure}  (STRM), used in \cite{ADHP}, for a slightly larger class discussed below and constructed in Section~\ref{sec:STRM}. 

If $B\in\N^{\ge 2}$, a $B$-ary CSTRM
associated with $p$ is the CSTRM with $\cL(X)$, the law of $X$, uniform on $\{0,1,\dots,B-1\}^d$ and $\beta=\frac{2\log B}{\log \mu}$, or equivalently, $\rho=B^{-1}$. 
Intuitively a typical point in $\supp(\nu)$ is given by $S=\sum_{m=1}^\infty B^{-m}X_m\in[0,1]^d$ where $\{X_m\}$ are i.i.d. copies of $X$. 
The individual offspring displacements correspond to the $B$-ary expansion of each coordinate of $S$. In particular, $\supp(\nu)\subset [0,1]^d$.  
See Section~\ref{sec:Bary} below for a more careful discussion. 

The class of CSTRMs was introduced in \cite{ADHP} as a toy model for the study of SBM. 
We now give an intuitive comparison of $B$-ary CSTRM $\nu$ and $X_1$. First note that 
\begin{equation}\label{nudisp}
\parbox[t]{0.9\textwidth}{given $\nu_n$, the displacements being added to the particles in generation $n$ to construct $\nu$ are now scaled by $B^{-n}$.}
\end{equation}
So conditional on $\sigma(\nu_k,\ k \le n)$, $\nu$ is a sum of a $\text{Binomial }(\mu^n\nu_n(\R^d),\P(W>0))$ number of independent clusters which are scaled copies of $\nu$ conditioned to be non-zero and relocated at the parent and where the spatial scaling factor is $B^{-n}=\mu^{-\beta n/2}$ as in \eqref{nudisp}. 

For the time being we allow our SBM $(X_t)$ to have $(1+\beta)$-stable branching, for some $\beta\in (0,1]$, and a branching scale factor of $\gamma>0$ so that the non-linear equation associated with $(X_t)$ has non-linearity $\phi(\lambda)=\frac{-\gamma}{2}\lambda^{1+\beta}$ (see p. 49 of \cite{DP91}). 
Here our $\gamma$ differs from that in \cite{DP91} by a factor of $1/2$ so that if we consider the usual SBM (when $\beta=1$), 
$\gamma$ corresponds to the variance of the
critical offspring law in the branching random walk approximation (see, e.g. Theorem~II.5.1(c) of \cite{Per02}). 
Consider the cluster decomposition of $X_1$ corresponding to the ancestors at time $1-B^{-2n}$. 
By Brownian scaling  this time lag will result in the Brownian displacements agreeing with those in \eqref{nudisp}. 
By Proposition~3.5 of  \cite{DP91}, conditional on $\sigma(X_t,t\le 1-B^{-2n})$, the measure $X_1$ is a sum of a Poisson number of independent clusters, where the Poisson  has mean $X_{1-B^{-2n}}(\R^d)((2/\gamma)B^{2n})^{1/\beta}=c\mu^n X_{1-B^{-2n}}(\R^d)$.
Proposition~3.5(b) of \cite{DP91} also identifies the laws of the individual clusters making up $X_1$, which again are scaled relocated copies of the ``cluster law'' associated with SBM described in Section~\ref{sec:SBM} below (it is the canonical measure of $X_1$ conditioned to be non-zero).    

The above cluster descriptions of $X_1$ and $\nu$ are very similar and in fact underlie a number of common properties of the two processes.  
For example, recall that $\dim(\supp(X_1))=(2/\beta)\wedge d$ a.s.\ on $\{X_1\neq 0\}$ where $\dim(A)$ denotes the Hausdorff measure of a set $A\subset\R^d$, and that for $d<\frac{2}{\beta}$, $X_1$ a.s.\ has a density. 
See Theorems~9.3.3.1 and 9.3.3.5 of \cite{D93} for the former, and \cite{F88} and Theorem~8.3.1 of \cite{D93} for the latter.  
For $B$-ary CSTRM the a.s.\ existence of a density for $d<\frac{2}{\beta}$ follows from Theorem~3 of \cite{ADHP} (the fact that we now allow $p_0=0$ does not affect the proof).
This implies that $\supp(\nu)$ has positive Lebesgue measure a.s.\ on $\{\nu\neq 0\}$.  The fact that 
\begin{equation}\label{dimforCSTRM}\dim(\supp(\nu))=(2/\beta)\wedge d\ \text{ a.s.\ on $\{\nu\neq 0\}$}
\end{equation}
follows from Theorem~\ref{Hdim} below which gives this result for our class of STRMs, extending a result for CSTRMs in \cite{ADHP}. 
We refer to the setting where $\frac{2}{\beta}=d$ (i.e. $\mu=B^d$) as the \emph{critical case}.

If $h:[0,\delta)\to[0, \infty)$ is a continuous strictly increasing function for some $\delta>0$ with $h(0)=0$, we let $h\text{-}\mathrm{m}(A)$ denote the $h$-Hausdorff measure of a set $A\subset \R^d$. 
In Section~\ref{sec:Bary} we refine the above dimension results and prove the following for a $B$-ary CSTRM:
\begin{proposition}\label{baryhmeas}
\begin{enumerate}[label=\emph{(\alph*)}, ref=(\alph*)]
\item\label{it:baryhmeas:a} In the critical case $\frac{2}{\beta}=d$ we have $x^d \log(1/x)\text{-}\mathrm{m}(\supp(\nu))<\infty$ a.s., and in fact is an integrable r.v.
In particular (recall also \eqref{dimforCSTRM}), $\supp(\nu)$ is a compact Lebesgue null set of full dimension a.s.\ on $\{\nu>0\}$.
\item\label{it:baryhmeas:b} If $\frac{2}{\beta}<d$, then $x^{\frac{2}{\beta}}$-$\mathrm{m}(\supp(\nu))<\infty$ a.s., and in fact is an integrable r.v.
\end{enumerate}
\end{proposition}

In Section~\ref{sec:Bary} we derive the following hitting estimates for $B(y,r)$, the open Euclidean ball in $\R^d$, centered at $y$ and of radius $r$. 
Again we are in the setting of a $B$-ary CSTRM, $\nu$.
\begin{proposition}\label{hitprob2}
There are constants $0<c_{\ref{hitprob2}}\le C_{\ref{hitprob2}}$, depending on $(d,B,\cL(Z))$ such that:
\begin{enumerate}[label=\emph{(\alph*)}, ref=(\alph*)]
	\item\label{it:hitprob2:a} If $\frac{2}{\beta}=d$, then for all $y\in[0,1]^d$ and $r\in(0,\frac{1}{2}]$, $\frac{c_{\ref{hitprob2}}}{\log(1/r)}\le \P(\nu(B(y,r))>0)\le \frac{C_{\ref{hitprob2}}}{\log(1/r)}$.
	\item\label{it:hitprob2:b} If $\frac{2}{\beta}<d$, then for all $y\in[0,1]^d$ and $r\in(0,1]$, $$c_{\ref{hitprob2}}r^{d-\frac{2}{\beta}}\le \P(\nu(B(y,r))>0)\le C_{\ref{hitprob2}}r^{d-\frac{2}{\beta}}.$$
\end{enumerate}
\end{proposition}
\begin{remark} 
If $\beta=1$ these are the same asymptotics that hold for SBM, $X_1$, at time $1$ say, under its canonical measure. 
For $d=2$ this is Theorem~2 of \cite{LG94} and for $d>2$ see Theorem~3.1 of \cite{DIP89}. 
(The latter reference works with the law of SBM, not the canonical measure, but the same PDE arguments apply in the latter setting.)  For (b) when $\beta<1$, the analogous bounds for SBM (with $\frac{2}{\beta}<d$) can be found in Theorem~4.3 of \cite{He13}. The situation for SBM with $\frac{2}{\beta}=d$ seems unresolved but polar set results in Theorem 3.5(iii) of \cite{Dyn92}  suggests the same bounds as in (a) hold but with $(\log (1/r))^{-1/\beta}$ in place of $(\log (1/r))^{-1}$. 

Both of the results in Proposition~\ref{baryhmeas} are well-known for SBM $X_1$ with $\beta=1$ where exact Hausdorff measure functions are even known (see Theorem~III.3.8 of \cite{Per02} and \cite{LGP}). For $\beta<1$ and $\frac{2}{\beta}<d$, the finiteness of $x^{\frac{2}{\beta}}-m(\supp(X_1)$ follows easily from the SBM analogue of Proposition~\ref{hitprob2}(b), noted above (as in, e.g., the proof of Proposition~\ref{baryhmeas} (a) in Section~\ref{sec:Bary}).
\end{remark}

Although we believe the connection between SBM and $B$-ary CSTRMs is quite close for $\beta=1$, this  connection for $\beta\neq1$ is a bit more tenuous. For example, $B$-ary CSTRMs will have finite second moments while SBM for $\beta<1$ will not.  Henceforth SBM will refer to the standard $\beta=1$ case.

In Section~\ref{sec:SBM} below we will show in fact that SBM (recall now $\beta=1$)  satisfies a construction very similar to that of $\nu$ above, except that the law of the sibling locations born in each generation are given by a symmetric (in the $k$ vectors) law $Q_k$ on $(\R^d)^k$ depending on the number of offspring $k$.  
The sibling locations may be dependent for a given $k$ but, as before, the displacements and offspring laws remain independent for distinct generations and distinct parents. 
In addition, all of the $Q_k$ marginals have the same law $\cL(X)$. 
We call such a random measure a \emph{super-tree random measure} (STRM). 
In Section~\ref{sec:STRM} we show the main properties derived for CSTRMs in \cite{ADHP} remain valid for this more general class. 
For example, if $S=\sum_{m=1}^\infty \rho^mX_m$ where $\{X_m\}$ are i.i.d.\ copies of $X$, then $\nu$ has mean measure $\E(\nu(A))=\P(S\in A)$ for all Borel $A\subset\R^d$ (Theorem~\ref{nuprops}(c)).  
Then in Section~\ref{sec:SBM} (see Theorem~\ref{thm:SBMSTRM}) we show that for any $\mu>1$, $\frac{2}{\gamma}X_1$  under its cluster law is in fact a STRM with $\beta=1$, with $\cL(Z)$ geometric with mean $\mu$, and with $\cL(X)$ a centered $d$-dimensional Gaussian with covariance $(\mu-1)I_d$, thus making the connection between SBM and STRMs even closer.  
The factor of $\frac{2}{\gamma}$ is only needed for normalization because as noted above the mean measure of any STRM is a probability.

Our main results (Theorems \ref{thm:perc2} and \ref{thm:totdisc} below) concern connectedness properties of $\supp(\nu)$.  
Let $\text{comp}(x)$ denote the connected component in $\supp(\nu)$ containing a point $x\in\supp(\nu)$ and use the same notation for $X_1$ in place of $\nu$. 
For $X_1$, Tribe \cite{T91} proved that if $d>2$ then w.p.$1$ for $X_1$-a.a.\ $x$, $\text{comp}(x)=\{x\}$. 
His proof just fails if $d=2$. 
In Section~\ref{sec:wdisc} we adapt Tribe's argument to prove the analogue for $B$-ary CSTRMs, which again agrees with SBM when $\beta=1$.
\begin{theorem}\label{thm:nutribe}
If $\nu$ is a $B$-ary CSTRM with $d>\frac{2}{\beta}$ then w.p.$1$ for $\nu$-a.a. $x$, $\text{comp}(x)=\{x\}$. 
\end{theorem}
\noindent Note that the conclusion above does not exclude the existence of a non-trivial $\nu$-null connected component of $\supp(\nu)$. 
In fact a natural example of this behaviour will follow from one of our main results below (see Remark~\ref{rem:ccomp}). 

To describe our result on the existence of a non-trivial connected component in $\supp(\nu)$ we need some notation and terminology. 
Let $\Z^d$ denote the standard cubic lattice (each site has $2d$ neighbours) and let $p_c(\Z^d)$ denote the critical probability for site percolation on $\Z^d$.  
Later we will use the same notation for other graphs.  
See Section~1.6 of \cite{G99} for basic information about site percolation on $\Z^d$.

\begin{definition}
Following \cite{FG92}, we say a subset $S$ of $[0,1]^d$ percolates iff $S$ contains a connected set which intersects both $\{0\}\times [0,1]^{d-1}$ and $\{1\}\times [0,1]^{d-1}$. 
\end{definition}
We will only use this definition for the closed support of a $B$-ary CSTRM, where it is particularly well-suited since it is a subset of $[0,1]^d$.

\begin{theorem}\label{thm:perc2} Assume $\nu$ is a $B$-ary CSTRM associated with the Poisson distribution and $d\ge 2$. There is a $B_0=B_0(d)\in\N^{\ge 2}$ so that the following hold:
\begin{enumerate}[label=\emph{(\alph*)}, ref=(\alph*)]
	\item\label{it:perc2:a} Let $d\ge 3$. 
	If $B\ge B_0$ there is an $\veps=\veps(d,B)>0$ such that if $0<\beta\le \frac{2}{d}+\veps$, then 
	\begin{equation}\label{percandnottdthm}
\P(\supp(\nu) \text{ percolates})>0 \quad \text{ and } \quad \P(\supp(\nu)\text{ not TD}\,|\, \nu\neq 0)=1.
\end{equation}
	In particular, \eqref{percandnottdthm} holds for the critical case $\beta=\frac{2}{d}$ and any $B\ge B_0$.
	\item\label{it:perc2:b} Let $d=2$ and assume
	\begin{equation}\label{pccond}
		p_c(\Z^2)<1-e^{-1}.
	\end{equation}
If $B\ge B_0$ there is an $\veps=\veps(2,B)>0$ such that if $0<\beta\le 1+\veps$, then \eqref{percandnottdthm} holds.
In particular, \eqref{percandnottdthm} holds for the critical case $\beta=1$ and any $B\ge B_0$.
\end{enumerate} 
\end{theorem}

\begin{remark}\label{rem:perclit}
There is a large literature on approximating $p_c(\Z^2)$ by simulations and by using exact computations for crossing large finite squares. 
For example, Table 1 of \cite{M22}  gives a chronology of progressively more precise  approximations culminating in the (non-rigorous) value $p_c(\Z^2)=.59274605\dots$.  
This gives strong evidence that $p_c(\Z^2)<1-e^{-1}=.632\dots$ is in fact true.  
The best rigorous bounds we know of are $.556<p_c(\Z^2)<.679492$, where the lower bound is from \cite{BE96} and the more relevant (for our purposes) upper bound is proved by Wierman in \cite{W95}.  
So although \eqref{pccond} is undoubtedly true, a proof remains just out of reach. \end{remark}
\begin{remark}\label{rem:ccomp}
Under the hypotheses of Theorem~\ref{thm:perc2} we see from that result and Theorem~\ref{thm:nutribe} that for $B\ge B_0$, $d\ge 2$ and some $\veps=\veps(d,B)>0$, if $\frac{2}{d}<\beta\le \frac{2}{d}+\veps$, then w.p.$1$ $\text{comp}(x)=\{x\}$ for $\nu$-a.a. $x\in\supp(\nu)$, but $\supp(\nu)$ is not TD whenever $\nu\neq 0$. Here we assume \eqref{pccond} if $d=2$. Another STRM for which the same conclusion holds (without any additional condition for $d=2$) arises in fractal percolation (see Remark~\ref{rem:weak disconnectedness result also holds for fp}).
\end{remark}
\begin{remark}
The question of whether a $B$-ary CSTRM with $B=2$ and reproduction distribution identically $4$ in dimension $d=2$ (hence with $\beta=1$) percolates with positive probability was raised by Nicolas Curien in 2017 in an open problem session at a workshop held at Bellairs institute in Barbados. 
This question was motivated by the potential link between this model and the trace of Brownian motion in 2D, inspired by \cite{Per96}. 
Unfortunately, Theorem~\ref{thm:perc2}\ref{it:perc2:b} does not cover this case, as the result only holds for $B$ large enough.
\end{remark}
The proof of Theorem~\ref{thm:perc2} uses a comparison with fractal percolation (see Section~\ref{sec:percn} for more on these random Cantor sets) and an asymptotic result in $B$ for $p_c$ for percolation on the resulting random Cantor set due to Falconer and Grimmett \cite{FG92}. 
The latter leads to the condition \eqref{pccond} on $p_c(\Z^2)$ when $d=2$. 
(Our proof uses the same bound \eqref{pccond} on $p_c(\Z^d)$ for $d\ge 3$  but here it is a known bound.)
We believe this condition for $d=2$ is spurious and the hypotheses in Theorem~\ref{thm:perc2} are not sharp as the comparison with fractal percolation leaves a gap. 
The Poisson offspring condition is used to couple the CSTRM with a Cantor set obtained through fractal percolation but should also not be needed for the result to hold. 
The proof also uses the fact that the first conclusion in \eqref{percandnottdthm} implies the second (see Lemma~\ref{lem:01}).
 Whether or not these conditions are equivalent remains open.

In Section~\ref{secTDBary} we establish the following simple sufficient condition for total disconnectedness of the support of a $B$-ary CSTRM which complements Theorem~\ref{thm:perc2}.  
\begin{theorem}\label{thm:totdisc}
Let $\nu$ be a $B$-ary CSTRM in $\R^d$ associated with $p$.  
If $d\ge \frac{4}{\beta}-1$, then $\supp(\nu)$ is TD in $\R^d$ a.s.
In particular for $\beta=1$ this holds for $d\ge 3$.
\end{theorem}
\begin{remark}\label{TDforSTRM}
\emph{(a)} A natural way to disconnect $\supp(\nu)$ into disjoint closed sets of small diameter is to take cubes of edge length $B^{-n}$ with ``lower left-hand" corners given by the points in $\supp(\nu_n)$. 
This follows the proof of TD for SBM for $d\ge 4$ and would only lead to TD for $d\ge \frac{4}{\beta}$ (to ensure asymptotic disjointness of the cubes).  
In the proof of  Theorem~\ref{thm:totdisc} a more involved decomposition is used and  the proof proceeds using a fairly tight supermartingale convergence argument. 
We believe that the condition $d\ge \frac{4}{\beta}-1$ may be sharp for a.s.\ TD of $\supp(\nu)$.  
If $\beta=1$ this suggests that $d=3$ is critical for TD to hold for $\supp(\nu)$.  

\emph{(b)} Assume the hypotheses of Theorem~\ref{thm:perc2}. In the critical case $2/\beta=d$ we know from Proposition~\ref{baryhmeas}(a) that $\supp(\nu)$ has finite $x^d\log(1/x)$-Hausdorff measure and so is a compact Lebesgue-null set which may percolate by Theorem~\ref{thm:perc2} (with the caveat of assuming \eqref{pccond} if $d=2$).  In fact by Theorem~\ref{thm:perc2} percolation can still occur
for $2/\beta$ slightly smaller than $d$ in which case $\supp(\nu)$ is even more singular as it has finite $x^{2/\beta}$-Hausdorff measure by Proposition~\ref{baryhmeas}\ref{it:baryhmeas:b}. 
Hence $\beta=\frac{2}{d}$ is not critical for being TD. If we set $\beta=1$ we conclude that the argument for proving TD for $d=2$ is ``not critical". See Remark~\ref{critpercdef} for a  more precise definition of ``critical".
\end{remark}
		
\begin{remark}\label{SBMconjs}  In view of the close connections between $B$-ary CSTRMs for $\beta=1$  and super-Brownian motion, $X_1$, described above, we believe that one can make the following analogous conclusions (to those in Remark~\ref{TDforSTRM}) for SBM,  $X_1$:\\
\emph{(a)} $d=3$ is critical for TD of $\supp(X_1)$, and in this case we guess that $\supp(X_1)$ is TD.\\
\emph(b)  We conjecture that in two dimensions $\supp(X_1)$ is not TD a.s.\ on $\{X_1\neq 0\}$ and in this case ``the result is not even critical". 

Although the connection between $B$-ary CSTRMs and SBM with $1+\beta$-stable branching is more tenuous, it would be interesting to see if the disconnection results for $B$-ary CSTRMs are valid for SBM in this more general setting. 
For example: Does TD of $\supp(X_1)$ hold a.s. for SBM with $1+\beta$-stable branching when $d\ge \frac{4}{\beta}-1$?
\end{remark}

\section{Super-Tree Random Measures}\label{sec:STRM}

We start with some notation from \cite{H} and \cite{ADHP}.
\begin{notation}
For $n\in\N$, $F_n=\N^{\{1,\dots,n\}}$, $F_0=\{0\}$, and $F=\cup_{n=0}^\infty F_n$.  We write $|f|=n$ if $f\in F_n$.  Set $I=\N^\N$, equipped with the product topology ($\N$ is given the discrete topology) and its Borel $\sigma$-field $\cB(I)$. If $f\in F_n$ and $1\le m\le n$ we let $f|m\in F_m$ denote the first $m$ coordinates of $f$ and similarly define $i|n\in F_n$ for $i\in I$. We set $i|0=f|0=0$ for $f,i$ as above.
We use $|A|$ to denote the cardinality of a finite set $A$. 
For $i,j\in I$ let $\kappa(i,j)=\sup\{m:i|m=j|m\}\in\Z_+\cup\{\infty\}$, where $\sup\emptyset=0$, and define $d(i,j)=2^{-\kappa(i,j)}$. Then $d$ is a metric on $I$, in fact an ultrametric, inducing the product topology on $I$. 
\end{notation}

We think of $F_n$ as the set of potential individuals in generation $n$ and $i\in I$ as a possible infinite line of descent.  The offspring law on $\Z_+$ is given by $p=(p_k,k\in\Z_+)$, satisfying \eqref{pk}.
We use $Z$ to denote a generic r.v.\ with law $p$.  Let $\{Z^f:f\in F\}$ be a collection of i.i.d.\ random variables with law $p$. Here $Z^f$ will be the number of children of individual $f$.  If $f\in F_m$ for some $m\in\Z_+$ and $n\ge m$, introduce the set of descendants of $f$ in generation $n$ given by
\ARXIV{
\begin{align}\label{eq:def K^f_n}
K^f_n=\{g\in F_n:g|m=f\text{ and }g_{j+1}\le Z^{g|j}\text{ for all }m\le j<n\}.
\end{align}
}
\SUBMIT{
	\begin{align*}
		K^f_n=\{g\in F_n:g|m=f\text{ and }g_{j+1}\le Z^{g|j}\text{ for all }m\le j<n\}.
	\end{align*}
}
Note that $n\mapsto |K^f_{m+n}|$, $n\in\Z_+$, is a Galton-Watson branching process starting at $1$ with offspring law $\cL(Z)$. 
For $f$ as above we therefore have (e.g. see Theorem~8.1 and Remark~1 in Section~8.1 of \cite{Har})
\begin{multline}\label{Wfprop}W^f=\lim_{n\to\infty}\mu^{-n}|K_n^f|\text{ exists a.s.\ and in $L^2$}, \quad   \E(W^f)=\mu^{-m},\\
\text{ and } \  (W^f>0) \ \text{ if and only if }  \  (|K^f_n|>0 \text{ for all }n\ge m), \  \text{a.s.}
\end{multline}
We define the set of possible finite and infinite ancestral lines of descent of $f\in F_m$, by
\ARXIV{
\begin{align}\label{eq:def D(f)}
D(f)=\{g\in F:g|m=f\}\quad \text{ and } \quad  \overline D(f)=\{i\in I:i|m=f\},
\end{align}
}
\SUBMIT{
	\begin{align*}
		D(f)=\{g\in F:g|m=f\}\quad \text{ and } \quad  \overline D(f)=\{i\in I:i|m=f\},
	\end{align*}
}
respectively.
Clearly for $i\in I$, $\overline D(i|m)$ is $\overline B(i,2^{-m}):=\{j\in I:d(i,j)\le 2^{-m})$.  
We can, and shall, assume that $Z^f(\omega)<\infty$ for all $f\in F$ and all $\omega$, which implies $|K^0_n|<\infty$ for all $n$ and $\omega$.
The set of finite lines of descent is given by $K_\infty^0:=\cup_{n=1}^\infty K^0_n$ while the infinite lines of descent in our population is given by the random set
\ARXIV{
\begin{equation}\label{Kdef}
K=\{i\in I: i_{j+1}\le Z^{i|j}\text{ for all }j\in\Z_+\}.
\end{equation}
}
\SUBMIT{
\begin{equation*}
		K=\{i\in I: i_{j+1}\le Z^{i|j}\text{ for all }j\in\Z_+\}.
\end{equation*}
}
If $M_m=\max\{Z^f:f\in K^0_{m-1}\}$ (finite for all $m\in\N$, $\omega\in \Omega$) then 
\begin{align}\label{Kbounds}
&K_n^0\subset\prod_{m=1}^n\{1,\dots,M_m\}\text{ for all }n\in\N,\text{ and }K
\text{ is a closed subset of }\prod_{m=1}^\infty\{1,\dots,M_m\}.
 \end{align}
Therefore $K$ is a (random) compact subset of $I$ by Tychonoff's theorem.

To define a random finite measure on the possible infinite lines of descent, fix $\omega$ outside a null set, $\Gamma^c$, so that the almost sure conclusions in \eqref{Wfprop} hold for all $f\in F$, and set
\begin{equation}\label{Ydef1}
Y(\omega,\overline D(f))=\begin{cases}
W^f(\omega)\  \ &\text{ if}\ f\in K^0_m \text{ for some }m\in\Z_+,\\
0&\text{ otherwise}.
\end{cases}
\end{equation}
In particular if $f=0$ we see that 
\begin{equation}\label{Ymass}
	Y(I)=W^0.
\end{equation}
The following result is essentially taken from \cite{H}. 
\begin{proposition}\label{Y} For $\omega\in\Gamma$, $Y(\omega,\cdot)$ extends uniquely to a finite random measure on $\cB(I)$ such that 
$\supp(Y)=K$ and $Y\neq 0$ iff $|K^0_n|>0$ for all $n\in\N$.
\end{proposition}
\begin{proof} The existence of $Y$ is in \cite{H}. The above discussion shows that \eqref{Ydef1} gives $Y$ on the set of closed balls in $I$. This is a class of sets that contains $I$ (recall $\overline D(0)=I$) and is closed under finite intersections in our ultrametric setting if one adds the empty set.  It follows that $Y$ is uniquely determined on $\cB(I)$, which is generated by the closed balls. Fix $\omega\in\Gamma$.  If $i\in K$, then for all $n\ge m\in\Z_+$, $i|n\in K^{i|m}_n$, and so by \eqref{Wfprop} with $f=i|m$ and the above definition, we have $Y(\overline B(i,2^{-m}))=Y(\overline D(i|m))=W^{i|m}>0$. Letting $m\to\infty$ we see that $i\in\supp(Y)$. 
Assume next that $i\notin K$. Then $i_{m+1}>Z^{i|m}$ for some $m\in\N$ which means that $i|(m+1)\notin K^0_{m+1}$. 
This implies that $Y(\overline B(i,2^{-m-1}))=Y(\overline D(i|(m+1)))=0$ and so $i\notin\supp(Y)$.  We have proved that $\supp(Y)=K$.  The final conclusion is immediate from $Y(I)=W^0>0$ iff $|K^0_n|>0$ for all $n$ which follows from \eqref{Wfprop} and our choice of $\omega$.
\end{proof}
We set $Y=0$ on the null set $\Gamma^c$.

Recall that $\rho=\mu^{-\beta/2}\in(0,1)$ for a given parameter $\beta>0$.  
To describe the spatial dynamics let $T=\{(m,k)\in\N\times\N:m\le k\}$ and let $Q^s$ (the superscript $s$ stands for \emph{spatial}) denote a law on $(\R^d)^T$ so that the coordinate vectors $\{X_{m,k}:(m,k)\in T\}$ are identically distributed under $Q^s$. 
We will let $X$ denote a generic random vector in $\R^d$ with this common distribution and assume $X$ has a finite mean.
 For each $k\in\N$ we let $Q_k$ denote the symmetrised law of $(X_{1,k},\dots,X_{k,k})$ on $(\R^d)^k$ under $Q^s$, that is,
\begin{equation}\label{eq:def symmetric displacement laws}
Q_k(\cdot)=\sum_{\pi\in S_k}(k!)^{-1}Q^s((X_{\pi 1,k},\dots,X_{\pi k,k})\in\cdot).
\end{equation}
For all $k\in \N$ the marginals of $Q_k$ are equal to $\cL(X)$, and $(Q_k, \, k\in\N)$ can be any family of symmetric laws on $(\R^d)^k$, $k\in\N$, (symmetric in the $k$ variables) with a common marginal law on $\R^d$ over all $k$. 
We call $\{Q_k\}$ the \emph{symmetric displacement laws} associated with $Q^s$.

Let $\left\lbrace X^f, \ f\in F\right\rbrace$ denote a collection of i.i.d.\ random vectors with common law $Q^s$ which is independent from the collection $\lbrace Z^f, \ f\in F\rbrace$.   
Here $\rho^{|f|+1}X^f_{m,k}$ will be the displacement of the $m$th child of individual $f$ from the location of its parent, $f$, given that $f$ has $k$ children. 
At times it will be convenient to assume the branching variables, $\{Z^f\}$, are defined on $(\Omega_b,\cF_b,\P_b)$, the displacement variables, $\{X^f\}$, are defined on $(\Omega_d,\cF_d,\P_d)$, and we work on product space under $\P=\P_b\times\P_d$. Sometimes we will write $X^f_{m,k}$ even for $m>k$, which we define to be $X^f_{k,k}$, although the choice of extension will not affect the processes of interest.

In \cite{ADHP} it was assumed that $X_{m,k}=X_m$ where $\{X_m\}$ is i.i.d., that is, the sibling locations are independent of each other and  of the total offspring number.  
Now we allow the locations of the siblings to be dependent and also allow this law to depend on the total number of siblings. 
There still is 
complete independence of the dynamics for distinct parents, and as a result the required changes in the proofs of the relevant results in \cite{ADHP} will be fairly easy to handle, as we show below. 

The position in $\R^d$ of particle $f\in F$ is given by
\begin{equation}\label{Sfdefn}S^f=\sum_{m=1}^{|f|} \rho^mX^{f|(m-1)}_{f_m,Z^{f|(m-1)}},\end{equation}
where the empty sum $S^0$ is $0$, and we are primarily interested in the above for $f\in K_\infty^0$. 
Let $\{X_m, \ m\in\N\}$ denote a generic i.i.d.\ sequence with $X_m$ equal in law to $X$ and set 
\begin{equation}\label{SSndef} S_n=\sum_{m=1}^n\rho^mX_m \quad \text{and} \quad S=\sum_{m=1}^\infty \rho^m X_m.
\end{equation}
The latter sum converges a.s.\ since $X$ has a finite mean, and the same holds for 
\[T_n=\sum_{m=n+1}^\infty \rho^mX_m.\]
If $f\in F_n$ it is easy to condition on the values of $Z^{f|(m-1)}$ for $1\le m\le n$ to see that $S^f$ is equal in law to $S_n$.  For $i\in I$ fixed, let 
\begin{equation}\label{Sidefn}S^i=\sum_{m=1}^\infty\rho^mX^{i|(m-1)}_{i_m,Z^{i|(m-1)}},\end{equation}
and note by the above reasoning it is equal in law to $S$ and the sum converges a.s.\ for our fixed line of descent $i$. To extend this to all values of $i$ in $K$ simultaneously we proceed as in \cite{ADHP} and introduce a tail condition on $X$ depending on a positive parameter ${\zeta}$:
\begin{equation}\label{TC}\tag*{$(TC)_\zeta$}
\text{there is a }C_{\zeta}>0\text{ such that }\P(|X|>r)\le C_{\zeta} r^{-{\zeta}}\text{ for all }r>0,
\end{equation}
where for any $x\in \R^d$, we write $|x|$ for the Euclidean norm of $x$.
Clearly \ref{TC} holds for ${\zeta}\le 1$ by Markov's inequality because $X$ has finite mean.
\begin{lemma}\label{CauchyC} Assume \ref{TC} for some $\zeta>\frac{2}{\beta}$. If $\Delta_n=\sup\{|S^g-S^{g|n}|:g\in K^0_\infty, |g|\ge n\}$, then 
\begin{equation*}
	\limsup_{n\to\infty} \frac{1}{n}\log \Delta_n\le \log(\rho\mu^{1/{\zeta}})\quad  \text{a.s.} 
\end{equation*}
\end{lemma}
\noindent For independent offspring locations and $p_0=0$ this is Equation (4.4) appearing in Lemma~4 of \cite{ADHP}. 
That simple proof goes through unchanged in the present dependent setting since it is based on simple union bounds.
We omit the argument.
\begin{proposition}\label{unifcontK} Assume \ref{TC} for some ${\zeta}>{2\over \beta}$. The following hold with probability one:
	\begin{enumerate}[label=\emph{(\alph*)}, ref=(\alph*), leftmargin=*]
		\item For all $i\in K$, the series defining $S^i$ converges.
		\item\label{it:uniform convergence Si} Let $\veps>0$ satisfy $(\rho+\veps)\mu^{1/{\zeta}}<1$ (it exists by hypotheses). Then there is an $N\in \N$ so that for $n\ge N$, $\sup_{i\in K}|S^i-S^{i|n}|\le ((\rho+\veps)\mu^{1/{\zeta}})^n$.  In particular $S^{i|n}$ converges to $S^i$ uniformly on $K$.
		\item\label{it:imapstoSi is continuous} The map $i\mapsto S^i$ is continuous on $K$ and the set $\{S^i:i\in K\}$ is compact in $\R^d$.
	\end{enumerate}
\end{proposition}
\begin{proof} Choose $\veps>0$ as in \ref{it:uniform convergence Si}.  Next choose $\omega$ outside a null set so that the conclusion of Lemma~\ref{CauchyC} holds. Then there is a $N\in\N$ so that for $m\ge n\ge N$,
\[ \sup_{i\in K}|S^{i|n}-S^{i|m}|\le \Delta_n\le ((\rho+\veps)\mu^{1/{\zeta}})^n.\]
This implies $S^{i|n}$ converges uniformly on $K$ as $n\to\infty$ to $S^i$, and letting $m\to\infty$ in the above gives the bound in \ref{it:uniform convergence Si}. 
The other conclusions now follow easily from this, the trivial continuity of $i\mapsto S^{i|n}$ on $I$, and the  compactness of $K$.
\end{proof}

\begin{remark} Recall the metric $d(i,j)=2^{-\kappa(i,j)}$ on $I$. It follows easily from \ref{it:uniform convergence Si} that w.p. $1$, the map $i\mapsto S^i$ is $r$-H\"older continuous on $K$ for any $r<\Bigl({\beta\over 2}-{1\over {\zeta}}\Bigr){\log \mu\over \log 2}$.
\end{remark}
\begin{equation}\label{TCH}
\text{{\bf Henceforth we will assume that \ref{TC} holds for some ${\zeta}>\frac{2}{\beta}$.} }
\end{equation}
\begin{definition}
The super-tree random measure (STRM) $\nu\in M_F(\R^d)$ associated with $(p,Q^s,\beta)$ is defined by
\begin{equation}\label{nudef}
\nu(A)=\int_K 1(S^i\in A)\,\dd Y(i).
\end{equation}
\end{definition}

 This is the definition used in \cite{ADHP} where $Q^s$ is given by setting $X_{i,k}=X_i$ for a single i.i.d.\ sequence $(X_i)$. 
 We slightly change  the terminology introduced there and refer to $\nu$ in that particular setting as a \emph{classical} super-tree random measure (CSTRM).  

Recalling the laws $Q_k$ from \eqref{eq:def symmetric displacement laws}, we see that $Q_k$ is the (symmetric) distribution of the unscaled displacements of the offspring from their parent conditioned on there being $k$ offspring. We shall see below that the law of $\nu$ depends on $Q^s$ only through the collection $\{Q_k\}$ (Remark~\ref{rem:Qklaw}) and so we also call $\nu$ the STRM associated with $(p,\{Q_k\},\beta)$.

Introduce the discrete time filtration $\cG_n=\sigma(Z^f,X^f:|f|<n)$ for $n\in\Z_+\cup\{\infty\}$, and a $(\cG_n)$-adapted $M_F(\R^d)$ process, $\nu_n=\sum_{f\in K_n^0}\mu^{-n}\delta_{S^f}$ for $n\in\Z_+$. Clearly $\nu_n$ is the suitably normalized empirical distribution of the population in generation $n$. We also introduce a non-adapted process, $\tilde \nu_n=\sum_{f\in K^0_n}W^f\delta_{S^f}$. 

\begin{lemma}\label{tildenu} $\tilde \nu_n\to \nu$ a.s.\ as $n\to \infty$.
\end{lemma}
\begin{proof} 
Note that by \eqref{Ydef1} for bounded measurable $\phi$ on $\R^d$,
\begin{equation}\label{tilderep}
\tilde\nu_n(\phi)=\sum_{f\in K^0_n}\nu(\overline D(f))\phi(S_f)=\int\phi(S^{i|n})\dd Y(i).
\end{equation}
Write $\phi\in\text{Lip}_1$ iff $\phi:\R^d\to \R$ satisfies $|\phi(x)-\phi(y)|\le |x-y|\wedge 1$. 
Then \eqref{tilderep} shows that 
\begin{align*}
\sup_{\phi\in\text{Lip}_1}|\nu(\phi)-\tilde \nu_n(\phi)|&=\sup_{\phi\in\text{Lip}_1}\Bigl|\int \phi(S^i)-\phi(S^{i|n})\,\dd Y(i)\Bigr|\\
&\le\int_K |S^i-S^{i|n}|\wedge 1\,\dd Y(i)\\
&\to 0\text{ a.s.\ as }n\to\infty,
\end{align*}
where the last two lines hold by Proposition~\ref{Y} and Proposition~\ref{unifcontK}\ref{it:uniform convergence Si}, respectively. The left-hand side of the above defines the Vasherstein metric on $M_F(\R^d)$ which metrizes the topology of weak convergence (see e.g. p. 150 of \cite{EK}) and so the result follows.
\end{proof}
The next result is the analogue of Theorem~1 from \cite{ADHP} in our dependent setting. Note also that now $\P(\nu=0)$ may be positive because we allow $p_0=0$.
\begin{theorem}\label{nuprops} The following holds:
\begin{enumerate}[label=\emph{(\alph*)}, ref=(\alph*)]
\item\label{it:nuprops:a} With probability $1$: \ $(\nu(\R^d)>0)$ \  iff \ $(K^0_n\neq\emptyset \text{ for all }n) $   \ iff \  $(\nu_n\neq 0\text{ for all }n)$.
\item\label{it:nuprops:b} $\supp\,(\nu)=\{S^i:i\in K\}$ and, in particular, $\supp\,(\nu)$ is compact a.s. 
\item\label{it:nuprops:c} For any $A\in\cB(\R^d)$, we have $\E(\nu(A))=\P(S\in A)$.
\item\label{it:nuprops:d} $\nu_n\to\nu$ a.s.\ as $n\to\infty$.
\end{enumerate}
\end{theorem}
\begin{proof}
\ref{it:nuprops:a} is immediate from Proposition~\ref{Y} and the definition of $\nu$.

\ref{it:nuprops:b} If $F$ is a continuous map from a compact metric space $K$ to a metric space ${\bf X}$ then the push forward of a finite measure $M$ with $\supp(M)=K$ by $F$ (i.e., the finite measure $M(F^{-1}(\cdot))$ on ${\bf X}$) has compact support $F(K)$.  
This is an easy exercise.  To prove \ref{it:nuprops:b} use Proposition~\ref{Y}, the a.s.\ continuity of $i\mapsto S^i$ on $K$ (from Proposition~\ref{unifcontK}\,\ref{it:nuprops:c}), and apply the above with $M=Y$, $F=S$ and ${\bf X}=\R^d$.  

\ref{it:nuprops:c} Let $\phi\in C_b(\R^d)$ be non-negative.  
Then \eqref{tilderep} and \eqref{Ymass} imply 
$|\int\phi \, \dd \tilde \nu_n|\le \Vert\phi\Vert_\infty Y(I)= \Vert\phi\Vert_\infty W^0\in L^2$. 
So Lemma~\ref{tildenu} and Dominated Convergence give 
\begin{equation}\label{intconv}
\E\Bigl(\int\phi\,\dd \nu\Bigr)=\lim_{n\to\infty} \E\Bigl(\int\phi\,\dd \tilde  \nu_n\Bigr)=\lim_{n\to\infty} \E\Bigl(\sum_{f\in K^0_n}\phi(S^f)W^f\Bigr).
\end{equation}
Note that for $f\in F_n$, the random variables $S^f$ and $1(f\in K^0_n)$ are $\cG_n$-measurable, $W^f$ is $\sigma(Z^{f|i},i\ge n)$-measurable, and that these $\sigma$-fields are independent. Therefore
\eqref{intconv} gives
\begin{align*}
\E\Bigl(\int\phi\,\dd \nu\Bigr)&=\lim_{n\to\infty}\sum_{f\in F_n}\E(1(f\in K^0_n)\phi(S^f))\E(W^f)\\
&=\lim_{n\to\infty}\E\Bigl(\sum_{f\in K^0_n}\E\Bigl(\phi\Bigl(\sum_{k=1}^n \rho^k X^{f|(k-1)}_{f_k,Z^{f|(k-1)}}\Bigr)\Bigl|Z^{f|m},0\le m<n\Bigr)\Bigr)\mu^{-n}.
\end{align*}
The above conditional expectation equals $\E(\phi(S_n))$ because if $Z^{f|m}=z_m$ for $0\le m<n$, then $\{X^{f|(k-1)}_{f_k,z_{k-1}},k=1,\dots n\}$ are i.i.d.\ and equal in law to $X$.  Therefore
\[\E\Bigl(\int\phi\,\dd \nu\Bigr)=\lim_{n\to\infty}\E(|K^0_n|)\mu^{-n}E(\phi(S_n))=\E(\phi(S)).\]
The above equality extends to any bounded Borel $\phi$ and \ref{it:nuprops:c} is proved. 

\ref{it:nuprops:d} Let $\phi\in C_b(\R^d)$. Then
\begin{align}
\nonumber\E\Bigl(\Bigl( \int\phi\,\dd \tilde \nu_n-\int\phi\,\dd \nu_n\Bigr)^2\Bigr)&=\E\Bigl(\Bigl(\sum_{f\in K_n^0}\phi(S^f)(W^f-\mu^{-n})\Bigr)^2\Bigr)\\
\nonumber&=\E\Bigl({\sum\sum}_{f\neq g\in K^0_n} \phi(S^f)\phi(S^g)\E((W^f-\mu^{-n})(W^g-\mu^{-n}) \,| \,\cG_n)\Bigr)\\
\label{L2exp}&\quad+\E\Bigl(\sum_{f\in K^0_n}\phi(S^f)^2\,\E((W^f-\mu^{-n})^2 \, | \, \cG_n)\Bigr).
\end{align}
For $f\neq g$ in $F_n$, $W^f-\mu^{-n}$ and $W^g-\mu^{-n}$ are independent mean zero r.v.'s which are jointly independent of $\cG_n$. 
Therefore the first term on the right-hand side of \eqref{L2exp} is zero.  
The second term is bounded by $\Vert\phi\Vert_\infty^2\E(|K^0_n|)\mu^{-2n}\text{Var}(W^0)= \Vert\phi\Vert_\infty^2\text{Var}(W^0)\mu^{-n}$ which is summable in $n$.  
Therefore the left-hand side of \eqref{L2exp} is also summable. 
For each $M\in \N$, apply  Chebychev and then Borel-Cantelli to the events $E_n=\Bigl\{\Bigl| \int\phi\,\dd \tilde \nu_n-\int\phi\,\dd \nu_n\Bigr|>M^{-1}\Bigr\}$, and so conclude that $\lim_{n\to\infty}\int\phi\,\dd \tilde \nu_n-\int\phi\,\dd \nu_n=0$ a.s. Lemma~\ref{tildenu} now implies that  $\lim_{n\to\infty}\int\phi\,\dd \tilde \nu_n=\int\phi\,\dd \nu$. 
The required result
now follows by choosing a countable convergence determining set of bounded continuous functions $\phi$.
\end{proof} 

We next turn to questions on the Hausdorff dimension of $\supp(\nu)$.  We let $\dim A$ denote the Hausdorff dimension of a set $A\subset\R^d$. 
The proof of the upper bound on $\dim(\supp(\nu))$ proceeds just as in Theorem~5 of \cite{ADHP} for CSTRMs, using Lemma~\ref{CauchyC} to carry through the direct covering argument, and so we omit the proof of the following result.
\begin{proposition}\label{ubnd}
Assume \ref{TC} for all ${\zeta}>0$.  Then $\dim(\mathrm{supp}(\nu))\le {2\over \beta}\wedge d$ a.s.
\end{proposition}

The lower bound on the dimension in our dependent setting does require a bit more effort. In particular, the second moment formulae in our dependent setting seems to be a bit more involved because at a branch point the dependence of the sibling locations enters. We will use the product setting $(\Omega,\cF,\P)=(\Omega_b,\cF_b,\P_b)\times(\Omega_d,\cF_d,\P_d)$, mentioned earlier in this section.  We start with an elementary Fubini lemma. 

\begin{lemma}\label{Fubini}If $\phi:I\times I\times\Omega\to[0,\infty]$ is product measurable, then
\begin{align*}\E\Bigl(\int\int &\phi(i,j,\omega_b,\omega_d)\,\dd Y(i)\,\dd Y(j)\Bigr)\\
&=\E_b\Bigl[\int_I\int_I\Bigl(\int_{\Omega_d}\phi(i,j,\omega_b,\omega_d)\dd\P_d(\omega_d)\Bigr)\,\dd Y(i,\omega_b)\,\dd Y(j,\omega_b)\Bigr].
\end{align*}
\end{lemma}
\begin{proof} If $\phi(i,j,\omega_b,\omega_d)=\phi_1(i,j)\phi_2(\omega_b)\phi_3(\omega_d)$ for measurable non-negative $\phi_\ell$, the result holds by Fubini's theorem because $\int\int\phi_3(\omega_b)\phi_1(i,j)\,\dd Y(i)\,\dd Y(j)$ is a measurable function of $\omega_b$.  By a Monotone Class Theorem (e.g. Corollary~4.4 of the Appendix in \cite{EK}) the result follows for any bounded measurable $\phi$ as above.  By monotone convergence it follows for non-negative $\phi$ as in the Lemma.
\end{proof}

Recall $S$ is as in \eqref{SSndef}.  
In addition,
\begin{equation}\label{S'def}
S'\text{ denotes a r.v.\ independent of $S$ and with the same law as $S$}.
\end{equation}

\begin{lemma}\label{secondmom}
 Let $\phi:\R^d\times\R^d\to[0,\infty]$ be Borel, and set
\ARXIV{
\begin{equation}\label{Phindef}
\Phi_n=\sup_{x_1,x_2\in\R^d}\E(\phi(S_n+\rho^{n+1}(x_1+S), S_n+\rho^{n+1}(x_2+S'))),
\end{equation}
}
\SUBMIT{
	\begin{equation*}
		\Phi_n=\sup_{x_1,x_2\in\R^d}\E(\phi(S_n+\rho^{n+1}(x_1+S), S_n+\rho^{n+1}(x_2+S'))),
	\end{equation*}
}
where $\{S_n\}$, $S$ and $S'$ are mutually independent, $\{S_n\}$ is distributed as in \eqref{SSndef}, and $S,S'$ are as in \eqref{S'def}.
 Then
\[\E\Bigl(\int\int \phi(x,y)\,\dd \nu(x)\,\dd \nu(y)\Bigr)\le (1-\mu^{-1})\E((W^0)^2)\sum_{n=0}^\infty \mu^{-n}\Phi_n.\]
\end{lemma}
\begin{proof} We slightly abuse our notation and set $S^i(\omega)=\infty$ if $i\notin K(\omega_b)$ and also set $\phi(s_1,s_2)=0$ if either $s_1$ or $s_2$ is $\infty$. For $n\in\Z_+$, let 
\[\psi_n(i,j,\omega)=\phi(S^i(\omega),S^j(\omega))1(\kappa(i,j)=n).\]  
Joint measurability of $\psi_n$ follows easily from the continuity of $i\mapsto S^i$ on $K$ (Proposition~\ref{unifcontK}(c)).
It is easy to check (e.g. see (2.5) of \cite{ADHP}) that 
\[\kappa(i,j)<\infty \text{ for }Y\times Y-\text{a.a.}\  (i,j)\ \text{a.s.}\]
If $\overline D_K(f)=\overline D(f)\cap K$ we may use the above and Lemma~\ref{Fubini} to obtain
\begin{align}\label{doubleint}
\nonumber\E\Bigl(&\int\int \phi(x,y)\,\dd \nu(x)\,\dd \nu(y)\Bigr)\\
\nonumber&=\sum_{n=0}^\infty\E\Bigl(\int\int \psi_n(i,j)\,\dd Y(i)\,\dd Y(j)\Bigr)\\
\nonumber&=\sum_{n=0}^\infty\E_b\Bigl(\int\int\Bigl[\int \psi_n(i,j,\omega_b,\omega_d)\,\dd \P_d(\omega_d)\Bigr]\,\dd Y(i)\,\dd Y(j)\Bigr)\\
\nonumber&=\sum_{n=0}^\infty\E_b\Bigl(\sum_{f\in K^0_n}\int_{\overline D_K(f)}\int_{\overline D_K(f)} \Bigl[\int\phi(S^i(\omega_b,\omega_d),S^j(\omega_b,\omega_d)\,\dd \P_d(\omega_d)\Bigr]\\
&\phantom{=\sum_{n=0}^\infty\E_b\Bigl(\sum_{f\in K^0_n}\int_{\overline D_K(f)}\int_{\overline D_K(f)}}\times 1(i_{n+1}\neq j_{n+1})\dd Y(i)\dd Y(j)\Bigr).
\end{align}
In the last line we sum over the common value $f=i|n=j|n$ on $\{\kappa(i,j)=n\}$. Now fix $\omega_b$ outside a $\P_b$-null set so that the conclusion of Proposition~\ref{unifcontK} holds for $(\omega_b,\omega_d)$ for $\P_d$-a.a.~$\omega_d$, let $f\in K^0_n(\omega_b)$, and let $i,j\in\overline D_K(f)(\omega_b)$ with $i_{n+1}\neq j_{n+1}$.  Define
\[T^i_{n+1}(\omega_b,\omega_d)=\sum_{m=n+2}^\infty\rho^{m-n-1}X^{i|(m-1)}_{i_m,Z^{i|(m-1)}(\omega_b)}(\omega_d),\]
and similarly define $T^j_{n+1}$. Then
\begin{align}\label{innerint}
\nonumber\int&\phi(S^i(\omega_b,\omega_d),S^j(\omega_b,\omega_d))\,\dd \P_d(\omega_d)\\
&=\int \phi(S^f(\omega_b,\omega_d)+\rho^{n+1}[X^f_{i_{n+1},Z^f(\omega_b)}(\omega_d)+T^i_{n+1}(\omega_b,\omega_d)],\\
\nonumber&\phantom{=\int \phi(S }S^f(\omega_b,\omega_d)+\rho^{n+1}[X^f_{j_{n+1},Z^f(\omega_b)}(\omega_d)+T^j_{n+1}(\omega_b,\omega_d)])\,\dd \P_d(\omega_d).
\end{align}
For our fixed $\omega_b$, the random variables $X^{i|(m-1)}_{i_m,Z^{i|(m-1)}(\omega_b)}$ for $m=n+1,n+2,\dots,$ are i.i.d.\ distributed like our generic $X$ under $\P_d$ and so $T^i_{n+1}(\omega_b,\cdot)$ and $T^j_{n+1}(\omega_b,\cdot)$ are both distributed like $S$ under $\P_d$.
Similarly under $\P_d$, $S^f(\omega_b,\cdot)$ is distributed as $S_n$. The fact that $i_{n+1}\neq j_{n+1}$ implies that $S^f(\omega_b,\cdot)$, $(X^f_{i_{n+1},Z^f(\omega_b)}(\cdot),X^f_{j_{n+1},Z^f(\omega_b)}(\cdot))$, $T^{i}_{n+1}(\omega_b,\cdot)$, and $T^{j}_{n+1}(\omega_b,\cdot)$ are independent random vectors under $\P_d$ (the coordinates of the second vector need not be independent) because they depend on disjoint collections of $X^g$'s ($g\in F$) which are independent under $\P_d$. So in calculating the integral on the right-hand side of \eqref{innerint}, we can integrate out $(X^f_{i_{n+1},Z^f(\omega_b)}(\cdot),X^f_{j_{n+1},Z^f(\omega_b)}(\cdot))$ last, and conclude that
\begin{equation*}
\int\phi(S^i(\omega_b,\omega_d),S^j(\omega_b,\omega_d))\,\dd \P_d(\omega_d)\le \Phi_n.
\end{equation*}
Insert the above into the right-hand side of \eqref{doubleint} and reassemble our decomposition of the event $\{\kappa(i,j)=n\}$ to see that the far left side of \eqref{doubleint} is at most
\begin{align*}
\sum_{n=0}^\infty& \Phi_n\E_b\Bigl(\int\int1(\kappa(i,j)=n)\,\dd Y(i)\,\dd Y(j)\Bigr)\\
&=\sum_{n=0}^\infty \Phi_n (\mu^{-n}-\mu^{-n-1})\E((W^0)^2)\qquad\text{(by (2.5) of \cite{ADHP}})\\
&=(1-\mu^{-1})\E((W^0)^2)\sum_{n=0}^\infty \mu^{-n}\Phi_n.
\end{align*}
The result is proved. 
\end{proof}

We need the following density condition which will hold in all the cases of interest in this work. 
Recall $S'$ from \eqref{S'def}.
\begin{equation*}\tag{DC}\label{DC}
	S-S' \text{ has a bounded density.}
\end{equation*}
\begin{notation} If $m$ is a measure on $\R^d$ and $\alpha>0$ let 
\[\cE_\alpha(m)=\int\int|x-y|^{-\alpha}\,\dd m(x)\,\dd m(y)\in[0,\infty].\]
\end{notation}
\begin{proposition}\label{lbnd}
Assume \eqref{DC}.  Then $\dim(\supp(\nu))\ge \frac{2}{\beta}\wedge d$ a.s.\ on $\{\nu\neq 0\}$. 
\end{proposition}
\begin{proof} By the energy method (see, e.g. Theorem~4.27 and Remark~4.28 in \cite{MP}) it suffices to let $0<\alpha<\frac{2}{\beta}\wedge d$ and show
\begin{equation}\label{finenergy}
\E\Bigl(\cE_\alpha(\nu)\Bigr)<\infty.
\end{equation}
Let $h$ denote the bounded density of $S-S'$. We will use Lemma~\ref{secondmom} with $\phi(x,y)=|x-y|^{-\alpha}$. 
Then 
\begin{align*}
\Phi_n=\rho^{-\alpha(n+1)}\sup_{x_1,x_2} \E(|x_1-x_2+S-S'|^{-\alpha})&=\rho^{-\alpha(n+1)}\sup_{x_1,x_2}\int|z|^{-\alpha}h(z+x_2-x_1)\,\dd z\\
&\le \rho^{-\alpha(n+1)}\Bigl[\Vert h\Vert_\infty c_d\int_0^1r^{-\alpha+d-1}\,\dd r +1\Bigr]\\
&\le C_{d,\alpha,\Vert h\Vert_\infty} \rho^{-\alpha(n+1)},
\end{align*}
where $\alpha<d$ is used in the last line. Use this in Lemma~\ref{secondmom} to conclude that
\[ \E(\cE_\alpha(\nu))\le C\sum_{n=0}(\mu\rho^{\alpha})^{-n}=C\sum_{n=0}(\mu^{1-(\alpha\beta/2)})^{-n}<\infty.\]
The fact that $\alpha<\frac{2}{\beta}$ is used in the last inequality. This gives \eqref{finenergy} and we are done. 
\end{proof}
The above result and Proposition~\ref{ubnd} give the following extension of Theorem~2(d) in \cite{ADHP} which was established there for CSTRMs. (\cite{ADHP} also assumed that $p_0=0$ which meant there was no need to restrict to $\{\nu>0\}$.) 
\begin{theorem}\label{Hdim}
Assume \ref{TC} for all ${\zeta}>0$ and \eqref{DC}.  Then $\dim (\supp(\nu))=\frac{2}{\beta}\wedge d$ a.s.\ on $\{\nu\neq 0\}$.
\end{theorem}

We close this section with a simple Markovian characterization of $(\nu_n, \ n\geq 0)$ which will be useful in the next section when showing how SBM fits into this framework. 

For a metric space $E$ let $N_F(E)$ be the subspace of $M_F(E)$ consisting of $\Z_+$-valued finite measures.
Recall (e.g. Ch. 3 of \cite{DI78}) the probability generating function (p.g.f.) of a random measure in $N_F(\R^d)$, $\bar\nu$, is 
\begin{equation*}
	G\xi=\E(\exp(\bar\nu(\log\xi))) =\E\left(\exp \left( \int_{\R^d} (\log \xi)  \dd \bar\nu\right)\right), 
\end{equation*}
where $\xi:E\to(0,1]$ is Borel. 
A simple monotone class theorem (e.g. Corollary~4.4 in Appendix~4 of \cite{EK}) shows that the p.g.f.\ even restricted to continuous $\xi$ uniquely determines the law of $\bar \nu$.

We let
\[\bar\nu_n=\mu^n\nu_n=\sum_{f\in K^0_n}\delta_{S^f}\in N_F(\R^d),\]
and note that the p.g.f. of $\bar\nu_n$ is (the empty product is $1$)
\begin{equation}\label{nunpgf}
G_n\xi=\E\Bigl(\prod_{f\in K^0_n}\xi(S^f)\Bigr).
\end{equation}
 Recall that by hypothesis, for every $k\in \N$ the marginals of $Q_k$ are all equal to $\cL(X)$.
Let $\xi:\R^d\to(0,1]$ be a Borel map and let $(Z,(X_{m,k},m\le k, k,m\in\N))$ have law $\cL(Z)\times Q^s$.  For $n\in\N$, define the function $H_n\xi:\R^d\to(0,1]$ by
\begin{equation}\label{Hndef}H_n\xi(x)=\E\Bigl(\prod_{m=1}^Z\xi(x+\rho^{n}X_{m,Z})\Bigr)=\sum_{k=0}^\infty \P(Z=k)\int_{(\R^d)^k}\prod_{m=1}^k\xi(x+\rho^nx_m)\,\dd Q_k(x_1,\dots,x_k).
\end{equation}
Finally for $\bar\nu\in N_F(\R^d)$ we introduce
\[G_{n,n+1}\xi(\bar\nu)=\exp(\bar\nu(\log(H_{n+1}\xi))),\quad n\in\Z_+.\]

\begin{notation} If $f\in F_n$ for some $n\in\N$ let $\pi f=(f_1,\dots, f_{n-1})\in F_{n-1}$ (if $n=1$ this is $0$) and call $\pi f$ the parent of $f$. 
\end{notation} 

\begin{proposition}\label{nunmarkov}
The conditional p.g.f.'s of $(\bar\nu_n)$ are given by 
\begin{align}\label{condpgf}\nonumber \E(\exp\{\bar\nu_{n+1}(\log\xi)\} \, | \, \cG_n)=G_{n,n+1}\xi(\bar\nu_n)&=\prod_{f\in K^0_n}H_{n+1}\xi(S^f)\\
&\qquad \text{for all } n\in\Z_+ \text{ and Borel }\xi:\R^d\to(0,1].
\end{align}
The $N_F(\R^d)$-valued process $\left(\bar \nu_n,n\in\Z_+\right)$ is the unique in law time-inhomogeneous $(\cG_n)$-Markov chain starting at $\bar\nu_0=\delta_0$, and satisfying \eqref{condpgf}.
\end{proposition}
\begin{proof} Let $\xi$ be as in \eqref{condpgf} and for $k,n\in\Z_+$ define
\[H_{k,n+1}\xi(x)=\E\left(\prod_{m=1}^k\xi(x+\rho^{n+1}X_{m,k})\right),\]
so that 
\begin{equation}\label{Hnform}H_{n+1}\xi(x)=\E(H_{Z,n+1}\xi(x)).
\end{equation}
Recall that for $f\in F_{n+1}$,
\begin{equation*}
f\in K^0_{n+1} \quad \text{iff} \quad \pi f\in K^0_n\text{ and }f_{n+1}\le Z^{\pi f},
\end{equation*}
and
\begin{equation*}
S^f=S^{\pi f}+\rho^{n+1}X^{\pi f}_{f_{n+1},Z^{\pi f}}.
\end{equation*}
Define $\bar \cG_n=\cG_n\vee\sigma(Z^f:f\in F_n)$. Then by \eqref{nunpgf} and the last two displays,
\begin{align}\label{condpdfexp}
\nonumber \E(\exp\{\bar\nu_{n+1}(\log\xi)\} \, | \, \cG_n)&=\E\Bigl(\E\Bigl(\prod_{f\in K^0_n}\prod_{f_{n+1}=1}^{Z^f}\xi(S^f+\rho^{n+1}X^f_{f_{n+1},Z^f})\, \Bigl| \, \bar \cG_n\Bigl) \, \Bigl| \, \cG_n\Bigl)\\
\nonumber&=\E\Bigl(\prod_{f\in K^0_n}\E\Bigl(\prod_{f_{n+1}=1}^{Z^f}\xi(S^f+\rho^{n+1}X^f_{f_{n+1},Z^f})\, \Bigl| \, \bar \cG_n\Bigr) \, \Bigr| \, \cG_n\Bigr)\\
&=\E\Bigl(\prod_{f\in K^0_n} H_{Z^f,n+1}(S^f)\, \Bigr| \, \cG_n\Bigr).
\end{align}
For the last two equalities we use the independence of $\{X^f:f\in F_n\}$, the independence of this collection from $\bar\cG_n$, and the fact that $K^0_n$ and $\{Z^f,S^f: f\in F_n\}$ are $\bar\cG_n$-measurable.  In fact $S^f$ and $K^0_n$ are $\cG_n$-measurable, while each $Z^f$ is independent of $\cG_n$. Therefore by \eqref{Hnform}  the right-hand side of \eqref{condpdfexp} equals
\[\prod_{f\in K^0_n}H_{n+1}\xi(S^f)=\exp\{\bar\nu_n(\log H_{n+1}\xi)\}.\]
Inserting this into \eqref{condpdfexp} gives \eqref{condpgf}. By the aforementioned monotone class theorem this  establishes the $\cG_n$-Markov property of $\bar \nu_n$ and uniquely identifies its time-inhomogeneous transition kernel. The result follows.
\end{proof}

\begin{remark}\label{rem:Qklaw}
Recall that $\nu=\lim_{n\to\infty}\mu^{-n}\bar\nu_n$ a.s.\ (by Proposition~\ref{nuprops}\ref{it:nuprops:d}).  
The above results shows that the joint law of $\left(\nu,(\bar\nu_n)_{n\in\N}\right)$ depends only on the parameter $\beta$, the reproduction law $\cL(Z)$ and the family of symmetric laws $\left(Q_k,k\in\N\right)$, all with common marginals $\cL(X)$, as we claimed earlier after defining a STRM.
\end{remark}

\section{Super-Brownian motion as a super-tree random measure}\label{sec:SBM}

Our goal is to show (Theorem~\ref{thm:SBMSTRM}) that if $X_1$ is $d$-dimensional super-Brownian motion with branching rate $\gamma>0$, then $\frac{2}{\gamma}X_1$ is a STRM $\nu$ with $\beta=1$ and the other parameters described in Theorem~\ref{thm:SBMSTRM} below.
Here, $X_1$ is considered under its canonical measure starting at $\delta_0$, conditioned to be non-zero, that is, a super-Brownian cluster starting with a single individual at the origin. Details are  recalled below. 

\begin{notation} $C=C([0,1],\R^d)$ is the space of continuous $\R^d$-valued paths on the unit interval with the topology of uniform convergence. For $t\in[0,1]$,   $\pi_t(y)=y(t)$ for $y\in C$ are the projection maps. 
We write $y^t$ for the stopped path $s\mapsto y(s\wedge t)$ in $C$ and $C^t=\{y^t,y\in C\}$ for the set of paths stopped at $t$. 
The space $C^0$ will be identified with $\R^d$. 
For $t\in[0,1]$, $\cC_t=\sigma(y(s),s\le t)$ defines the canonical filtration on $C$.  If $s\in [0,1]$ and $y,w\in C$ satisfy $w(0)=y(s)$, define $y/s/w\in C$ by 
\ARXIV{
\begin{equation}\label{eq:definition concatenation paths}
(y/s/w)(t)=\begin{cases} y(t),\ &\text{ if }t\in[0,s),\\
w(t-s), \ &\text{ if }t\in[s,1].
\end{cases}
\end{equation}}
\SUBMIT{
\begin{equation*}
		(y/s/w)(t)=\begin{cases} y(t),\ &\text{ if }t\in[0,s),\\
			w(t-s), \ &\text{ if }t\in[s,1].
		\end{cases}
\end{equation*}}
If $E$ is a metric space and $I\subset [0,\infty)$ is a left-closed interval, $D(I,E)$ is the space of c\`adl\`ag paths with the Skorohod topology (see Ch. 3 of \cite{EK} and the references there). 
\end{notation}

If $(B_t, \, t\ge 0)$ is a standard Brownian motion in $\R^d$ starting at $x\in \R^d$ under $P_x$, we will consider the time-inhomogeneous $C$-valued Markov process $t\mapsto B^t$, taking values in $C^t$ at time $t$ for $t\in[0,1]$. 
If $\phi:C\to\R$ is bounded Borel, the time inhomogeneous semigroup associated with the above process is (see Ch. 2 of \cite{DP91}) 
\ARXIV{
\begin{equation}
	\label{eq:def semi group HBM}
T_{s,t}\phi(y)=E_{y(s)}(\phi(y/s/B^{t-s}))\quad \text{ for }y\in C\text{ and }0\le s\le t\le 1.
\end{equation}}
\SUBMIT{
\begin{equation*}
	T_{s,t}\phi(y)=E_{y(s)}(\phi(y/s/B^{t-s}))\quad \text{ for }y\in C\text{ and }0\le s\le t\le 1.
\end{equation*}
}
Therefore $T_{s,t}:\mathcal{B}_b(C)\to \mathcal{B}_b(C^s)$, i.e.\ $T_{s,t}$ maps any bounded Borel functional over $C$ to a Borel functional over $C^s$.

Historical Brownian motion $(H_t,t\in[0,1])$ (the restriction to $[0,1]$ is just to suit our purposes) is an enrichment of super-Brownian motion in which the past history of a particle alive at time $t$ is also recorded.  It is simply the superprocess whose underlying spatial motion is $t\mapsto B^t$, rather than $t\mapsto B(t)$, as is the case for SBM.  $H$ is a time-inhomogeneous $M_F(C)$-valued Markov processes such that $H^t$ is supported by $C^t$ for all $t$.  
If $m\in M_F(\R^d)$ we will assume that under the probability $\Q_m$, the process $(H_t,t\in[0,1])$ is a historical Brownian motion starting at $H_0=m$ with branching rate $\gamma>0$.  
Again $\gamma>0$ arises as the branching variance in a branching particle approximation to $H$ (see Section~7 of \cite{DP91}).
The $M_F(\R^d)$-valued process $X_t(\cdot)=H_t(\pi_t^{-1}(\cdot))$ is a SBM on $[0,1]$, starting at $m$ with branching rate $\gamma$ (see (II.8.4) in \cite{Per02}). See \cite{DP91} or Section~II.8 of \cite{Per02} for more information on historical Brownian motion and historical processes in general. 

We describe a construction of $H_1$ using a branching particle system taken from Chapter~3 of \cite{DP91}. Assume $H_0=\delta_0$ for now.  The canonical measure associated with $H_1$ (an infinitely divisible random measure) will be denoted by $R_1$ (see Sections II.7 and II.8 of \cite{Per02}).  Recall from (II.8.6) of \cite{Per02} that $R_1$ is a finite measure on $M_F(C)\setminus\{0\}$ with total mass $2/\gamma$ such that $H_1=\int\eta \, \Xi(d\eta)$, where $\Xi$ is a Poisson point process on $M_F(C)\setminus\{0\}$ with intensity $R_1$. If $\P^*=R_1/(2/\gamma)$ is the ``cluster law", this means that 
\begin{equation}\label{Poisdec}
	\parbox[t]{0.9\textwidth}{$\displaystyle H_1=\sum_{i=1}^NH^i_1$ where $\{H^i_1:i \in\N\}$ are i.i.d.\ random measures with law $\P^*$ and $N$ is an independent Poisson r.v.\ with mean $2/\gamma.$}
\end{equation}
This decomposes $H_1$ into a Poisson number of i.i.d. clusters corresponding to the distinct ancestors at time $0$ with descendants alive at $t=1$.  In many respects $\P^*$ is a more fundamental law than $\Q_{\delta_0}(H_1\in\cdot)$, but it is usually easy to transfer properties from one to the other.  We often will abuse notation slightly and use $H_1$ to denote a historical cluster when working under $\P^*$, while working with clusters denoted by $H^i_1$ under $\Q_{\delta_0}$. As noted above, under $\Q_{\delta_0}$, $X_1(\cdot)=H_1(\pi_1^{-1}(\cdot))$ is SBM starting at $\delta_0$ with branching rate $\gamma$, and by (II.8.7) of \cite{Per02}, 
\begin{equation}\label{SBMcluster}
\text{$X_1$ has cluster law $\P^*(H_1\circ\pi_1^{-1}\in\cdot)$.}
\end{equation}
 (Recall from the Introduction that the cluster law of $X_1$ is its canonical measure conditioned on $X_1\neq 0$.)

Let $\lambda:[0,1)\to [0,\infty)$ be continuous.  Consider a particle system starting at positions $x_1,\dots, x_{N(0)}$ where the $x_i$'s need not be distinct and $N(0)\in \N$. Each particle has an associated independent   inhomogeneous Poisson process with intensity $\lambda(s)\,ds$, and follows an independent standard $d$-dimensional Brownian motion until the first jump of its Poisson process.    At that time the ``parent" particle is replaced by two particles at the same location as the parent. After the first binary split, the $N(0)+1$ particles follow independent Brownian motions with their independent Poisson processes, and the alternating pattern of Brownian migration and binary branching continues for $t\in[0,1)$.  If $t<1$, the number of particles at time $t$, $N(t)$, can be dominated by a pure birth process at time $t$ with constant birth rate $\sup_{s\le t}\lambda(s)$, and so is finite on $[0,1)$. If $(y_1(s), s\le t),\dots,(y_{N(t)}(s), s\le t)$ are the past Brownian trajectories of the $N(t)$ particles alive at time $t\in[0,1)$, extend each $y_i$ to $[0,1]$ by making it constant on $[t,1]$ and define
\begin{equation}\label{Hrep1}H^*_t=\sum_{i=1}^{N(t)}\delta_{y_i}\in N_F(C^t).
\end{equation}
Then $(H^*_t,0\le t<1)$ is an inhomogeneous $N_F(C)$-valued Markov process with sample paths in $D([0,1),N_F(C))$ and law $Q^*_{0,m}$ on this space. We call $H^*$ a historical branching Brownian motion with rate function $\lambda$. More information can be found on pp. 44-46 in \cite{DP91} where a more general class of branching particle systems is studied.   As discussed there, we may of course start $H^*$ at time $s\in[0,1)$ in state $m\in\N_F(C^s)$ and let $Q^*_{s,m}$ denote the law of $H^*$ on $D([s,1),N_F(C))$.  

The probability generating function of $H^*_t$ is
\begin{equation}\label{eq:def overline G st}
	\overline{G}_{s,t}\theta(m)=Q^*_{s,m}(\exp(H^*_t(\log\theta)))\text{ for Borel }\theta:C\to(0,1],\ m\in N_F(C^s),\  0\le s\le t< 1.
\end{equation}
Clearly $\overline{G}_{s,t}\theta\in(0,1]$,  and for each $s,t$ as above, $\overline{G}_{s,t}$ uniquely determines the transition probabilities $Q^*_{s,m}(H^*_t\in \cdot)$ for $m\in N_F(C^s)$, e.g.\ by Lemma~II.5.9 of \cite{Per02}. We write $\overline{G}_{s,t}\theta(y)$ for $\overline{G}_{s,t}\theta(\delta_y)$ for $y\in C^s$.  The independent evolution of the particle system from each path in the support of the initial measure implies that if $m=\sum_{i=1}^n\delta_{y_i}$ for some $y_1,\dots,y_n\in C^s$, then
\begin{equation}\label{multprop}\overline{G}_{s,t}\theta(m)=\prod_{i=1}^n\overline{G}_{s,t}\theta(y_i)=\exp(m(\log \overline{G}_{s,t}\theta)).
\end{equation}
This is often called the multiplicative property of the historical BBM (e.g.\ see Section~3 of \cite{DI78}), and shows that the mappings $(s,y)\mapsto \overline{G}_{s,t}\theta(y)$ on $W_t=\{(s,y):y\in C^s,s\in[0,t]\}$, for $t\in[0,1)$ and $\theta$ as above, uniquely determine the laws $Q^*_{s,m}$ of $H^*$.
Theorem~3.6(a) of \cite{DP91} implies that  $(s,y)\mapsto \overline{G}_{s,t}\theta(y)$ on $W_t$ (so $s\in [0,t]$) is the unique bounded solution of 
\begin{align}\label{brnle}
\nonumber\overline{G}_{s,t}\theta(y)&=\exp\Bigl(-\int_s^t\lambda(v)\,dv\Bigr)T_{s,t}\theta(y)+\int_s^tT_{s,u}((\overline{G}_{u,t}\theta)^2)(y)\exp\Bigl(-\int_s^u\lambda(v)\,dv\Bigr)\lambda(u)\,du,\\
\overline{G}_{t,t}\theta(y)&=\theta(y).
\end{align}
Intuitively, the first term on the right side of \eqref{brnle} is the contribution from the event where there is no branching in $[s,t]$, and the second term calculates the contribution where there is a first branch time $u\in[s,t]$ by conditioning on $u$ and integrating it out. 
The usual  proof of uniqueness using Gronwall's lemma also gives local uniqueness. 
By this we mean that if $t_0\in[0,t)$ and $\tilde G_{s,t}(y)$ (bounded by $M$, say) satisfies \eqref{brnle} for $(s,y)\in W_t$ and $s\ge t_0$, then $\tilde G_{s,t}(y)=\overline{G}_{s,t}\theta(y)$ for $(s,y)$ as above.  To see this, let $h^s(y)=| \tilde G_{s,t}(y)-\overline{G}_{s,t}\theta(y)|$ for $(s,y)\in W_t$ with $s\ge t_0$, $K(du)=
\lambda(u)\,du$ ($0\le u\le t$), recall $\bar G_{s,t}\theta\le 1$, and note that (for $s\in [t_0,t]$),
\begin{align*}h^s(y)\le\int_s^t |T_{s,u}((\tilde G_{u,t})^2-(\overline{G}_{u,t}\theta)^2)(y)|K(du)
&\le \int_s^t (1+M)T_{s,u}(h^u)(y)K(du)\\
&\le (1+M)E_{s,y}\Bigl[\int_s^t h^u(B^u)K(du)\Bigr],
\end{align*}
where the expectation $E_{s,y}$ is under the measure $P_{s,y}$ under which the path-valued process $(B^u)_{u\geq s}$ is started from $y\in C^s$ at time $s$. 
A generalized Gronwall lemma of Dynkin (see Lemma~3.2 of \cite{Dyn91}) implies $h=0$, as required.

Let $(P_t,t\ge 0)$ be the standard $d$-dimensional Brownian semigroup.
The ordinary branching Brownian motion (BBM) $(X^*_t, t<1)$ with branching rate $\lambda$ associated with $H^*$ is the $N_F(\R^d)$-valued time-inhomogeneous Markov process 
given by
\begin{equation}\label{XHdefn}X^*_t(\cdot)=H^*_t(\pi_t^{-1}(\cdot))=\sum_{i=1}^{N(t)}\delta_{y_i(t)}(\cdot),\text{ where $H^*_t$ is as in \eqref{Hrep1}}.\end{equation}
Formally, and independently of the above representation in terms of $H^*$, $X^*$ has laws $P^*_{s,\eta}$ for $s\in[0,1)$ and $\eta\in\N_F(\R^d)$, which by the multiplicative property are uniquely determined by its p.g.f.,
\begin{equation}\label{eq:def G st xi}
G_{s,t}\xi(x)=P^*_{s,\delta_x}(\exp(X^*_t(\log\xi))) \text{ for Borel $\xi:\R^d\to(0,1]$ and }0\le s\le t<1.
\end{equation}
The map $(s,x)\mapsto G_{s,t}\xi(x)$ is in turn the unique (and locally unique) solution of (c.f. \eqref{brnle})
\begin{align}\label{genpdfeq}
\nonumber G_{s,t}\xi(x)&=\exp\Bigl(-\int_s^t\lambda(v)\,dv\Bigr)P_{t-s}\xi(x)+\int_s^t P_{u-s}((G_{u,t}\xi)^2)(x)\exp\Bigl(-\int_s^u\lambda(v)\,dv\Bigr)\lambda(u)\,du,\\
G_{t,t}\xi(x)&=\xi(x)
\end{align}
(see, e.g., Section~4 of \cite{DI78} for the constant $\lambda$ case).
If we set $\theta(y)=\xi(y(t))$ for $y\in C$, and $\xi$ as above, then $T_{s,t}\theta(y)=P_{t-s}\xi(y(s))$ for $y\in C^s$, and $s\le t$.  It is easy to check that $\widehat G_{s,t}(x)$ satisfies \eqref{genpdfeq} iff $\widehat{\overline G}_{s,t}(y):=\widehat G_{s,t}(y(s))$ satisfies \eqref{brnle}. Therefore 
\begin{equation}\label{HXpgfs}\overline G_{s,t}\theta(y)=G_{s,t}\xi(y(s)),
\end{equation}
which also gives independent confirmation that if $H^*$ starts at $(s,m)$ for $m\in N_F(C^s)$, then \eqref{XHdefn}  defines BBM starting at $(s,m\circ\pi_s^{-1})$, as should be clear.  (It also gives independent verification of existence, uniqueness, and local uniqueness in \eqref{genpdfeq} from that of \eqref{brnle}.)

Henceforth we will set
\[\lambda(s)=\frac{1}{1-s}\text{ for }s\in[0,1),\]
and write $Q^*$ for $Q^*_{0,\delta_0}$. In this case \eqref{brnle} becomes
\begin{equation*}
\overline{G}_{s,t}\theta(y)=\frac{1-t}{1-s}T_{s,t}\theta(y)+\int_s^tT_{s,u}((\overline{G}_{u,t}\theta)^2)(y)(1-s)^{-1}du,\ (s,y)\in W_t,
\end{equation*}
and \eqref{genpdfeq} becomes
\begin{align}\label{BBMnle}
\nonumber  &G_{s,t}\xi(x)=\frac{1-t}{1-s}P_{t-s}f(x)+\int_s^tP_{u-s}((G_{u,t}\xi)^2)(x)(1-s)^{-1}\,du\text{ for }(s,x)\in[0,t]\times\R^d,\\ &G_{t,t}\xi(x)=\xi(x)\text{ for all }x\in\R^d.
\end{align}

\begin{notation}  If $s\in [0,1]$, define the restriction map $r_s:M_F(C)\to M_F(C)$ by 
\[r_sm(A)=m(\{y:y^s\in A\}).\]
\end{notation}
\begin{proposition}\label{bpshbm}
Consider $H_1$ under its cluster law $\P^*$ and with branching rate $\gamma>0$. One can define a historical branching Brownian motion $H^*$, with law $Q^*$, as a functional of $H_1$ so that for any $\veps_n\downarrow 0, \veps_n\in(0,1)$, 
\begin{equation}\label{bpconv}
\frac{\veps_n\gamma}{2}H^*_{1-\veps_n}\to H_1 \quad  \P^*-\text{a.s.}
\end{equation}
Let $X_1(\cdot)=H_1(\pi_1^{-1}(\cdot))$ denote the associated SBM under its cluster law (by \eqref{SBMcluster}) and let $X^*_t(\cdot)=H^*_t(\pi_t^{-1}(\cdot)),\,t<1$ be the BBM associated with the historical BBM $H^*$. Then
\begin{equation*}
\frac{\veps_n\gamma}{2}X^*_{1-\veps_n}\to X_1 \quad \P^*-\text{a.s.}
\end{equation*}
\end{proposition}
\begin{proof} This follows readily from Theorem~3.9(b) and Theorem~3.10 of \cite{DP91} in the setting there with $g(\veps)=g_b(\veps)= (\veps \gamma/2)$ and $t=1$. So we are setting the parameter $\beta=1$ in that reference and warn the reader that the $\gamma$ in that reference corresponds to our $\gamma/2$.   
The first theorem cited above constructs $H^*_s$ from $H_1$, and the second (Theorem~3.10) establishes the a.s.\ convergence to $H_1$.  
Theorem~3.10 of the above reference works under the law of historical Brownian motion, $\Q_{\delta_0}$, not $\P^*$.  
It is not hard
to use the cluster decomposition \eqref{Poisdec} and consider the convergence under $\Q_{\delta_0}$ on the set $\{N=1\}$ to obtain the result under $\P^*$.  
A second minor issue is that to make the approximating sequence in Theorem~3.10 correspond to the historical branching particle system constructed in Theorem~3.9 one must apply the restriction mapping $r_{1-\veps_n}$ to the approximating measures in the a.s.\ convergence in Theorem~3.10 and show that one still has a.s.\ convergence. 
As $r_{1-\veps_n}$ approaches the identity this turns out to be trivial and we obtain \eqref{bpconv}.  
The final assertion is then immediate because the map $m\mapsto m\circ\pi_1^{-1}$ from $M_F(C)$ to $M_F(\R^d)$ is continuous.
\ARXIV{A more detailed proof can be found in Section~\ref{app:subsec:bpshbm} of the Appendix.}
\end{proof}

\begin{remark}\label{H*defn}
To see what $H^*_{1-\veps_n}$ actually is, recall that  $r_{1-\veps_n}H_1$ has a finite number of atoms in $C^{1-\veps_n}$ (see Theorem~III.1.1 of \cite{Per02} and note the same reasoning applies under the canonical measure) corresponding to the histories of particles at $t=1-{\veps_n}$ which have descendants alive at time $1$. 
$H^*_{1-\veps_n}$ records the evolution of these particles: it has atoms at the same positions as $r_{1-\veps_n}H_1$ does and they all have mass $1$.
\end{remark}

\begin{notation}  $N_t=H^*_t(C)=X^*_t(\R^d)$ is the number of points in the branching Brownian motion at time $t$.  We often write $N_{s,t}$ for $N_t$ if we are starting at time $s\in[0,t]$.  
\end{notation}

\begin{proposition}\label{geometricZ} If $0\le s<t <1$ and $y\in C^s$, $Q^*_{s,\delta_y}(N_{s,t}\in\cdot)$ is a geometric distribution with mean $\frac{1-s}{1-t}$. In particular, $Q^*_{1-\mu^{-n},\delta_y}(N_{1-\mu^{-n},1-\mu^{-(n+1)}}\in\cdot)$ is a geometric distribution with mean $\mu$, for any $\mu>1$ and $n\in\Z_+$. 
\end{proposition}
\begin{proof} This follows from Theorem~3.11(a) and Theorem~3.9(b) of \cite{DP91}. Alternatively it is easy to prove  directly. If $T_1\ge s$ is the first branch time after time $s$, then for $Q_s^*=Q^*_{s,\delta_y}$ (the probabilities are independent of the choice of $y$ by the Markov property of $X^*$ and translation invariance) and $t>s$,
\[Q_s^*(T_1\le t)=Q_s^*(N_{s,t}>1)=1-Q_s^*(N_{s,t}=1)=1-\exp\Bigl\{-\int_s^t\frac{1}{1-u}\,du\Bigr\}=\frac{t-s}{1-s}.\]
So in particular $T_1$ is uniform on $[s,1]$. By conditioning on this uniform time we get for $k\ge 2$,
\[Q_s^*(N_{s,t}=k)=\sum_{j=1}^{k-1}\int_s^t Q_u^{*}(N_{u,t}=j)Q_u^{*}(N_{u,t}=k-j)du(1-s)^{-1}.\]
The obvious induction on $k$ (prove the result for all  times $s<t<1$) gives the result. 
\end{proof}

\begin{proposition}\label{HBBMmean} 
If $t\in(0,1)$, then for $k\in\N$ and bounded Borel $\phi:C([0,1],\R^d)\to \R$,
\begin{equation}\label{condhmm}
Q^*(H^*_t(\phi) \, | \, H^*_t(1)=k)=kE_0(\phi(B^t)),
\end{equation}
where $B$ is a standard $d$-dimensional Brownian motion under $P_0$.  In particular, for any bounded Borel $\psi:\R^d\to\R$, 
\begin{equation}\label{condmm}
Q^*(X^*_t(\psi) \, | \, X^*_t(1)=k)=kE_0(\psi(B_t)).
\end{equation}
\end{proposition}
\begin{proof} (Sketch). This can be proved using the construction of $H^*$ through a tree-indexed sequence of i.i.d.\  Poisson processes with rate $ds/(1-s)$ (giving the branch times and ``tree shape" of $H^*_t$) and an independent tree-indexed i.i.d.\ collection of standard Brownian motions which fill in the spatial motions between the branch times (e.g., see \cite{JB} or \cite{W86}).  If we condition on all the Poisson processes, $H^*_t$ is then a sum of point masses at (correlated) Brownian paths stopped at time $t$.  The conditional mean of $H^*_t(\phi)$ is then just $E_0(\phi(B^t))H^*_t(1)$. Integrating this conditional expectation over the event $\{H^*_t(1)=k\}$ (which is in the sigma-field generated by the Poisson processes) gives \eqref{condhmm}, and \eqref{condmm} is then immediate.  
\ARXIV{A detailed proof can be found in Section~\ref{app:subsec:HBBMmean} of the Appendix.}
\end{proof}

\begin{notation} If $x,y\in\R^{d'}$ and $c>0$, let $S_c(y)=\sqrt c y$ and $\tau_x(y)=x+y$. Often $d'$ will be $d$.
\end{notation} 

Here is an elementary scaling result for $G_{s,t}$.

\begin{lemma}\label{lem:scaling}
If $0\le r\le t<1$, then for any Borel $\xi:\R^d\to(0,1]$, 
\begin{equation}\label{scaeq}
\forall (s,x)\in[r,t]\times\R^d, \qquad G_{\frac{s-r}{1-r},\frac{t-r}{1-r}}(\xi\circ S_{1-r})(x/\sqrt{1-r})=G_{s,t}\xi(x).
\end{equation}
\end{lemma}
\begin{proof} Let $\hat G_{s,t}(x)\in(0,1]$ denote the left-hand side of \eqref{scaeq}.  By the local uniqueness of solutions to \eqref{BBMnle} it suffices to check that $\hat G$ solves \eqref{BBMnle} for $(s,x)\in[r,t]\times\R^d$.  
First we have
\[\hat G_{t,t}(x)=\xi\circ S_{1-r}(x/\sqrt{1-r})=\xi(x).\]
The definition of $\hat G$ and \eqref{BBMnle} imply that for $s\in[r,t]$,
\begin{align}\label{scaleq1}
\hat G_{s,t}(x)&=\frac{1-\frac{t-r}{1-r}}{1-\frac{s-r}{1-r}}P_{\frac{t-s}{1-r}}(\xi\circ S_{1-r})\Bigl(\frac{x}{\sqrt{1-r}}\Bigr)\\
\nonumber&\quad+\int_{\frac{s-r}{1-r}}^{\frac{t-r}{1-r}}P_{v-\frac{s-r}{1-r}}\Bigl((G_{v,\frac{t-r}{1-r}}(\xi\circ S_{1-r}))^2\Bigr)\Bigl(\frac{x}{\sqrt{1-r}}\Bigr)\Bigl(1-\frac{s-r}{1-r}\Bigr)^{-1}dv.
\end{align}
Brownian scaling implies that for any $u\ge 0$,
\begin{equation}\label{Brsc}
P_{\frac{u}{1-r}}(\xi\circ S_{1-r})\Bigl(\frac{x}{\sqrt{1-r}}\Bigr)=P_u\xi(x).
\end{equation}
Use this to see the first term on the right-hand side of \eqref{scaleq1} is
\begin{equation}\label{term1}\frac{1-t}{1-s}P_{t-s}\xi(x).
\end{equation}
First do a change of variable ($v=(u-r)/(1-r)$) and then use \eqref{Brsc}
to see that the second term on the right-hand side of \eqref{scaleq1} is
\begin{align}\label{term2}
\nonumber \int_s^t&P_{\frac{u-s}{1-r}}\Bigl(G_{\frac{u-s}{1-r},\frac{t-r}{1-r}}(\xi\circ S_{1-r})^2\Bigr)\Bigl(\frac{x}{\sqrt{1-r}}\Bigr)((1-r)-(s-r))^{-1}du\\
\nonumber&=\int_s^t P_{\frac{u-s}{1-r}}((\hat G_{u,t}\circ S_{1-r})^2)\Bigl(\frac{x}{\sqrt{1-r}}\Bigr)(1-s)^{-1}du\\
&=\int_s^t P_{u-s}((\hat G_{u,t})^2)(x)(1-s)^{-1}du.
\end{align}
Insert \eqref{term1} and \eqref{term2} into the right-hand side of \eqref{scaleq1} and conclude
\[
\forall s\in[r,t], \forall x\in\R^d,
\qquad\hat G_{s,t}(x)=\frac{1-t}{1-s}P_{t-s}\xi(x)+\int_s^t P_{u-s}((\hat G_{u,t})^2)(x)(1-s)^{-1}du. 
\]
Therefore $\hat G_{s,t}$ satisfies \eqref{BBMnle} for $s\in[r,t]$, as required.
\end{proof}

We need a simple result about random point processes.  
\begin{proposition}\label{Qk}
Let $\bar\nu$ be a random measure in $N_F(\R^d)$ such that $\P(\bar\nu(1)=k)>0$ for all $k\in\N$.  For each $k\in\N$ there is a unique symmetric probability, $\tilde Q_k$, on $(\R^d)^k$ (symmetric in the $k$ $\R^d$-valued components) such that for all Borel $\xi:\R^d\to(0,1]$,
\begin{equation*}
\int\prod_{i=1}^k\xi(x_i)\dd \tilde Q_k(x_1,\dots,x_k)=\E(\exp\{\bar\nu(\log\xi)\} \, | \, \bar\nu(1)=k).
\end{equation*}
Moreover for all $i\le k$, 
\begin{equation}\label{Qkmar}
\int 1(x_i\in\cdot)\dd \tilde Q_k(x_1,\dots,x_k)=\frac{1}{k}\E(\bar\nu(\cdot) \, | \, \bar\nu(1)=k).
\end{equation}
\end{proposition}
\ARXIV{The proof can be found in Section~\ref{app:subsec:Qk} of the Appendix.}
\SUBMIT{The proof of this result is elementary and left to the reader.}

\begin{definition} We call $\{\tilde Q_k\}$ the conditional support measures associated with $\bar\nu$. 
\end{definition} 

Recall from \eqref{SBMcluster} that $\P^*(X_1\in\cdot)$ is the cluster law of a SBM, $X_1$, with branching rate $\gamma>0$. 
Recall also the historical BBM, $(H^*_t, \,t<1)$, and BBM, $(X^*_t(\cdot), \,t<1)$, constructed from $H_1$ in Proposition~\ref{bpshbm} under $\P^*$. See Figure~\ref{fig:SBM is STRM} for a pictorial representation of the following result.

\begin{theorem}\label{thm:SBMSTRM}
Let $\mu>1$ and work under $\P^*$. 
Let $\{\tilde Q_k\}$ be the conditional support measures of the BBM $X^*_{1-\frac{1}{\mu}}$ and define $Q_k$ on the Borel sets in $(\R^d)^k$ by $Q_k=\tilde Q_k\circ S_{\mu}^{-1}$.  
Then the following properties hold:
\begin{enumerate}[label=(\alph*)]
	\item The random measure $\frac{2}{\gamma}X_1$ is a STRM with $\beta=1$, offspring law, $\cL(Z)$, geometric with mean $\mu$, and symmetric displacement laws $\{Q_k\}$.
	\item  The common marginal distribution of the $Q_k$'s, $\cL(X)$, is the $d$-dimensional Gaussian law with mean $0$ and covariance $(\mu-1)I_d$. 
	\item  The approximating sequence $(\nu_n)$ corresponding to $\nu=\frac{2}{\gamma}X_1$ is given by $\nu_n=\mu^{-n}X^*_{1-\mu^{-n}}$.
\end{enumerate}
\end{theorem}
\begin{proof} Fix $\mu>1$ and let $\xi:\R^d\to(0,1]$ be Borel. Let $(\cF^*_t)_{t<1}$ be the right-continuous filtration generated by $H^*$. The Markov and multiplicative properties of $H^*$ (see \eqref{multprop}) imply that if $\theta(y)=\xi(y(1))$, then
\begin{align}\label{condpgf2}
\nonumber \P^*(\exp\{X^*_{1-\mu^{-n-1}}(\log\xi)\} \, | \, \cF^*_{1-\mu^{-n}})&=\exp(H^*_{1-\mu^{-n}}(\log \overline G_{1-\mu^{-n},1-\mu^{-n-1}}\theta))\\
&=\exp(X^*_{1-\mu^{-n}}(\log G_{1-\mu^{-n},1-\mu^{-n-1}}\xi)),
\end{align}
where \eqref{HXpgfs} is used in the last equality, along with the fact that under $\overline G_{1-\mu^{-n},1-\mu^{-n-1}}$ paths are stopped at $1-\mu^{-n-1}$ (recall \eqref{eq:def overline G st}).
 Use translation invariance and then the scaling property (Lemma~\ref{lem:scaling} with $r=s=1-\mu^{-n}$ and $t=1-\mu^{-n-1}$) to see that
\begin{align}\label{pgfrepn}
\nonumber G_{1-\mu^{-n},1-\mu^{-n-1}}\xi(x)&=G_{1-\mu^{-n},1-\mu^{-n-1}}(\xi\circ\tau_x)(0)\\
\nonumber&=G_{0,1-\mu^{-1}}(\xi\circ\tau_x\circ S_{\mu^{-n}})(0)\\
\nonumber&=\P^*(\exp(X^*_{1-\mu^{-1}}(\log(\xi\circ\tau_x\circ S_{\mu^{-n}}))))\\
 &=\sum_{k=1}^\infty \P^*(X^*_{1-\mu^{-1}}(\R^d)=k)\\
\nonumber&\phantom{=\sum_{k=1}^\infty}\times \P^*\Bigl(\exp(X^*_{1-\mu^{-1}}\left(\log \left(\xi\circ\tau_x\circ S_{\mu^{-n}}\right)\right)) \, \Bigl| \, X^*_{1-\mu^{-1}}(\R^d)=k\Bigr).
\end{align}
By Proposition~\ref{geometricZ}
\begin{equation}\label{massgeom}
\P^*(X^*_{1-\mu^{-1}}(\R^d)=k)=\P(Z=k),\text{where $Z$ has a geometric distribution with mean $\mu$}.
\end{equation}
Recall that $\{\tilde Q_k\}$ are the conditional support measures of $X^*_{1-\mu^{-1}}$. 
Proposition~\ref{Qk} implies 
\begin{multline}\label{Qkrep}
\P^*\Bigl(\exp\left(X^*_{1-\mu^{-1}}\left(\log \left(\xi\circ\tau_x\circ S_{\mu^{-n}}\right)\right)\right) \, \Bigl| \, X^*_{1-\mu^{-1}}(\R^d)=k\Bigr)\\
=\int\prod_{i=1}^k\xi(x+\mu^{-n/2}x_i)\dd \tilde Q_k(x_1,\dots,x_k).
\end{multline}
Substitute \eqref{massgeom} and \eqref{Qkrep} into \eqref{pgfrepn} and conclude
\begin{align*}
G_{1-\mu^{-n},1-\mu^{-n-1}}\xi(x)&=\sum_{k=1}^\infty \P(Z=k)\int\prod_{i=1}^k\xi(x+\mu^{-n/2}x_i)\dd \tilde Q_k(x_1,\dots,x_k)\\
&=\sum_{k=1}^\infty \P(Z=k)\int\prod_{i=1}^k\xi(x+\mu^{-(n+1)/2}x_i)\dd Q_k(x_1,\dots,x_k). 
\end{align*}
Use this representation in the expression \eqref{condpgf2} for the conditional p.g.f. and compare with that in
 Proposition~\ref{nunmarkov} (using also \eqref{Hndef}) to see that the $N_F(\R^d)$-valued Markov Chain, $X^*_{1-\mu^{-n}}$, has the same law as the chain $\{\bar\nu_n\}$ defined from a STRM, $\nu$, with $\beta=1$, offspring law $\cL(Z)$ geometric with mean $\mu$, and symmetric displacement measures $\{Q_k\}$ as in the Theorem.   
 Here note that by \eqref{condmm} and \eqref{Qkmar} each marginal of $Q_k$ is $P_0(\mu^{1/2}B_{1-\mu^{-1}}\in\cdot)$, i.e., is mean zero Gaussian on $\R^d$ with covariance $(\mu-1)I_d$. Therefore \ref{TC} holds for all $\zeta>0$
  and so $\beta=1$, $\cL(Z)$ and $\{Q_k\}$ do determine a unique in law STRM, $\nu$, with its associated $\{\nu_n\}$.  
  (See Remark~\ref{rem:Qklaw} for the uniqueness.) Proposition~\ref{bpshbm} with $\veps_n=\mu^{-n}$ implies that
\[\nu_n=\mu^{-n}\bar\nu_n=\mu^{-n}X^*_{1-\mu^{-n}}\to\frac{2}{\gamma}X_1\quad  \P^*-\text{a.s.}\text{ as }n\to\infty.\]
Hence $\frac{2}{\gamma}X_1$ is the STRM $\nu$ associated with the sequence $\nu_n=\mu^{-n}X^*_{1-\mu^{-n}}$, and the proof is complete.
\end{proof} 

\begin{remark}\label{SBMdim} One easily sees from the above that the random sum $S$ in \eqref{SSndef} giving the mean measure for $\frac{2}{\gamma}X_1$ has a standard normal distribution on $\R^d$, and so the hypothesis \eqref{DC} of Theorem~\ref{Hdim} holds. Therefore a special case of that result and the above theorem give $\dim(\supp(X_1))=2\wedge d$ $\P^*$-a.s., as is well known (\cite{DH79} and \cite{P89}).
\end{remark}

\begin{figure}
	\centering
	\includegraphics{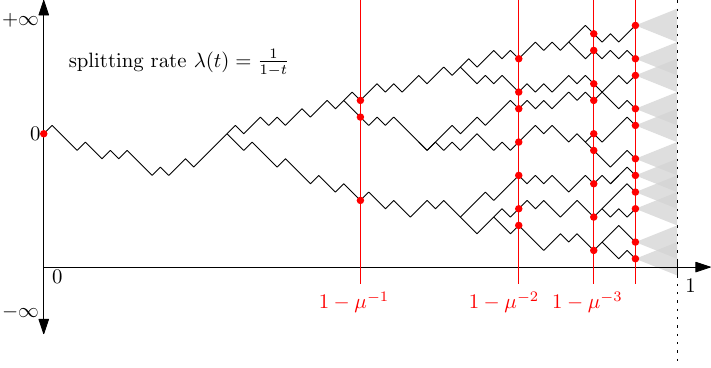}
	\caption{ \label{fig:SBM is STRM}
	Under  the cluster law $\P^*$ the past trajectories of all the  particles ``present at time $1$ in $H_1$"  can be described as those of a particle system that follow independent Brownian motions and split into two at time any $t$ with rate $\lambda(t)=\frac{1}{1-t}$. 
	Note that at any time $t<1$, this system only contains a finite number of particles.  
	Proposition~\ref{bpshbm} ensures that we can recover the random measure $H_1$, historical Brownian motion at time $1$, from the trajectories of this particle system on $[0,1)$. 
	We then take snapshots of the positions of the particles in this process at times $1-\mu^{-n}$ for $n\geq 1$. Theorem~\ref{thm:SBMSTRM} ensures that the evolution of this process along this discrete sequence of times indeed fits into the STRM framework. }
\end{figure} 

\begin{remark}\label{Qkprops} To describe $Q_k$ in the above it suffices to do so for $\tilde Q_k$.  
Recall that $\tilde Q_k$ in the above is the symmetric law of the $k$ points $\{X_1,\dots,X_k\}$ in the support of $X^*_{1-\frac{1}{\mu}}$ under $\P^*(\, \cdot\,| \, X^*_{1-\frac{1}{\mu}}(\R^d)=k)$. 
We sketch a simple construction of these points.  
If $k=1$ we clearly have a single $d$-dimensional Brownian motion $B$ starting at $0$ and $X_1$ is $B(1-\frac{1}{\mu})$.  
For $k\ge 2$, let $V_1,\dots,V_{k-1}$ be i.i.d.\ random times with density $f(v)=(1-v)^{-2}/(\mu-1)$ for $v\in[0,1-\frac{1}{\mu}]$, and let $0<U_1<\dots<U_{k-1}<1-\frac{1}{\mu}$ be the associated order statistics.  
Then $U_1,\dots,U_{k-1}$ will be the $k-1$ branch times of $X^*$ on $[0,1-\frac{1}{\mu}]$. 
To see this let $u_0=0$, $u_k=1-\frac{1}{\mu}$, and let $0<T_1<\dots<T_{k-1}<1-\frac{1}{\mu}$ be the ordered branch times. 
Then $\{T_i\in du_i, i=1,\dots,k-1\text{ and }X^*_{1-\frac{1}{\mu}}(\R^d)=k\}$ holds iff there are no branching events on $(u_{i-1},u_i)$ for $i=1\dots k$ and branching occurs in $du_i$ for $i=1,\dots,k-1$.  
A short calculation then shows this probability is $P((U_1,\dots,U_{k-1})\in du_1... du_{k-1})$.  The construction should now be clear.  
Condition on the branch times occurring at $u_1<\dots<u_{k-1}$. 
Run a $d$-dimensional Brownian path starting from $0$ on $[0,u_1]$. 
Then run two independent Brownian paths on $[u_1,u_2]$ from the end location.  
Randomly pick one of the two new end locations at time $u_2$ and run two independent Brownian paths from this point and a third from the other end location, all on $[u_2,u_3]$.  
Now randomly pick one of the three end locations at time $u_3$ and have this location split into two. 
Continue in this way until we arrive at time $u_{k-1}$ with $k-1$ end locations.  
Pick one at random to split into two particles and run an independent Brownian path from each of these $k$ particles, each ending at time $1-\frac{1}{\mu}$. 
The resulting $k$ end locations give $\{X_1,\dots, X_k\}$.  
The symmetrised joint law of the resulting collection of $k$ correlated Brownian paths on $[0,1-\frac{1}{\mu}]$ leading up to these endpoints, gives the $k$th conditional support measure of $H^*_{1-\frac{1}{\mu}}$. 
\end{remark}

\section{$B$-ary Classical Super-Tree Random Measures}\label{sec:Bary}


Recall the definition of a $B$-ary CSTRM in $\R^d$, associated with $p=(p_k,k\in\Z_+)$ from the Introduction.
In this case we have $\rho=\mu^{-\beta/2}=B^{-1}$ and in the definition of the spatial displacement law $Q^s$ we may take $X_{m,k}=X_m$ for an i.i.d.\ sequence $\{X_m\}$ of uniform r.v.'s on $\{0,\dots,B-1\}^d$. We also may assume that $X^f(\omega)\in\{0,\dots,B-1\}^d$ for all $f\in F$ and all $\omega$. We will alter the notation in Section~\ref{sec:STRM} slightly and, as in \cite{ADHP}, assume $\{X^f:f\in F\setminus F_0\}$ are i.i.d.\ uniform r.v.'s on $\{0,\dots,B-1\}^d$ and, instead of \eqref{Sfdefn} and \eqref{Sidefn}, write
\begin{equation}\label{BaryS}
S^f=\sum_{m=1}^n B^{-m}X^{f|m},\quad f\in F_n; \quad S^i=\sum_{m=1}^\infty B^{-m}X^{i|m},\quad i\in I.
\end{equation}
This new notation amounts to replacing $X^{f|(m-1)}_{f_m}$ in \eqref{Sfdefn} (recall in our simple setting there is no branching dependence, and hence no second subscript) with $X^{f|m}$ for $1\le m\le |f|$, and so clearly does not affect any of the processes defined in Section~\ref{sec:STRM}. Therefore in this setting we have 
\begin{equation}\label{newGdef}
\cG_n=\sigma(Z^f:0\le |f|<n)\vee\sigma(X^f:1\le |f|\le n)\ \ \text{ for }n\in\Z_+.
\end{equation}
Note that $S^i\in[0,1]^d$ for all $i\in I$ and so 
\[\supp(\nu)\subset [0,1]^d.\]
In this setting $S=\sum_{m=1}^\infty B^{-m}X_m$ corresponds to the construction of Lebesgue measure by a random base $B$ expansion in each coordinate (see, e.g., the discussion in Section~V.3a in \cite{Feller2}) and we have
\begin{equation}\label{Bmeanmeas}
\P(S\in A)=\int_{[0,1]^d}1_A(x)\, \dd x.
\end{equation}
Therefore Theorem~\ref{nuprops}(c) shows the mean measure of $\nu$ is Lebesgue measure and we have,
\begin{equation}\label{meanLeb} \E\Bigl(\int \phi(x) \,\dd \nu(x)\Bigr)=\int_{[0,1]^d}\phi(x)\,\dd x\quad\text{for all bounded measurable $\phi$ on $[0,1]^	d$}.
\end{equation}
The boundedness of $X_m$ and \eqref{Bmeanmeas} show that the hypotheses of Theorem~\ref{Hdim} hold and so (as noted in the Introduction) for our $B$-ary CSTRM,
\begin{equation*}
\dim(\supp(\nu))=\frac{2}{\beta}\wedge d=\frac{\log \mu}{\log B}\wedge d\quad \text{a.s.}
\end{equation*}

\begin{notation}  If $x,y\in\R^d$, let $[x,y)=\{z\in \R^d:x_\ell\le z_\ell<y_\ell\text{ for }\ell=1,\dots,d\}$ and let $\vec 1=(1,1,\dots, 1)\in\R^d$. If $m\in\Z_+$, $\vec k\in\{0,1,\dots, B^m-1\}^d$ and $x=B^{-m}\vec k \in [0,1-B^{-m}]^d$, let 
\[ C_m(x)=[x,x+B^{-m}\vec 1)\subset[0,1)^d.\]
That is, $C_m(x)$ is the $d$-dimensional cube of edge length $B^{-m}$ with ``lower left corner" $x\in G_m:=[0,1)^d\cap B^{-m}\Z^d$, see Figure~\ref{fig:definition cube Cm(x)}.
\end{notation}

\begin{figure}
	\centering
	\includegraphics[page=4]{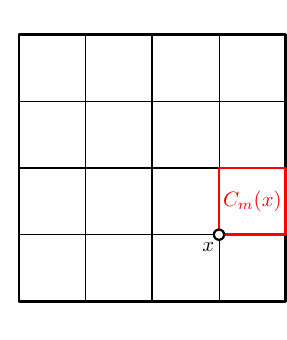}
	\caption{For some site $x\in G_m$ on the grid, the cube $C_m(x)$ is the cube of edge-length $B^{-m}$ whose point with lowest coordinates is $x$.
	\label{fig:definition cube Cm(x)}}
\end{figure}
\begin{figure}
	\centering
	\begin{tabular}{ccc}
		\includegraphics[page=9,width=4cm]{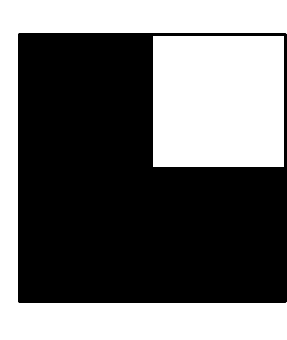} & 	\includegraphics[page=10,width=4cm]{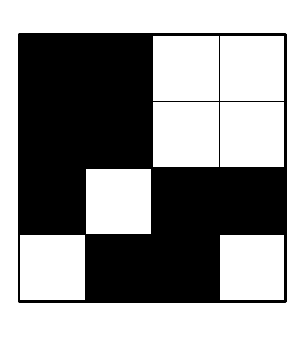} &
		\includegraphics[page=11,width=4cm]{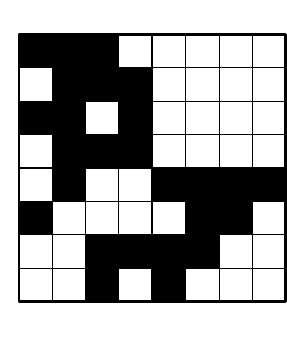}
	\end{tabular}
	\caption{Black cubes correspond to positions $x$ on the grid with at least one particle whereas white cubes corresponds to positions with no particles. The union of the black cubes decreases to $\supp \ \nu$ as $m\rightarrow \infty$, see Lemma~\ref{Cantor}. 
	\label{fig:construction supp cantor set}}
\end{figure}

The following simple random Cantor set description of $\supp(\nu)$, illustrated in Figure~\ref{fig:construction supp cantor set}, will be used frequently.

\begin{lemma}\label{Cantor}
With probability one, $\supp(\nu)=\cap_{m=1}^\infty\cup_{f\in K^0_m}\overline{C_m(S^f)}$.
\end{lemma}
\begin{proof}
 By \eqref{BaryS} we have for all $m\in\N$,
\[S^i\in[S^{i|m},\infty)\text{ for all $i\in I$ and }\sup_{i\in I}\Vert S^i-S^{i|m}\Vert_\infty\le B^{-m}.\]
Therefore for all $m\in\N$ and $i\in K$ we have $S^i\in \overline{C_m(S^{i|m})}$ and $i|m\in K^0_m$.  This proves that 
\[\{S^i:i\in K\}\subset \cap_{m=1}^\infty\cup_{f\in K^0_m}\overline{C_m(S^f)},\]
and so  by Theorem~\ref{nuprops}\ref{it:nuprops:b}, 
\begin{equation}\label{inc1}\supp(\nu)=\{S^i:i\in K\}\subset\cap_{m=1}^\infty\cup_{f\in K^0_m}\overline{C_m(S^f)}\quad \text{a.s.}
\end{equation}

Now let $x\in\cap_{m=1}^\infty\cup_{f\in K^0_m}\overline{C_m(S^f)}$ and for $m\in\N$ choose $f_m\in K^0_m$ so that $x\in \overline{C_m(S^{f_m})}$. 
Define $\bar f_m\in I$
by appending an infinite string of $1$'s to the finite string $f_m$.  
By \eqref{Kbounds} $\bar f_m\in\prod_{\ell=1}^\infty\{1,\dots,M_\ell\vee 1\}$ for all $m$ and so by Tychonoff's theorem there is a subsequence $(\bar f_{m_n})$ converging to $i$ in $I$.  For each $m\in\N$ there is a natural number $L_m\ge m$ so that for $n\ge L_m$, $i|m=f_{m_n}|m\in K^0_m$ (since $f_{m_n}\in K^0_{m_n}$).  This proves $i\in K$.  If $m\in \N$ and $n\ge L_m$, we have 
\begin{align*}
\Vert x-S^i\Vert_\infty&\le \Vert x-S^{f_{m_n}}\Vert_\infty+\Vert S^{f_{m_n}}-S^{f_{m_n}|m}\Vert_\infty+\Vert S^{f_{m_n}|m}-S^i\Vert_\infty\\
&\le B^{-m_n}+2 B^{-m},
\end{align*}
the last by the choice of $f_{m_n}$ and \eqref{BaryS}. Let $n\to\infty$ and then $m\to\infty$ on the right-hand side to see that $x=S^i$.
This gives the converse inclusion to \eqref{inc1}, and we are done. 
\end{proof}

\begin{notation}  If $x^{(j)}$ denotes the $j$th coordinate of $x\in\R^d$, $m\in\Z_+$, and $x\in G_m$, we write $x=.x_1\dots x_m$, where $x_\ell\in\{0,\dots,B-1\}^d$ are the unique vectors such that for $j=1,\dots,d$, $x^{(j)}=\sum_{\ell=1}^mx_\ell^{(j)}B^{-\ell}$. 
If $g\in F_m$ and $\ell\in\N$, let 
\begin{equation}\label{veenotation}
\text{$g\vee \ell\in F_{m+1}$ be the index obtained by adding $\ell$ to $g$ as the last digit.}
\end{equation} 
\end{notation}

\begin{proposition}\label{hitprob1}
There is a constant $c_{\ref{hitprob1}}>0$ and sequences $\delta_m,\veps_m\to0$, depending only on $(B^d,\cL(Z))$, such that:
\begin{enumerate}[label=\emph{(\alph*)}, ref=(\alph*)]
	\item\label{it:hitprob1:a} If $\frac{2}{\beta}=d$, then
	\[\forall m\in\N, \forall x\in G_m,\quad \P(\nu(C_m(x))>0)=\P(\nu(\overline{C_m(x)})>0)=\frac{2\mu(1+\delta_m)}{(\mu-1)m}.\]
	\item\label{it:hitprob1:b} If $\frac{2}{\beta}<d$, then
	\[\forall m\in\Z_+, \forall x\in G_m,\quad \P(\nu(C_m(x))>0)=\P(\nu(\overline{C_m(x)})>0)=c_{\ref{hitprob1}}(1+\veps_m)(B^{-m})^{(d-\frac{2}{\beta})}.\]
\end{enumerate}
\end{proposition}
\begin{proof} By \eqref{meanLeb} $\nu(\partial C_m(x))=0$ a.s.\ and so we only need prove the above without the closures.  
	Let $m\in\Z_+$ and $x=.x_1\dots x_m\in G_m$.  If $n\ge m$  and $f\in F_n$, then by \eqref{BaryS}, $S^f\in C_m(x)$ iff $S^{f|m}=x$. 
	Therefore
\begin{equation*} \nu_n(C_m(x))=\mu^{-n}|\{f\in K_n^0:S^{f|m}=x\}|=\sum_{g\in K^0_m}1(S^g=x)\mu^{-n}|K^g_n|.
\end{equation*}
Use the above with the facts that $\nu_n\to \nu$ a.s.\ (Theorem~\ref{nuprops}\ref{it:nuprops:d}) and $\nu(\partial C_m(x))=0$ a.s.\ to see that 
\begin{align}\label{nuCm}
\nonumber\nu(C_m(x))=\lim_{n\to\infty}\nu_n(C_m(x))&=\lim_{n\to\infty}\sum_{g\in K^0_m}1(S^g=x)\mu^{-n}|K^g_n|\\
&=\sum_{g\in K_m^0}1(S^g=x)W^g.
\end{align}
If $x\in G_m$, set $J_0^x=K^0_0=\{0\}$ if $m=0$ (and hence $x=0$), and for $m\ge 1$ define
\begin{align}\label{Jdefn}
\nonumber J_m^x&=\{g\in K_m^0:S^g=x\}\\
\nonumber &=\{f\vee g_m: f\in K^0_{m-1},\ S^f=.x_1,\dots x_{m-1},\ 1\le g_m\le Z^f,\ X^{f\vee g_m}=x_m\}\\
&=\dot\cup_{f\in J_{m-1}^{.x_1\dots x_{m-1}}}\dot\cup_{g_m=1}^{Z^f}\{f\vee g_m:\, X^{f\vee g_m}=x_m\}.
\end{align}
Here $\dot \cup$ denotes union of disjoint sets and the empty $B$-ary expansion denotes zero.  
Therefore \eqref{Jdefn} shows that
\begin{equation}\label{eq:evolution Jmx}
|J^x_m|=\sum_{f\in J^{.x_1\dots x_{m-1}}_{m-1}}\sum_{\ell=1}^{Z^f}1(X^{f\vee\ell}=x_m),\quad m\ge 1,\ x\in G_m.
\end{equation}
Note that $J^x_m$ is $\cG_m$-measurable, while for $m\ge 1$, $\{Z^f:f\in F_{m-1}\}$ and $\{X^{f'}:f'\in F_m\}$ are independent i.i.d.\ collections, jointly independent of $\cG_{m-1}$. 
This shows that for $m\in \N$ and $x\in G_m$, $k\mapsto |J^{.x_1\dots x_k}_k|$, $k=0,\dots m$ is a Galton-Watson branching process starting at $1$ and with offspring law, that of $R=\sum_{\ell =1}^Z e_\ell$, where $Z$ is our branching variable for $\nu$, and $\{e_\ell\}$ are i.i.d.\ Bernoulli r.v.'s with parameter $p=B^{-d}$, independent of $Z$. 
Note that this law depends on $(B^d,\cL(Z))$ and not on the choice of $x\in G_m$ or $m$, and satisfies
\[\E(R)=\frac{\mu} {B^{d}}=B^{\frac{2}{\beta}-d},\quad \text{Var}(R)=\sigma^2_R=\frac{\mu}{B^d}(1-B^{-d}).  \]
From \eqref{nuCm} we have for $m\in\Z_+$ and $x\in G_m$,
\begin{align}\label{indWJ}
\nonumber\nu(C_m(x))=\sum_{g\in J^x_m}W^g,&
\text{ where }
\{W^g:g\in F_m\}\text{ are i.i.d.\ r.v.'s with common law $\cL(\mu^{-m}W^0)$}\\
&\text{ and are independent of }\cG_m,
\text{ and hence of }J^x_m.
\end{align}
Note that $q:=\P(W^0=0)<1$ (recall $\E(W^0)=1$).  By \eqref{indWJ},
\begin{align}\label{nuposexp}
\nonumber \P(\nu(C_m(x))>0)&=\P(\exists\, g\in J^x_m \text{ such that }W^g>0)\\
\nonumber&=\P(|J^x_m|>0)\P(\exists\, g\in J^x_m\text{ such that }W^g>0 \ |\, |J^x_m|>0)\\
&=\P(|J^x_m|>0)(1-\E(q^{|J^x_m|}|\,|J^x_m|>0)).
\end{align}
The right-hand side is clearly independent of the choice of $x\in G_m$ by the above branching process description of $|J^x_m|$.

\ref{it:hitprob1:a} Assume now $\frac{2}{\beta}=d$, so that $B^d=\mu$, a	nd let $m\in\N$. 
Then $\E(R)=1$ and we have a critical branching process in the above.  
By a theorem of Kolmogorov (see (10.8) in \cite{Har}) there is a sequence $\delta^{(1)}_m\to 0$ such that 
\begin{equation}\label{prpos1}
\forall x\in G_m,\qquad \P(|J^x_m|>0)=\frac{2(1+\delta^{(1)}_m)}{m\sigma_R^2}=\frac{2\mu(1+\delta_m^{(1)})}{(\mu-1)m}.
\end{equation}
By Yaglom's theorem (see Theorem~10.1 in \cite{Har}) 
\begin{equation*}
\P\left(\frac{|J^x_m|2}{m\sigma^2_R}\in\cdot \ \Bigl| \ |J^x_m|>0\right)\text{ converges weakly to an $\text{Exponential}(1)$ r.v.}
\end{equation*}
This implies that $\E(q^{|J^x_m|} \, | \, |J^x_m|>0)=\delta^{(2)}_m\to 0$ as $m\to\infty$ (and is independent of the choice of $x\in G_m$). 
Use this and \eqref{prpos1} in \eqref{nuposexp} to complete the proof of \ref{it:hitprob1:a}. 
Note that the final sequence $(\delta_m)$ depends only on $\cL(R)$, and hence only on $\cL(Z)$ (recall $B^d=\E(Z)$ in this case), and not on the choice of $x\in G_m$. 

\ref{it:hitprob1:b} Assume $\frac{2}{\beta}<d$ so that $\E(R)<1$ and the above branching process is subcritical. 
Let $m\in\Z_+$. 
By (9.5) in Chapter I of \cite{Har} there is a $c_1>0$ and a sequence $\veps^{(1)}_m\to 0$ so that 
\begin{equation}\label{prpos2}
\P(|J^x_m|>0)=c_1(1+\veps_m^{(1)})\E(R)^m=c_1(1+\veps_m^{(1)})(B^{-m})^{(d-\frac{2}{\beta})}.
\end{equation}
By Theorem~9.1 in Chapter~I of \cite{Har} there is a random variable $J_\infty\in\N$ (whose law depends only on $\cL(R)$) so that 
\[ \lim_{m\to\infty}\E\left(q^{|J^x_m|}\, | \, |J^x_m|>0\right)=\E(q^{J_\infty}):=c_2\in(0,q].\]
Therefore for some $\veps_m^{(2)}\to 0$, we may write
\begin{equation*}\E\left(1-q^{|J^x_m|} \, | \, |J^x_m|>0\right)=(1-c_2)(1+\veps^{(2)}_m).
\end{equation*}
Use this and \eqref{prpos2} in \eqref{nuposexp} to complete the proof of \ref{it:hitprob1:b}, where  $c_{\ref{hitprob1}}$ and $\{\veps_m\}$ depend only on $\cL(R)$, and hence only on $(B^d,\cL(Z))$. 
\end{proof}

\begin{notation}  Let $\Lambda_m=\Lambda_m(d)=\{\overline{C_m(x)}:x\in G_m\}$. 
\end{notation}

It is easy to use Proposition~\ref{hitprob1} to prove Proposition~\ref{baryhmeas}.  

\begin{proof}[Proof of Proposition~\ref{baryhmeas}] \ref{it:baryhmeas:a}  Write $\diam(C)$ for the Euclidean diameter for a set in $\R^d$. We have $\supp(\nu)\subset\cup_{C\in\Lambda_m, \nu(C)>0}\,C$, and so it follows easily from the definition of Hausdorff measure, Fubini's theorem and the fact that $\nu(\partial C)=0$ for all $C\in \Lambda_m$ a.s.\ that
\begin{align*}
\E&(x^d\log(1/x)\text{-}\mathrm{m}(\supp(\nu))\\
&\le \E\Bigl(\liminf_{m\to\infty}\sum_{x\in G_m}\diam(\overline{C_m(x)})^d\log(1/\diam(\overline{C_m(x)}))1(\nu(\overline{C_m(x)})>0)\Bigr)\\
&\le \liminf_{m\to\infty}\sum_{x\in G_m}(dB^{-2m})^{d/2}\log(B^m)\P(\nu(\overline{C_m(x)})>0)\\
&\le\liminf_{m\to\infty}B^{dm}(dB^{-2m})^{d/2}m\log(B)C(B^d,\cL(Z))m^{-1}\\
&\le C(d,B,\cL(Z))<\infty.
\end{align*}
We have used Proposition~\ref{hitprob1}\ref{it:hitprob1:a} in the penultimate line of the display.  
This proves \ref{it:baryhmeas:a}.

\ref{it:baryhmeas:b} Virtually the same argument using Proposition~\ref{hitprob1}\ref{it:hitprob1:b} gives \ref{it:baryhmeas:b}.
\end{proof}

Proposition~\ref{hitprob1} also readily gives Proposition~\ref{hitprob2}. 
\begin{proof}[Proof of Proposition~\ref{hitprob2}.] If $B_\infty(y,r)$ denotes the open $L^\infty$ ball in $\R^d$, then the inclusion $B_\infty(y,r/\sqrt d)\subset B(y,r)\subset B_\infty(y,r)$ shows it suffices to prove the result for $B_\infty(y,r)$ in place of $B(y,r)$. 
For
 $r>0$ sufficiently small, 
choose $m= - \lceil \frac{\log(r/2) }{ \log B} \rceil\in\N$ so that $B^{-m}<r/2\le B^{-m+1}$. 
If $y\in [0,1]^d$ choose $x\in G_m$ so that $\Vert y-x\Vert_\infty\le B^{-m}$. The triangle inequality and the choice of $m$ imply that $C_m(x)\subset B_\infty(y,r)$ and so the required lower bounds now follow from Proposition~\ref{hitprob1} and the choice of $m$.  From the fact that $2r\le (4B)B^{-m}$, one easily sees that $B_\infty(y,r)\cap [0,1]^d$ is contained in the union of at most $(4B+2)^d$ cubes in $\Lambda_m$.  The required upper bounds now also follow from Proposition~\ref{hitprob1}.
\end{proof}

\section{Total Disconnectedness of $B$-ary Classical Super-Tree Random Measures}\label{secTDBary}

In this section we prove Theorem~\ref{thm:totdisc} giving sufficient conditions for total disconnectedness of the support of $B$-ary CSTRM.
\begin{notation} If $m\in\N$ and $x\in G_m$, for $n\ge m$ we define a closed subset of $[0,1]^d$ by
\[
\widetilde C_{m,n}(x)=\cup_{f\in K^0_m,S^f=x}\cup_{f'\in K^f_n} \overline{C_n(S^{f'})}.
\]
\end{notation}

The key step in the proof will be the following disjointness lemma
which ensures that if two particles are at positive distance at some point in the process, then the distance between their respective descendants will remain bounded away from $0$ for all times.

\begin{lemma}\label{disjcells} Let $d\ge \frac{4}{\beta}-1$.  If $m\in\N$ and $f,g\in F_m$, then w.p. $1$ there is an a.s. finite random variable  $N=N(m,f,g)\in\N^{\ge m}$ a.s. so that 
\begin{multline*}S^f\neq S^g\text{ and }\overline{C_m(S^f)}\cap\overline{C_m(S^g)}\neq\emptyset\\
\text{ imply that} \quad  \forall n\ge N, \  \forall (f',g')\in K_n^f\times K_n^g, \ \overline{C_n(S^{f'})}\cap\overline{C_n(S^{g'})}=\emptyset.
\end{multline*}
\end{lemma}
We will prove this result below.  
The core of the proof, contained in Lemma~\ref{gammaext} below, relies on a supermartingale convergence argument requiring $d\ge \frac{4}{\beta}-1$. For now we will show how it easily gives Theorem~\ref{thm:totdisc}. 
\begin{proof}[Proof of Theorem~\ref{thm:totdisc}]
We may fix $\omega$, outside a null set, so that the above Lemma holds for all $m\in\N$ and all $f,g\in F_m$, the conclusion of Lemma~\ref{Cantor} holds, and $|K^0_m|<\infty$ for all $m\in\N$. If $N_m(\omega)=\max_{f,g\in K^0_m}N(m,f,g)\in \N^{\ge m}$, then 
\begin{align}\label{disjoint}
\nonumber\forall m\in\N, \ \forall (f,g)\in (K_m^0)^2,\quad &\left(S^f\neq S^g \text{ and }\overline{C_m(S^f)}\cap\overline{C_m(S^g)}\neq\emptyset\right) \text{ implies }\\ 
\forall n\ge N_m, \
&\forall(f',g')\in K_n^f\times K_n^g, \quad  \overline{C_n(S^{f'})}\cap\overline{C_n(S^{g'})}=\emptyset.
\end{align}
Note also that if $f,g\in K^0_m$ and $\overline{C_m(S^f)}\cap\overline{C_m(S^g)}=\emptyset$, then the conclusion in \eqref{disjoint} holds easily for all $n\ge m$.  This follows from \eqref{BaryS} because that result implies for all $f,g\in K^0_m$, 
\begin{equation}\label{nested}
\cup_{f'\in K_n^f}\overline{C_n(S^{f'})}\subset \overline{C_m(S^f)}\text{ and }\cup_{g'\in K_n^{g}}\overline{C_n(S^{g'})}\subset \overline{C_m(S^g)}.
\end{equation}
Therefore we have from \eqref{disjoint} that
\begin{multline}\label{disjoint2}
\forall m\in\N, \  \forall (f,g)\in (K_m^0)^2, \\
S^f\neq S^g \quad \text{ implies } \quad  \forall n\ge N_m, \ \forall (f',g')\in K_n^f\times K_n^g, \quad  \overline{C_n(S^{f'})}\cap\overline{C_n(S^{g'})}=\emptyset.
\end{multline}
It follows easily from \eqref{nested} and the definition of $\widetilde C_{m,n}(x)$ that for $n\ge m$,
\begin{equation}\label{diam}
\widetilde C_{m,n}(x)\subset\overline{C_m(x)},\text{ and so }\diam(\widetilde C_{m,n}(x))\le B^{-m}.
\end{equation}
Here the diameter is taken with respect to the $L^\infty$ norm on $\R^d$.  We also have for $n\ge m$, 
\begin{align}\label{suppnudiscn}\nonumber\cup_{x\in G_m}\widetilde C_{m,n}(x)=\cup_{x\in G_m}\cup_{f\in K^0_m,\,S^f=x}\cup_{f'\in K^f_n}\overline{C_n(S^{f'})}&=\cup_{f\in K^0_m}\cup_{f'\in K^f_n}\overline{C_n(S^{f'})}\\
&=\cup_{f'\in K_n^0}\overline{C_n(S^{f'})}\supset\supp(\nu).
\end{align}
We have used Lemma~\ref{Cantor} in the last inclusion, and in the next to last inclusion note that $f'\in K^0_n$ iff $f'|m\in K^0_m$ and $f'\in K_n^{f'|m}$. 

If $x,y$ are distinct points in $G_m$, then we have for $n\ge N_m$,
\begin{align*}\widetilde C_{m,n}(x)\cap\widetilde C_{m,n}(y)=\cup_{f\in K_m^0,S^f=x}\cup_{g\in K^0_m,S^g=y}\cup_{(f',g')\in K^f_n\times K^g_n} \overline{C_m(S^{f'})}\cap\overline{C_m(S^{g'})}=\emptyset,
\end{align*}
where the final equality holds by \eqref{disjoint2} because $S^f=x\neq y=S^g$. So this and \eqref{suppnudiscn} allows us to write $\supp(\nu)$ (for $n>N_m$) as a finite disjoint union of closed sets, $\supp(\nu)\cap \widetilde C_{m,n}(x)$ $x\in G_m$, each of $L^\infty$-diameter at most $B^{-m}$ by \eqref{diam}.  
As $m\in\N$ is arbitrary, this proves $\supp(\nu)$ is a.s.\ totally disconnected.
\end{proof} 

Turning to the proof of Lemma~\ref{disjcells}, we introduce some terminology. 

\begin{definition} If $C,C'\in \Lambda_m$ for $m\in\N$, we say $C$ and $C'$ are neighbours iff $C\neq C'$ and $\overline{C}\cap\overline{C'}$ is non-empty.
\end{definition} 

\begin{notation}  If $x,y\in G_m$, let $L=L(x,y)=L(C_m(x),C_m(y))=\{1\le k\le d:x^k=y^k\}$, so that $|L|\in\{0,\dots,d-1\}$ for neighbouring cubes in $\Lambda_m$. $L^c$ denotes the complement of $L$ in $\{1,\dots,d\}$.
\end{notation}

\begin{lemma}\label{ellneighb} Let $x=(x^1,\dots,x^d), y=(y^1,\dots,y^d)\in G_m$, for $m\in\N$. If $L=L(x,y)$, then
\[\text{$C_m(x)$ and $C_m(y)$ are neighbours $\iff$ $x\neq y$ and for any $k\in L^c$, $|x^k-y^k|=B^{-m}$.}\]
In this case
\begin{equation}\label{Cinters}\overline{C_m(x)}\cap\overline{C_m(y)}=\{z:z^k\in[x^k,x^k+B^{-m}]\text{ for }k\in L,\text{ and }z^k=x^k\vee y^k\text{ for }k\in L^c\}.
\end{equation}
\end{lemma}
\begin{proof}
$(\Rightarrow)$ Note that if $|x^k-y^k|>B^{-m}$, then $|x^k-y^k|\ge 2B^{-m}$, and so $\overline{C_m(x)}\cap\overline{C_m(y)}$ is empty, implying that $C_m(x)$ and $C_m(y)$ are not neighbours.\\
$(\Leftarrow)$ 
Assume $x\neq y$ and $|x^k-y^k|=B^{-m}$ for $k\in L^c$.  Then 
\begin{align*} z\in \overline{C_m(x)}\cap\overline{C_m(y)}\quad &\text{ iff } \quad \forall k\le d, \ z^k\in[x^k,x^k+B^{-m}]\cap[y^k,y^k+B^{-m}]\\
&\text{ iff } \quad \forall k\in L, \ z^k\in[x^k,x^k+B^{-m}] \quad  \text{and} \quad \forall k\in L^c, \ z^k=x^k\vee y^k.
\end{align*}
In particular, we see that $C_m(x)$ and $C_m(y)$ are neighbours because this last set of $z$ is non-empty. This also proves \eqref{Cinters}.
\end{proof}

\begin{definition}
For neighbouring $C,C'\in \Lambda_m$ and $\ell\in\{0,\dots, d-1\}$, we say $C$ and $C'$ are $\ell$-neighbours iff $|L(C,C')|=\ell$. (Geometrically this means $\overline C\cap \overline{C'}$ is an $\ell$-dimensional face, see Figure~\ref{fig:neighbour cubes} for an illustration.) We say $f$ and $g$ in $F_m$ are $\ell$-neighbours iff $C_m(S^f)$ and $C_m(S^g)$ are.  
\end{definition}

\begin{notation}  For $f\in F$ and $i\le d$ we let $S^{f,i}$ denote the $i$th coordinate of $S^f$, and similarly for $X^{f,i}$.
\end{notation}
\begin{figure}
	\centering
	\begin{tabular}{ccc}
		\includegraphics[page=1,width=4cm]{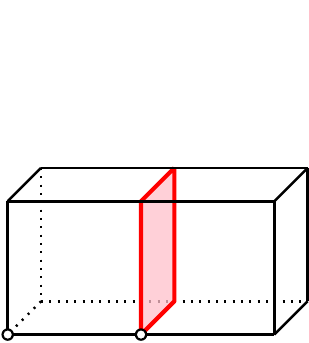} & 	\includegraphics[page=2,width=4cm]{neighboring_cubes} & 	\includegraphics[page=3,width=4cm]{neighboring_cubes}
	\end{tabular}
	\caption{From left to right, examples of $\ell$-neighbour cubes, for $\ell = 2,1,0$. The intersection of their boundary is shown in red. 
	\label{fig:neighbour cubes}}
\end{figure}
The next two results involve only deterministic reasoning, but for subsequent application we state them for random variables.
\begin{lemma}\label{ellnab}
Let $n_1\le n_2$  be in $\N$, $f,g\in F_{n_2}$, and $\ell\in\{0,\dots,d-1\}$.
\begin{enumerate}[label=\emph{(\alph*)}, ref=(\alph*)]
\item\label{it:L is non-increasing} $L(S^f,S^g)\subset L(S^{f|n_1},S^{g|n_1})$, and so in particular, $S^{f|n_1}\neq S^{g|n_1}$ implies $S^f\neq S^g$.
\item\label{it:neighbours descend from neighbours} If $C_{n_2}(S^f)$ and $C_{n_2}(S^g)$ are $\ell$-neighbours and $S^{f|n_1}\neq S^{g|n_1}$, then $C_{n_1}(S^{f|n_1})$ and $C_{n_1}(S^{g|n_1})$ are $\ell'$-neighbours for some $\ell'\ge \ell$.
\end{enumerate}
\end{lemma}
\begin{proof}Let $f,g,n_1,n_2$ be as above. If $k\in\{1,\dots,d\}$, then by \eqref{BaryS}
\begin{align}\label{Sincr}
\nonumber|S^{f,k}-S^{g,k}|&\ge |S^{f|n_1,k}-S^{g|n_1,k}|-\sum_{m=n_1+1}^{n_2}B^{-m}|X^{f|m,k}-X^{g|m,k}|\\
\nonumber&\ge |S^{f|n_1,k}-S^{g|n_1,k}|-B^{-n_1-1}(B-1)\sum_{m=0}^{n_2-n_1-1}B^{-m}\\
&= |S^{f|n_1,k}-S^{g|n_1,k}|-B^{-n_1}+B^{-n_2}.
\end{align}

If $k\in L(S^{f|n_1},S^{g|n_1})^c$, then $|S^{f|n_1,k}-S^{g|n_1,k}|\ge B^{-n_1}$ and so \eqref{Sincr} implies that \break
 $|S^{f,k}-S^{g,k}|\ge B^{-n_2}>0$. This shows that $k\in L(S^f,S^g)^c$, proving \ref{it:L is non-increasing}.

For \ref{it:neighbours descend from neighbours}, assume that $C_{n_2}(S^f)$ and $C_{n_2}(S^g)$ are $\ell$-neighbours, and $S^{f|n_1}\neq S^{g|n_1}$. 
Let $k\in L(S^{f|n_1},S^{g|n_1})^c$.  
Then $k\in L(S^f,S^g)^c$ by \ref{it:L is non-increasing}, and so the left-hand side of \eqref{Sincr} is $B^{-n_2}$ by Lemma~\ref{ellneighb}. 
Therefore \eqref{Sincr} implies that
\[|S^{f|n_1,k}-S^{g|n_1,k}|\le B^{-n_1}.\]
Recalling that $S^{f|n_1,k}\neq S^{g|n_1,k}$ by the choice of $k$, we deduce that $|S^{f|n_1,k}-S^{g|n_1,k}|=B^{-n_1}$. 
In addition we have $S^{f|n_1}\neq S^{g|n_1}$ and so Lemma~\ref{ellneighb} implies that $C_{n_1}(S^{f|n_1})$ and $C_{n_1}(S^{g|n_1})$ are neighbours. 
By \ref{it:L is non-increasing} they are $\ell'$ neighbours for some $\ell'\ge \ell$. 
\end{proof}

\begin{notation}  If $R$ is a $\Z_+$-valued r.v., then for any $f ,g\in F_R$, $\ell\in\{0,\dots,d-1\}$ and $n\in\N$, we let 
\[\Gamma_{R,n}^\ell(f,g):=\{(f',g')\in K_{n\vee R}^f\times K_{n\vee R}^g:\, C_{n\vee R}(S^{f'})\text{ and }C_{n\vee R}(S^{g'})\text{ are }\ell-\text{neighbours}\}.\]
\end{notation}

\begin{corollary}\label{gamma0} 
Let $R$ be a $\Z_+$-valued random variable. 
	For all  $f ,g\in F_R$, and $\ell\in\{0,\dots,d-1\}$, if $S^f\neq S^g$, then for all $n\in\N$, we have	
\[ \left(\forall \ell'\in\{\ell,\dots,d-1\}, \ \Gamma_{R,n}^{\ell'}(f,g)=\emptyset\right)  \Rightarrow \left(\forall n'\ge n, \forall \ell'\in\{\ell,\dots,d-1\}, \Gamma_{R,n'}^{\ell'}(f,g)=\emptyset\right).\]
In particular ($\ell=0,\, n=R$),  
if $f,g\in F_R$ are not neighbours, then 
for all $n'\ge R$, 
$f'\in K^f_{n'}$ and $g' \in K^g_{n'}$ are not neighbours.
\end{corollary}

\begin{remark}	
Note that it is immediate that stating the above result for any $\Z_+$-valued r.v. $R$ is equivalent to stating it for all possible values in $\Z_+$. 
The current formulation happens to be more convenient for our proofs. 
The same will be true for the statement of Lemma~\ref{gammaext} below; see the first paragraph of the proof of Lemma~\ref{gammaext}.
\end{remark}

\begin{proof}[Proof of Corollary~\ref{gamma0}]
	 Assume $S^f\neq S^g$ and $\Gamma_{R,n}^{\ell'}(f,g)=\emptyset$ for all $\ell'\ge \ell$.  
	Let $\ell'\in\{\ell,\dots,d-1\}$, $n'\ge n$ and, proceeding by contradiction, let $(f',g')\in \Gamma^{\ell'}_{R,n'}(f,g)$. 
	Then $(f'|(n\vee R),g'|(n\vee R))\in K^f_{n\vee R}\times K^g_{n\vee R}$ and, by Lemma~\ref{ellnab}\ref{it:L is non-increasing},  $S^{f'|n\vee R}\neq S^{g'|n\vee R}$ because $S^{f'|R}=S^f\neq S^g=S^{g'|R}$. 
	By Lemma~\ref{ellnab}\ref{it:neighbours descend from neighbours} if $C_{n'\vee R}(S^{f'})$ and $C_{n'\vee R}(S^{g'})$ are $\ell'$-neighbours, then $C_{n\vee R}(S^{f'|(n\vee R)})$ and $C_{n\vee R}(S^{g'|(n\vee R)})$ are $\ell''$-neighbours for some $\ell''\in\{\ell',\dots,d-1\}$. 
	We have shown that $(f'|(n\vee R),g'|(n\vee R))\in \Gamma^{\ell''}_{R,n}(f,g)$ which is empty by assumption.  
	This contradiction proves that 
$\Gamma^{\ell'}_{R,n'}(f,g)$ is empty, and the proof is complete.
\end{proof}

\begin{lemma}\label{gammaext}
Assume $d\ge\frac{4}{\beta}-1$, $\ell\in\{0,\dots,d-1\}$, and $R\in\N\cup\{\infty\}$ is a random
variable.
With probability $1$, 
if $R<\infty$ and $f,g\in F_R$ are such that $C_R(S^f)$ and $C_R(S^g)$ are not $\ell'$-neighbours for any $\ell'\in\{\ell+1,\dots,d-1\}$, then 
\begin{equation}\label{Rfinite}R_R^\ell(f,g):=\inf\{n\ge R:\, \Gamma^\ell_{R,n}(f,g)=\emptyset\}<\infty.
\end{equation}
\end{lemma}

If $C_R(S^f)$ and $C_R(S^g)$ are not $\ell$-neighbours, then the above conclusion is  trivial because 
$\Gamma^\ell_{R,R}(f,g)=\emptyset$ so that 
 $R^\ell_R(f,g)=R <\infty$ in that case.  
We will prove Lemma~\ref{gammaext} below. 
The route to establishing Lemma~\ref{disjcells} from Corollary~\ref{gamma0} and Lemma~\ref{gammaext} is now clear.  If
$f,g$ are as in Lemma~\ref{disjcells}, then for some $\ell\le d-1$ they are $\ell$-neighbours but not $\ell'$-neighbours for any $\ell'>\ell$.  Corollary~\ref{gamma0} and Lemma~\ref{gammaext} imply that after a large random time none of their descendants are $\ell'$-neighbours for any $\ell'>\ell-1$.  Iterating this argument $\ell$ times shows that after a sufficiently large time no pair of descendants of $f$ and $g$ are neighbours, which implies Lemma~\ref{disjcells}.  Here are the details.

\begin{proof}[Proof of Lemma~\ref{disjcells}]
Let $m\in\N$, $\ell\in\{0,\dots,d-1\}$ and $f,g\in F_m$. It suffices to show that w.p.$1$ on $\Omega_\ell=\{\omega:\ f\text{ and }g \text{ are $\ell$-neighbours}\}$,
\begin{equation}\label{stepzero}
\exists N_\ell\ge m\text{ such that }\forall n\ge N_\ell, \  \forall(f',g')\in K^f_n\times K^g_n, \quad \overline{C_n(S^{f'})}\cap\overline{C_n(S^{g'})}=\emptyset.
\end{equation}
Indeed, take the union over $\ell$, let $N$ be the maximum of the resulting $N_\ell$'s in \eqref{stepzero}, and combine  the null sets, to get the conclusion of Lemma~\ref{disjcells}.

Apply Lemma~\ref{gammaext} with $R=m$ to see that if $\inf\emptyset=\infty$ and we set $R_1=\infty$ on $\Omega_\ell^c$, then
\begin{equation}\label{R1fin}
R_1:=\inf\{n\ge m:\, \Gamma^\ell_{m,n}(f,g)=\emptyset\}<\infty\ \text{a.s.\ on }\Omega_\ell.\end{equation}
Corollary~\ref{gamma0}, with $R=n=m$ and $n'=R_1$, shows that on $\{R_1<\infty\}$, for all $\ell'>\ell$ we have  
$\Gamma^{\ell'}_{m,R_1}(f,g)=\emptyset$, while the definition of $R_1$ ensures $\Gamma^\ell_{m,R_1}(f,g)=\emptyset$ on $\{R_1<\infty\}$. 
So Corollary~\ref{gamma0}, with $R=n=R_1$, now implies
\begin{equation}\label{stepR1}
\forall n'\in\N^{\ge R_1}, \ \forall (f',g')\in K_{n'}^f\times K^g_{n'}, \ f'\text{ and }g'\text{ are not $\ell'$-neighbours for all }\ell'\ge \ell,
\end{equation}
where we note this conclusion is vacuously true if $R_1=\infty$.  

By \eqref{stepR1} on $\{R_1<\infty\}$ any $f'\in K^f_{R_1}$ and $g'\in K^g_{R_1}$ are not $\ell'$-neighbours for any $\ell'\ge \ell$, and so Lemma~\ref{gammaext} (with $R=R_1$) implies that w.p.$1$ on $\{R_1<\infty\}$, for all $(f',g')\in K^f_{R_1}\times K^g_{R_1}$, we have  $R^{\ell-1}_{R_1}(f',g')<\infty$.  
We conclude that if we set $R_2=\infty$ on $\{R_1=\infty\}$, then 
\begin{equation}\label{R2fin}R_2:=\max_{(f',g')\in K^f_{R_1}\times K^g_{R_1}}R^{\ell-1}_{R_1}(f',g')<\infty\text{ a.s.\ on }\{R_1<\infty\}.
\end{equation}
Assume $R_1<\infty$ and let $(f',g')\in K^f_{R_1}\times K^g_{R_1}$. 
We have $S^f\neq S^g$ (since $R_1<\infty$ implies $\Omega_\ell$) and therefore $S^{f'}\neq S^{g'}$ by Lemma~\ref{ellnab}\ref{it:L is non-increasing}, and, as noted above, $f'$ and $g'$ are not $\ell'$-neighbours for all $\ell'\ge \ell$ by \eqref{stepR1}. 
So we can apply Corollary~\ref{gamma0} with $(f',g')$ in place of $(f,g)$, $R=n=R_1$,  and $n' = R^{\ell-1}_{R_1}(f',g')$, to conclude that $\Gamma^{\ell'}_{R_1,R^{\ell-1}_{R_1}(f',g')}(f',g')=\emptyset$ for $\ell'\ge \ell$ on $\{R^{\ell-1}_{R_1}(f',g')<\infty\}$. 
The same conclusion holds for $\ell'=\ell-1$ by the definition of 
$R^{\ell-1}_{R_1}(f',g')$. 
Another application of Corollary~\ref{gamma0} with $R=R_1$, $n=R^{\ell-1}_{R_1}(f',g')$ and $n'\ge R_2$ shows that
for all $(f',g')\in K^f_{R_1}\times K^g_{R_1}$ and all $n'\ge R_2$, $\Gamma^{\ell'}_{R_1,n'}(f',g')=\emptyset$ for all $\ell'\ge \ell-1$ (this is vacuous if $R_2=\infty$).  
This means that
\begin{multline}\label{noneighb}
\forall n'\in\N^{\ge R_2}, \ \forall (f',g')\in K^f_{R_1}\times K^g_{R_1}, \forall (f'',g'')\in K^{f'}_{n'}\times K^{g'}_{n'}, \\
\text{$f''$ and $g''$ are not $\ell'$-neighbours for all $\ell'\ge \ell-1$}.
\end{multline}
Note that for any $n'\ge R_2$ and $f'',g''\in F_{n'}$, $(f'',g'')\in K^f_{n'}\times K^g_{n'}$ iff $(f',g')=(f''|R_1,g''|R_1)\in K^f_{R_1}\times K^g_{R_1}$ and $(f'',g'')\in K_{n'}^{f'}\times K_{n'}^{g'}$.  
Therefore we can rewrite \eqref{noneighb} as 
\begin{equation}\label{stepR2} \forall n'\in\N^{\ge R_2}, \ \forall (f',g')\in K_{n'}^f\times K_{n'}^g, \ f'\text{ and }g'\text{ are not $\ell'$-neighbours for all }\ell'\ge \ell-1.
\end{equation}
Comparing \eqref{stepR1} to \eqref{stepR2}, and noting that $R_2<\infty$ a.s.\ on $\Omega_\ell$ by \eqref{R1fin} and \eqref{R2fin},  we see that we can iterate this argument $\ell$ times and conclude there is a $(\cG_n)$-stopping time $R_{\ell+1}$, which is a.s.\ finite on $\Omega_\ell$, such that
\begin{equation*}
\forall n'\in\N^{\ge R_{\ell+1}}, \ \forall (f',g')\in K^f_{n'}\times K^g_{n'}, \ f'\text{ and }g'\text{ are not $\ell'$-neighbours for all }\ell'\ge 0.
\end{equation*}
This gives \eqref{stepzero} with $N_\ell=R_{\ell+1}$ and we are done.
\end{proof}

So to complete the proof of Theorem~\ref{thm:totdisc} it remains to establish Lemma~\ref{gammaext}. This will be done by a supermartingale convergence argument and we will need the following elementary result.

\begin{lemma}\label{supermart} Assume $\{M_n:n\in\N\}$ is a $\Z_+$-valued $(\cF_n)$-supermartingale such that for any $k\in\N$ there is a $r_k>0$ so that for all $n\in\N$,
\begin{equation}\label{martcond}\P(M_{n+1}\neq k \, | \, \cF_n)\ge r_k\text{ on }\{M_n=k\}.
\end{equation}
Then $M_n\to 0$ a.s.
\end{lemma}
\begin{proof} By supermartingale convergence we know $M_n\to M_\infty\in\Z_+$ a.s.\ The above condition easily implies for $k\in\N$ and $n\ge N$,
\[\P(\forall j\in[N,n+1], \ M_j=k)\le (1-r_k)\cdot\P(\forall j\in[N,n], \ M_j=k),\]
and therefore by induction, for all $k\in\N$ and $n\ge N$,
\[\P(\forall j\in[N,n], \ M_j=k )\le (1-r_k)^{n-N}.\] 
This proves that for all $k\in \N$, $\P(\cup_{N=1}^\infty\{\forall j\ge N, \ M_j=k \})=0$, and therefore \break
$\P(M_\infty=k)=0$.  We therefore conclude that $M_\infty=0$ a.s.\ and the proof is complete.
\end{proof}

\begin{proof}[Proof of Lemma~\ref{gammaext}]
As noted after the statement of Lemma~\ref{gammaext}, we may assume $f$ and $g$ are $\ell$-neighbours when verifying \eqref{Rfinite}. 
Next we claim that it suffices to prove the result for $R=m$ a constant.
Assume this constant case and let $f,g\in F_R$ be $\ell$-neighbours.  
So on $\{R=m\}$, $f,g\in F_m$ are $\ell$-neighbours. 
The result for $R$ identically $m$ gives w.p.$1$ $R_m^\ell(f,g)<\infty$.  
This gives \eqref{Rfinite} w.p.$1$ on $\{R=m\}$.  
Combine the resulting null sets to conclude that w.p.$1$, \eqref{Rfinite} holds.  

So we clearly may fix $m\in\N$, $f\neq g\in F_m$ and $\ell\in\{0,\dots,d-1\}$, and it suffices to prove that with probability $1$ on 
\[\Omega_\ell=\{\text{$f$, $g$ are $\ell$-neighbours}\}\]
which is in $\cG_m$, we have
\begin{equation}\label{Rmfinite}
R^\ell_m(f,g)<\infty.
\end{equation}
Until otherwise indicated we assume $\omega\in\Omega_\ell$. 
We simplify notation and for $n\ge m$ write $\Gamma_n^{\ell'}(f,g)$ for $\Gamma_{m,n}^{\ell'}(f,g)$. 

We first claim that 
\begin{equation}\label{constL}
\forall n\ge m, \forall (f',g')\in \Gamma_n^\ell(f,g), \quad L(S^{f'},S^{g'})=L(S^f,S^g).
\end{equation}
For this, note that Lemma~\ref{ellnab}\ref{it:L is non-increasing} implies that for all $n\ge m$ and $(f',g')\in K^f_n\times K_n^g\supset \Gamma_n^{\ell}(f,g)$, we have  $L(S^{f'},S^{g'})\subset L(S^f,S^g)$. 
On the other hand  $(f',g')\in\Gamma^\ell_n(f,g)$ implies $|L(S^{f'},S^{g'})|=\ell=|L(S^f,S^g)|$ (recall we work on $\Omega_\ell$), and so \eqref{constL} follows. 

If $n\ge m$ and $x,y\in G_n$, let 
\[L'_n(x,y)=\{1\le k \le d:x^k=y^k+B^{-n}\}.\]
It follows from Lemma~\ref{ellneighb} that ($\dot\cup$ denotes a disjoint union)
\begin{equation}\label{Lcdecomp}
\text{if $f',g'\in F_n$ are neighbours then $L(S^{f'},S^{g'})^c=L'_n(S^{f'},S^{g'})\dot\cup L'_n(S^{g'},S^{f'})$.}
\end{equation}
Together \eqref{constL} and \eqref{Lcdecomp} give us
\begin{equation}\label{unioneq}
L'_n(S^{f'},S^{g'})\cup L'_n(S^{g'},S^{f'})=L'_m(S^f,S^g)\cup L'_m(S^g,S^f)\ \forall (f',g')\in \Gamma_n^\ell(f,g)\ \forall \ n\ge m.
\end{equation}
Next we claim that 
\begin{equation}\label{constL'}
\forall n\ge m, \forall(f',g')\in \Gamma_n^\ell(f,g), \ L'_n(S^{f'},S^{g'})=L'_m(S^f,S^g)\text{ and }L'_n(S^{g'},S^{f'})=L'_m(S^g,S^f).
\end{equation}
For this, argue as in \eqref{Sincr} but without the absolute values, to see that for $(f',g')\in\Gamma^\ell_n(f,g)$ $(n\ge m$) and $k\in L'_m(S^f,S^g)$, 
\begin{align*}
 S^{f',k}-S^{g',k}&=S^{f,k}-S^{g,k}+\sum_{j=m+1}^n B^{-j}(X^{f'|j,k}-X^{g'|j,k})\\
&=B^{-m}+\sum_{j=m+1}^nB^{-j}(X^{f'|j,k}-X^{g'|j,k})\\
&\ge B^{-n}.
\end{align*}
The fact that $f'$ and $g'$ are neighbours implies $|S^{f',k}-S^{g',k}|=0$ or $B^{-n}$ by Lemma~\ref{ellneighb}. So the above implies $S^{f',k}-S^{g',k}=B^{-n}$, and we have proved $L'_m(S^f,S^g)\subset L'_n(S^{f'},S^{g'})$, and therefore $L'_m(S^g,S^f)\subset L'_n(S^{g'},S^{f'})$ by symmetry.  
These inclusions and the equality in \eqref{unioneq} now give \eqref{constL'}.

We claim that the evolution in $n\ge m$ of  $\Gamma^\ell_n(f,g)$ satisfies
\begin{align}\label{gammaevol}
(f',g')\in \Gamma^\ell_{n+1}(f,g)\ \ \ \text{implies }&(\pi f',\pi g')\in\Gamma^\ell_n(f,g),\ f'_{n+1}\le Z^{\pi f'},\ g'_{n+1}\le Z^{\pi g'},\\ \nonumber&\forall k\in L(S^f,S^g), \ X^{f',k}=X^{g',k},\\
\nonumber&\forall k\in L'_m(S^f,S^g), \ X^{g',k}=B-1, X^{f',k}=0,\text{ and}\\
\nonumber&\forall k\in L'_m(S^g,S^f), \ X^{f',k}=B-1, X^{g',k}=0.
\end{align}
(In fact it is also not hard to show the converse implication holds in the above but we will not need this.)  To prove the above, note first that for any $n\ge m$ we clearly have
\begin{equation}\label{Kevol}
(f',g')\in K^f_{n+1}\times K^f_{n+1}\quad \text{iff} \quad (\pi f',\pi g')\in K^f_n\times K^g_n \text{ and } f'_{n+1}\le Z^{\pi f'}\text{ and }g'_{n+1}\le Z^{\pi g'}.
\end{equation}
Let $(f',g')\in \Gamma^\ell_{n+1}(f,g)$ where $n\ge m$. Lemma~\ref{ellnab}\ref{it:L is non-increasing} implies
\[L(S^{f'},S^{g'})\subset L(S^{\pi f'},S^{\pi g'})\subset L(S^f,S^g).\]
The cardinalities of the first and last sets above are both $\ell$ (recall we are on $\Omega_\ell$), and so we conclude
\begin{equation}\label{Lsame}
L(S^{f'},S^{g'})=L(S^{\pi f'},S^{\pi g'})=L(S^f,S^g),
\end{equation}
and
\begin{equation}\label{Lcard}
|L(S^{\pi f'},S^{\pi g'})|=\ell.
\end{equation}
Recall that $S^f\neq S^g$ (they are in fact $\ell$-neighbours on $\Omega_\ell$) and so by Lemma~\ref{ellnab}\ref{it:L is non-increasing}, $S^{\pi f'}\neq S^{\pi g'}$. 
We can therefore apply Lemma~\ref{ellnab}\ref{it:neighbours descend from neighbours} to see that $\pi f'$ and $\pi g'$ are $\ell'$-neighbours for some $\ell'\ge \ell$, and hence are neighbours, because $f'$ and $g'$ are $\ell$-neighbours.  
Now \eqref{Lcard} shows that $\pi f'$ and $\pi g'$ must be $\ell$-neighbours.  
This conclusion and \eqref{Kevol} show that 
\begin{equation}\label{piinG}
(\pi f',\pi g')\in \Gamma_n^\ell(f,g).
\end{equation}
By \eqref{Lsame} for all $k\in L(S^f,S^g)$, $S^{\pi f',k}=S^{\pi g',k}$ and $S^{f',k}=S^{g',k}$, which implies
\begin{equation}\label{equalonL}
\forall k\in L(S^f,S^g), \ X^{f',k}=X^{g',k}.
\end{equation}
By \eqref{piinG} we may apply \eqref{constL'} to $(\pi f',\pi g')$, as well as $(f',g')$, and deduce
\begin{equation*}
L'_n(S^{\pi f'},S^{\pi g'})=L'_m(S^f,S^g)\quad \text{and} \quad L'_{n+1}(S^{f'},S^{g'})=L'_m(S^f,S^g).
\end{equation*}
By definition this means
\[\forall k\in L_m'(S^f,S^g), \quad S^{\pi f',k}=S^{\pi g',k}+B^{-n}\text{ and }S^{f',k}=S^{g',k}+B^{-n-1}.
\]
Take differences in the above to conclude that $X^{f',k}=X^{g',k}+1-B$, and therefore we obtain
\begin{equation}\label{L'cond2}
\forall k\in L'_m(S^f,S^g), \ X^{g',k}=B-1\text{ and }X^{f',k}=0.
\end{equation}
Reversing the roles of $f$ and $g$ gives
\begin{equation}\label{L'cond3}
\forall k\in L'_m(S^g,S^f), \ X^{f',k}=B-1\text{ and }X^{g',k}=0.
\end{equation}
Now combine \eqref{Kevol}, \eqref{piinG}, \eqref{equalonL}, \eqref{L'cond2} and \eqref{L'cond3} to complete the proof of \eqref{gammaevol}. 

Recall from \eqref{veenotation} the notation $f'\vee k\in F_{n+1}$ if $f'\in F_n$ and $k\in\N$. For $n\ge m$ define $M_{n}=1_{\Omega_\ell}|\Gamma_n^\ell(f,g)|$. Clearly $M_n$ is $\cG_n$-measurable. By \eqref{gammaevol} for $n\ge m$,
\begin{align}\label{Mnbnd}
M_{n+1}\le 1_{\Omega_\ell}\sum_{(f',g')\in \Gamma_n^{\ell}(f,g)}\sum_{i=1}^{Z^{f'}}\sum_{j=1}^{Z^{g'}}&1(\forall k\in L(S^f,S^g), \ X^{f'\vee i,k}=X^{g'\vee j,k})\\
\nonumber&\times 1(\forall k\in L_m'(S^f,S^g), \ X^{g'\vee j,k}=B-1,\,X^{f'\vee i,k}=0)\\
\nonumber&\times 1(\forall k\in L_m'(S^g,S^f), \ X^{f'\vee i,k}=B-1,\,X^{g'\vee j,k}=0).
\end{align}
Let $\bar\cG_n=\cG_n\vee\sigma(Z^{f'},f'\in F_n)$. 
Condition first on $\bar \cG_n$ and use the independence properties of $(X^f,f\in F)$ and $(Z^f,f\in F)$ to see that 
\begin{align*}
\E(M_{n+1}\, | \, \cG_n)&\le 1_{\Omega_\ell}\sum_{(f',g')\in \Gamma^\ell_n(f,g)}\E\Bigl(\sum_{i=1}^{Z^{f'}}\sum_{j=1}^{Z^{g'}}B^{-\ell}B^{-2(|L_m'(S^f,S^g)|+|L_m'(S^g,S^f)|)} \ \Bigl| \ \cG_n\Bigr)\\
&=M_n\mu^2B^{-\ell}B^{-2(d-\ell)}\quad\text{(by \eqref{Lcdecomp} and $|L(S^f,S^g)^c|=d-\ell$)}\\
&\le M_nB^{4/\beta}B^{-1-d}\le M_n,
\end{align*}
where we have used $\ell\le d-1$, $\mu=B^{2/\beta}$ and $d\ge 4/\beta-1$ in the last line. Therefore $(M_n)$ is a $\Z_+$-valued $(\cG_n)$-supermartingale.  

To prove that \eqref{Rmfinite} holds a.s.\ on $\Omega_\ell$ it clearly suffices to show $M_n\to 0$ a.s.\ and for this we will verify the hypothesis \eqref{martcond} of Lemma~\ref{supermart}. 
On $\Omega_\ell$ we have by \eqref{Lcdecomp} that $0<|L(S^f,S^g)^c|=|L'_m(S^f,S^g)|+|L'_m(S^g,S^f)|$. 
So one of $L_m'(S^f,S^g)$ or $L'_m(S^g,S^f)$ is nonempty and we may assume without loss of generality it is the former and set $k_m=\min L'_m(S^f,S^g)$, which is $\cG_m$-measurable.  
Let $\gamma_{n}^\ell(f,g)\subset K^f_n$ be the projection of $\Gamma_n^\ell(f,g)$ onto the first variable. 
By \eqref{Mnbnd} on $\{L'_m(S^f,S^g)\neq \emptyset\}\cap\Omega_\ell \,\in\cG_m$, 
\begin{align}\label{lbnd0}
\nonumber\P(M_{n+1}=0\, | \, \cG_n)&\ge\P\Bigl(\cap_{(f',g')\in \Gamma_n^\ell(f,g)}\cap_{i=1}^{Z^{f'}}\cap_{j=1}^{Z^{g'}}\{\exists k\in L'_m(S^f,S^g)\\
\nonumber&\phantom{\ge\P\Bigl(\cap_{(f',g')\in \Gamma_n^\ell(f,g)}\cap_{i=1}^{Z^{f'}}\cap_{j=1}^{Z^{g'}}} s.t.\ (X^{f'\vee i,k},X^{g'\vee j,k})\neq(0,B-1)\}\ \Bigl| \ \cG_n\Bigr)\\
&\ge \P\Bigl(\cap_{f'\in \gamma_{n}^\ell(f,g)}\cap_{i=1}^{Z^{f'}}\{X^{f'\vee i,k_m}\neq 0\}\ \Bigl| \ \cG_n\Bigr).
\end{align}
Conditional on $\bar\cG_n$, $\{1(X^{f'\vee i,k_m}\neq 0):f'\in\gamma^\ell_{n}(f,g),i\le Z^{f'}\}$ are i.i.d.\ $\mathrm{Bernoulli}(p)$ r.v.'s with $p=1-B^{-1}$, while conditional on $\cG_n$, $\{Z^{f'}:f'\in\gamma^\ell_{n}(f,g)\}$ are i.i.d.\ copies of the branching random variable $Z$.  
So condition first on $\bar\cG_n$ and then on $\cG_n$ to see that the right-hand side of \eqref{lbnd0} is (on $\Omega_\ell$)
\[\E\Bigl(\prod_{f'\in\gamma^\ell_{n}(f,g)}(1-B^{-1})^{Z^{f'}}\ \Bigl| \ \cG_n\Bigr)=\E\Bigl((1-B^{-1})^Z\Bigr)^{|\gamma^\ell_{n}(f,g)|}\ge \E((1-B^{-1})^Z)^{M_n},\]
where the last inequality holds because $|\gamma^\ell_{n}(f,g)|\le |\Gamma^\ell_n(f,g)|=M_n$ on $\Omega_\ell$. 
Recall that $M_{n+1}=0$ on $\Omega_\ell^c$ and so for any $k\in\N$, on $\{M_n=k\}$,
\[\P(M_{n+1}=0 \, | \, \cG_n)\ge \E((1-B^{-1})^Z)^{k}=:r_k>0.\]
This proves \eqref{martcond} and the proof is complete.
\end{proof}

\section{Weak disconnectedness result for subcritical $B$-ary Classical Super-Tree Random Measures}
\label{sec:wdisc}
We continue to work in the setting of $B$-ary classical super-tree random measures. Recall from Section~\ref{sec:STRM} the definition of the random measure $Y$ on $I$, which has total mass given by the random variable $W^0$, see \eqref{Ydef1} and \eqref{Ymass}.
On the event $\{W^0>0\}$, we denote by $V = (V_1,V_2,V_3,\dots )$ a random element taken under the probability measure $\frac{1}{W^0}\cdot Y$, ``conditionally on everything else". Formally one can work on the product space $(\Omega\times I, \cF\times \cB(I))$, where for $(\omega,V)\in\{W^0>0\}\times I$, $\P(V\in A \, | \, \omega)=Y(\omega)(A)/W^0$ and $V=(1,1,\dots)$, say, on $\{W^0=0\}$. 
An explicit description of the joint law of $(\omega,V)$ under this Campbell measure is given below. 
By definition of $\nu$ as the push-forward of $Y$ by the function $i\mapsto S^i$, conditionally on $\nu$ and on the event $\{\nu(1)>0\}$, the point $S:=S^V$ has distribution $\frac{1}{\nu(1)}\cdot \nu$. 
In this section we show that under the assumption that the process is ``subcritical", i.e.,\ $\frac{2}{\beta} <d$, the connected component of $S$ is almost surely reduced to a point.
\begin{theorem}\label{thm:weak disconnectedness}
If $\frac{2}{\beta} <d$, on the event $\{W^0> 0\}$ the connected component of $S$ is almost surely reduced to a point, i.e.
\begin{align*}
	\P\left( \{S\} \text{ is a connected component of }\supp(\nu)\, | \, \nu(1)>0\right) = 1.
\end{align*}
\end{theorem}
\begin{remark}\label{rem:Wdisc}
	\begin{enumerate*}[label=\emph{(\alph*)}, ref=(\alph*),leftmargin=*]
	\item 
	The above result is a restatement of Theorem~\ref{thm:nutribe}.
	Note that the above conclusion is consistent with the existence of a non-trivial connected component for $\supp(\nu)$. Interesting examples where such components exist and the result above is true are given in Remark~\ref{rem:weak disconnectedness result also holds for fp} in the next section and Remark~\ref{rem:ccomp} in the Introduction. \\
	\item Theorem~\ref{thm:weak disconnectedness}  also ensures that almost surely, any connected component of $\supp(\nu)$ has vanishing $\nu$-mass because $\nu$ is atomless by \eqref{finenergy}.
	\end{enumerate*}
\end{remark}

The proof of Theorem~\ref{thm:weak disconnectedness} uses a different probability measure $\Q$, given by 
\begin{align*}
	\frac{\mathrm{d} \Q}{\mathrm{d}\P} = W^0.
\end{align*}
In Theorem~\ref{thm:weak disconnectedness} we are working on the event where the density $W^0$ of $\Q$ relative to $\P$ is non-zero, and so 
we just need to prove that under $\Q$, the singleton $\{S\}$ is almost surely a connected component of $\supp  \ \nu$. 
As we now show, the law of all our random variables has a simple description under $\Q$. 

Recall $(p_k)_{k\geq 0}$ is the reproduction distribution and recall that $Z$ denotes a random variable with this distribution. 
We introduce $Z^*$ the \emph{size-biased version} of $Z$, whose distribution is characterized by the relation
\begin{align*}
	\E\left[h(Z^*)\right] = \frac{\E \left[Z h(Z)\right]}{\E[Z]},
\end{align*} 
for all bounded functions $h$. This amounts to saying that $\P(Z^*=k)=\frac{1}{\mu} k p_k$ for all $k\geq 0$.

Using standard arguments in the literature on branching processes (see, for example, Section~2 of \cite{LPP95}),
we can explicitly describe the law of the random variables $V=(V_1,V_2,\dots)$, $(Z^f, \ f\in F)$ and $(X^f, \ f\in F\setminus\{0\})$ under $\Q$: 
\begin{equation}\label{qdist1}
	\parbox[t]{0.9\textwidth}{The sequence $((V_{k+1},Z^{V|k}), \ k\geq 0)$ is i.i.d.\  with the law of $(N,Z^*)$, where $Z^*$ is
		the size-biased version of $Z$, and conditionally on $Z^*$, the random variable $N$ is uniform in $\{1,2,\dots Z^*\}$.}
\end{equation}
\begin{equation}\label{qdist2}
		\parbox[t]{0.9\textwidth}{ Conditionally on $((V_{k+1},Z^{V|k}), \ k\geq 0)$, the random variables $(Z^f, \ f\in F\setminus \{V|m \ : \ m\geq 0\})$ and $(X^f, \ f\in F\setminus\{0\})$ have the same joint distribution as under $\P$.}
\end{equation}
At any generation $m$ we refer to $V|m$ as the special individual (or particle).
In what follows, we will consider the filtration $(\cH_n)_{n\geq 0}$ defined as
\begin{align*}
	\cH_n:= \sigma(V_k:k \leq n)\vee \sigma(Z^f:|f|<n) \vee  \sigma(X^f:1\le |f|\leq n).
\end{align*}
To simplify notation, we denote by $S_m$ the position $S^{V|m}$ of the special particle at time $m$. Recall from Section~\ref{sec:Bary} that $J_m^x=\{f\in K_m^0:S^f=x\}$ for $x\in G_m$ and $m\in\Z_+$, so that $J_m^{S_m}=\{f\in K^0_m:S^f=S_m\}$ is the set of particles present at time $m$ at position $S_m$. 
Note that the process $(S_m)_{m\geq 0}$ is $(\cH_m)_{m\geq 0}$-adapted, as is the process $(J_m^{S_m})_{m\geq 0}$.
The proof of Theorem~\ref{thm:weak disconnectedness} will follow the immediate surroundings of the position $S_m$ of the distinguished particle $V|m$ at every time $m$. The idea will be to show that infinitely often as $m$ grows, the position $S_m$ will be surrounded by empty sites, thus disconnecting the descendants of $S_m$ from the the rest of the population.

We start by analyzing the number of particles at position $S_m$. 
We describe below the resulting process as a Galton-Watson process \emph{with immigration}: 
from one generation to the next, every particle is replaced by an independent number of offspring sampled from the reproduction law and an additional number of immigrant particles, sampled from the \emph{immigration} law, is added, independently of the reproduction, see for example \cite{Hea65}. 
As in the proof of Proposition~\ref{hitprob1}, we write $R=\sum_{\ell =1}^Z e_\ell$, where $\{e_\ell\}$ are i.i.d.\ Bernoulli r.v.'s with parameter $p=B^{-d}$, independent of $Z$. 
We also write $\tilde{R}=\sum_{\ell =1}^{Z^*-1} e_\ell$, where $Z^*$ is the size-biased version of $Z$ and again $\{e_\ell\}$ are i.i.d.\ Bernoulli r.v.'s with parameter $p=B^{-d}$, independent of $Z^*$.
\begin{lemma}\label{lem:brimm}
	Under $\Q$, the process $\left(|J_m^{S_m}|-1\right)_{m\geq 0}$ is an $(\cH_m)_{m\geq 0}$-adapted Galton-Watson process with immigration starting at $0$ at $m=0$. 
	The reproduction law is that  of $R$ and the immigration law is that of $\tilde{R}$. 
	This process is recurrent if $\frac{2}{\beta} < d$.
\end{lemma}
\begin{proof}
Let $m\geq 1$. We decompose $J^{S_m}_m$ into three subsets: one containing the special individual, one containing the siblings of the special individual, and one containing the rest of the particles. 
In order to keep the same notation as \eqref{Jdefn} we write $X^{V|m}= x_m$ for the displacement of the special particle. By \eqref{Jdefn} we have
\begin{align*}
J^{S_m}_m&=\bigcup_{f\in J_{m-1}^{S_{m-1}}}\bigcup_{g_m=1}^{Z^f}\{f\vee g_m:\, X^{f\vee g_m}=x_m\}\\
&=\{V|m\} \cup \bigcup_{\substack{g_m=1\\ g_m\neq V_m}}^{Z^{V|m-1}} \{(V|m-1)\vee g_m:\, X^{(V|m-1)\vee g_m}=x_m\}\\
&\phantom{=\{V|m\}\ } \cup \bigcup_{\substack {f\in J_{m-1}^{S_{m-1}}\\ f \neq(V|m-1)}}\bigcup_{g_m=1}^{Z^f}\{f\vee g_m:\, X^{f\vee g_m}=x_m \}.
\end{align*}
Hence
	\begin{align*}
		|J^{S_m}_m|= 1 + \sum_{\substack{\ell=1\\ \ell\neq V_m}}^{Z^{V|m-1}} 1(X^{(V|m-1)\vee\ell}=x_m) + \sum_{\substack {f\in J_{m-1}^{S_{m-1}}\\ f \neq(V|m-1)}}\sum_{\ell=1}^{Z^f}1(X^{f\vee\ell}=x_m).
	\end{align*}
By \eqref{qdist1} and \eqref{qdist2}, under $\Q$, conditionally on $\cH_{m-1}$, the number of non-special children $Z^{V|m-1}-1$ of the special particle in generation $m-1$ is distributed as $Z^*-1$, and the number of children of any non-special particle in $J_{m-1}^{S_{m-1}}$ is distributed as the law of $Z$, and these variables are all independent.
Conditionally on these numbers and $\cH_{m-1}$, the spatial displacements of all children of particles in $J_{m-1}^{S_{m-1}}$ are i.i.d.\ and each such non-special child has probability $B^{-d}$ to make the same displacement as the special child.
This proves the first part of the statement.

The main theorem in \cite{Hea66} implies that sufficient conditions for recurrence of the above Galton-Watson process are $\E\left[R\right]<1$ and $\E[\tilde{R}]< \infty$ (it is easy to check irreducibility and aperiodicity here).   
The first condition holds because 
 $\E\left[R\right] = \mu B^{-d} = B^{\frac{2}{\beta} - d} <1$ and the second condition holds thanks to assumption \eqref{pk}. 
\end{proof}

\begin{remark}
In the critical setting, i.e.\ if $\frac{2}{\beta} = d$, we have $\E[R]=1$ and an easy computation gives us $\E[R(R-1)]= B^{-2d} \E[Z(Z-1)]=\E[\tilde{R}]$. 
This ensures, using  the result \cite[Theorem~1]{Pak71} on the recurrence and transience of critical Galton-Watson processes with immigration, that our process is transient: our proof in the sub-critical case would break down in the critical case.
\end{remark}
\begin{figure}
	\centering
	\includegraphics[page=8]{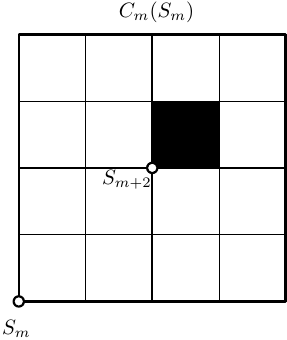}
	\caption{On the event $E_m$, the special particle is alone at $S_m$ and at time $m+2$ all its descendants are at position $S_m+B^{-(m+1)} \vec{1}$. This ensures that all the cubes surrounding $C_{m+2}(S_{m+2})$ are going to remain empty.
	\label{fig:special particle is surrounded}}
\end{figure}
Now, for $m\geq 0$, consider the following event 
\begin{align*}
	E_m:= \{|J_m^{S_m}| = 1 \text{ and }\forall f \in K_{m+2}^{V|m}, \ S^f= S_{m+2}=S_m + B^{-(m+1)} \vec{1} + B^{-(m+2)} \vec{0} \},
\end{align*}
that is, the event where at time $m$, the special particle is alone at site $S_m$, all its children have displacement $B^{-(m+1)} \vec{1}$ and then all of its grand-children have displacement $B^{-(m+2)} \vec{0}$,  i.e.\ no displacement. Note that $E_m$ is $\cH_{m+2}$-measurable.
\begin{lemma}\label{lem:sequence of events happens io}
The sequence of events $(E_m)_{m\geq 0}$ occurs infinitely often $\Q$-a.s.
\end{lemma}
\begin{proof}
First, note that on the event $\{|J_m^{S_m}| = 1\}$ we have 
\begin{align*}
	\Q\left(E_m \, | \, \cH_m\right) = \Q\left( \forall f \in K_{2}^{0}, \ S^f= S_{2}=B^{-1} \vec{1} + B^{-2} \vec{0} \right) >0
\end{align*}	
which does not depend on $m$.
By the recurrence of $\left(|J_m^{S_m}|-1\right)_{m\geq 0}$ and the above display, we know that almost surely the sum $\sum_{m=1}^\infty \Q\left( E_m \ \vert \ \cH_m\right)$ is infinite. 
Separating between even and odd terms, we get that almost surely, one of the two sums $\sum_{k=1}^\infty \Q\left( E_{2k} \ \vert \ \cH_{2k}\right)$ or $\sum_{k=1}^\infty \Q\left( E_{2k+1} \ \vert \ \cH_{2k+1}\right)$, is infinite.
Using the second Borel-Cantelli lemma (see for example \cite[Theorem~5.3.2]{Du10}), we get that on the event $\{\sum_{k=1}^\infty \Q\left( E_{2k} \ \vert \ \cH_{2k}\right)=\infty\}$ the sequence of event $(E_{2k})_{k\geq 1}$ a.s.\ occurs infinitely often, and similarly for $(E_{2k+1})_{k\geq 1}$ on $\{\sum_{k=1}^\infty \Q\left( E_{2k+1} \ \vert \ \cH_{2k+1}\right)=\infty\}$.
Combining this with the above, the sequence of events $(E_m)_{m\geq 0}$ occurs infinitely often, $\Q$-almost surely.
\end{proof}

\begin{proof}[Proof of Theorem~\ref{thm:weak disconnectedness}]
Recall that it suffices to prove  that the connected component of the point $S=\lim_{m\to\infty} S_m$ is reduced to a point almost surely under $\Q$.
Note that on $E_m$, the set of cubes in $\Lambda_{m+2}$ surrounding $\overline{C_{m+2}(S_{m+2})}$ are empty (see Figure~\ref{fig:special particle is surrounded}).
This ensures that the connected component of $S$ is included in the cube $\overline{C_{m+2}(S_{m+2})}$. 
If this occurs infinitely often as $m\rightarrow \infty$, this means that the connected component of $S$ is contained in a sequence of sets whose diameter tends to $0$ so it is reduced to a point.
We therefore conclude from Lemma~\ref{lem:sequence of events happens io} that the component of $S$ is $\Q$-a.s.\ reduced to a point, as required. 
\end{proof}

\section{Percolation on $B$-ary Classical Super-Tree Random Measures}\label{sec:percn}

We work in the context of a $B$-ary CSTRM associated with $(p_k)$. Our goal is to find sufficient conditions under which $\supp(\nu)$ will contain a non-trivial connected set, that is, $\supp(\nu)$ is not TD (totally disconnected).  We begin with a $0-1$ type law for this event. Recall the definition of $\supp(\nu)$ percolating prior to Theorem~\ref{thm:perc2}.
\begin{lemma} \label{lem:01}If $\P(\supp(\nu)\,\text{percolates})=p>0$, then 
$\P(\supp(\nu)\text{ is not TD}\,|\,\nu\neq 0)=1$.
\end{lemma}
\noindent We will need some terminology for the proof.

\begin{definition}  $T:I\to I$ is the shift map $Ti=(i_2,i_3,\dots)$, and we abuse this notation by also using $T$ for the map from $F_n$ to $F_{n-1}$ given by $Tf=(f_2,\dots,f_n)$ ($Tf=0$ if $n=1$).  Let $m\in\N$ and $f\in F_m$. For $g\in F$, extend the notation in \eqref{veenotation} and let $f\vee g\in F_{m+|g|}$ be the concatenation of $f$ followed by $g$.  Also let $Z^g_{[f]}=Z^{f\vee g}$ and (for $|g|>0$) $X^g_{[f]}=X^{f\vee g}$ be the $f$-shifted analogues of our original branching and migration r.v.'s.  Define measures $Y_{[f]}$ on $(I,\cB(I))$ and $\nu_{[f]}$ on $([0,1]^d,\cB([0,1]^d)$ just as $Y$ and $\nu$, respectively, but now using the collections $\{Z_{[f]}^g:g\in F\}$ and $\{X_{[f]}^g:g\in F\setminus F_0\}$ in place of $\{Z^g:g\in F\}$ and $\{X^g:g\in F\setminus F_0\}$, respectively.  
Finally define $Y^f$ on $(I,\cB(I))$ by
\begin{equation*}
	Y^f(A)=\int_{\overline D(f)}1(T^mi\in A)\dd Y(i),
\end{equation*}
where $T^m$ denotes $m$-fold composition.
\end{definition}

Recalling $\cG_m$ from \eqref{newGdef}, we see that
\begin{equation}\label{nufiid} \text{$\{\nu_{[f]}:f\in F_m\}$ are i.i.d.\ copies of $\nu$, and are jointly independent of $\cG_m$.} 
\end{equation}

\begin{lemma}\label{Yfequiv}For any $f\in F_m$, we have  $Y^f=1(f\in K^0_m)\mu^{-m}Y_{[f]}$.
\end{lemma}
\begin{proof}
As in the uniqueness proof of Proposition~\ref{Y}, it suffices to fix $n\in\Z_+$ and $g\in F_n$, and show that
\[Y^f(\overline D(g))=1(f\in K^0_m)\mu^{-m}Y_{[f]}(\overline D(g)).\]
\SUBMIT{This is a straightforward calculation left for the interested reader.}
\ARXIV{A complete proof can be found in Section~\ref{app:subsec:Yfequiv} of the Appendix.}
\end{proof}

\begin{proof}[Proof of Lemma~\ref{lem:01}]
The idea of the proof is as follows: on the event $\{\nu \neq 0\}$, the number of particles living at time $n$ tends almost surely to infinity as $n\rightarrow \infty$, so if we stop the process at any large time $m\in \N$, the $m$-th generation of particles is made of large number of individuals.  
Since they each evolve independently of each other, and each of them has a constant probability $p$ of creating a non-trivial connected component (by percolating its corresponding cube of length $B^{-m}$), this should ensure that $\supp(\nu)$ contains non-trivial components. 
Note here that a non-trivial connected component created by single particle is contained in a connected component of $\supp(\nu)$.
Let us make this line of reasoning precise.
 Let $f\in F_m$.  Define $\nu^f$ on $([0,1]^d,\cB([0,1]^d)$ by
\begin{equation*}
	 \nu^f(A)=\int_{\overline D(f)}1_A(S^i)\,\dd Y(i).
\end{equation*}
By \eqref{Ydef1}, $Y(\overline D(f))=0$ if $f\notin K^0_m$, and so, since $\{\overline D(f):f\in F_m\}$ partitions $I$, we have
\begin{equation}\label{nudecompn}
\nu=\sum_{f\in K^0_m}\nu^f.
\end{equation}
It follows from its definition that 
\begin{equation}\label{nufrepn}
\nu^f(A)=\int_{\overline D(f)}1_A(S^{f\vee T^mi})\,\dd Y(i)=\int 1_A(S^{f\vee j})\,\dd Y^f(j).
\end{equation}
For $j\in I$, define
\begin{equation}\label{Sfdefn2}
S^j_{[f]}:=\sum_{k=1}^\infty B^{-k}X_{[f]}^{j|k}=\sum_{k=1}^\infty B^{-k}X^{f\vee(j|k)},
\end{equation}
so that we have
\begin{equation}\label{nufrep}
\nu_{[f]}(A)=\int 1(S^j_{[f]}\in A)\, \dd Y_{[f]}(j).
\end{equation}
For $j\in I$ \eqref{Sfdefn2} implies
\begin{equation*}
S^{f\vee j}=\sum_{n=1}^\infty B^{-n}X^{(f\vee j)|n}=S^f+B^{-m}S^j_{[f]}.
\end{equation*}
Use this in \eqref{nufrepn} and apply Lemma~\ref{Yfequiv} to conclude that
\begin{align*}\nu^f(A)&=\int 1_A(S^f+B^{-m}S^j_{[f]})\mu^{-m}\,\dd Y_{[f]}(j)\, 1(f\in K^0_m)\\
&=\mu^{-m}1(f\in K^0_m)\int 1_A(S^f+B^{-m}x)\,\dd \nu_{[f]}(x)\quad\text{(by \eqref{nufrep})}.
\end{align*}
Let $\tau^f(x)=S^f+B^{-m}x$ and let $(\tau^f)^*(\nu_{[f]})(A)=\nu_{[f]}((\tau^f)^{-1}(A))$ denote the pushforward of $\nu_{[f]}$ by $\tau^f$. Then by the above and \eqref{nudecompn} we have
\begin{equation}\label{nudecomp2}
\nu=\sum_{f\in K^0_m}\mu^{-m}(\tau^f)^*(\nu_{[f]}).
\end{equation}
Clearly $\supp(\mu^{-m}(\tau^f)^*(\nu_{[f]}))=\tau^f(\supp(\nu_{[f]}))$ and so 
\begin{equation}\label{TPDeq}\supp(\mu^{-m}(\tau^f)^*(\nu_{[f]}))\text{ is TD implies }\supp(\nu_{[f]})\text{ does not percolate}.
\end{equation}
It follows from \eqref{nufiid}, \eqref{nudecomp2} and \eqref{TPDeq} that (recall $p$ is as in the statement of the Lemma)
\begin{align}\label{TDub}
\nonumber\P(\supp(\nu)\text{ is TD} \, | \, \cG_m)&\le \P(\cap_{f\in K^0_m}\{\supp(\nu_{[f]}) \text{ does not percolate}\} \, | \, \cG_m)\\
&=(1-p)^{|K^0_m|}.
\end{align}
Therefore 
\[\P(\supp(\nu)\text{ is TD})\le \E((1-p)^{|K^0_m|}).\]
Recall from Proposition~\ref{Y} that either $|K^0_m|=0$ for large $m$ (iff $\nu=0$) or $\lim_m|K^0_m|=\infty$. Therefore
\begin{align*}
\P(\supp(\nu)\text{ is TD})\le \lim_{m\to\infty}\E((1-p)^{|K^0_m|})=\P(|K^0_m|=0\text{ for large }m)=\P(\nu=0).
\end{align*}
On the other hand $\{\nu=0\}\subset \{\supp(\nu)\text{ is TD}\}$. Therefore these two sets are a.s.\ equal and the result follows. 
\end{proof}
\begin{notation} For $m\in\Z_+$ let $\Lambda_m=\{\overline{C_m(x)}:x\in G_m\}$.
\end{notation}
To study the percolation properties of $\supp(\nu)$ we will set up a coupling with a random Cantor set arising from fractal percolation and then use the corresponding properties of the latter. We recount some features of fractal percolation from \cite{CCD88} and \cite{FG92}. Recall from Section~\ref{sec:Bary} (prior to Proposition~\ref{hitprob1} and Lemma~\ref{Cantor}, respectively) the notation $.x_1\dots x_m\in G_m$ for $x_\ell\in\{0,\dots,B-1\}^d$ and $C_m(x)$ for $x\in G_m$. 

Let $B\in \N^{\ge 2}$, $p\in[0,1]$ and $\{e_{x,m}: x\in G_m, \ m\geq 1\}$ be i.i.d.\ Bernoulli r.v.'s with parameter $p$.  
Fractal percolation builds a random Cantor set, $A_\infty$ by setting $A_0=[0,1]^d$ at stage~$0$ and at the $m$th stage divides each of the remaining cubes in $\Lambda_m$ into its non-overlapping subcubes in $\Lambda_{m+1}$ and keeps each subcube with probability $p$, independently of the other cubes and the past selections.  
$A_\infty$ is the residual limiting set obtained by letting $m\to\infty$. More formally let $A_0=[0,1]^d$ and for $m\in\N$ set 
\begin{equation*}
	A_m:=A_{m-1}\cap \bigcup_{\substack{x \in G_m\\ e_{x,m}=1}} \overline{C_m(x)}.
\end{equation*}
Clearly $A_m$ is decreasing in $m$. We call $A_\infty:=\cap_m A_m$ the fp-Cantor set with parameters $(B,p)$. Although our labeling is slightly different, the definition is clearly equivalent to that, say, prior to Theorem~2 in \cite{FG92}. 

Note that from this definition, for any $m\geq 1$ and any $x=.x_1\dots x_m\in G_m$, a cube $\overline{C_m(x)}$ is included in $A_m$ if and only if all the cubes $\overline{C_k(.x_1\dots x_k)}$ for $1\leq k\leq m$ have been kept during the process i.e.\ 
\begin{align}\label{Iande}
	I_m^x:=1(\overline{C_m(x)}\subset A_m) = \prod_{k=1}^{m} e_{.x_1\dots x_k, k}.
\end{align}
We can then write $A_\infty$ as 
\begin{align}\label{eq:Ainfty as intersection}
	A_\infty = \bigcap_{m\geq 1} \bigcup_{\substack{x \in G_m\\ I_m^x>0}} \overline{C_m(x)}. 
\end{align}

A simple branching process argument (see p. 308 of \cite{CCD88}) shows 
\begin{equation*}
	\P(A_\infty\neq\emptyset)>0\text{ iff }p>B^{-d}.
\end{equation*}
A number of authors have studied $\Theta(B,p)=\P(A_\infty\text{ percolates})$ and $p_c(B,d)=\sup\{p:\Theta(B,p)=0\}$.  \cite{CCD88} shows that $0<p_c(B,d)<1$ for any $d\ge 2$. Although finding $p_c$ appears to be hard in general, quite a bit is known about the large $B$ asymptotics. To describe these results let $\cL^d$ denote the graph with vertex set $\Z^d$ and an edge between $x$ and $y$ iff $|x_i-y_i|\le 1$ for all $i\le d$, $x_i=y_i$ for some $i\le d$, and $x\neq y$.  So each point now has $3^d-2^d-1$ neighbours rather than $2d$ neighbours as for the standard cubic lattice which we denote by $\Z^d$.  In two dimensions the graphs coincide and so the critical probabilities for site percolation on these graphs are equal. 
For $d>2$, $0<p_c(\cL^d)<p_c(\Z^d)<1/2$ (see the references in Section~1 of \cite{FG92} for the middle inequality, \cite{CP85} for the last inequality, and Theorem~2 of \cite{GS} for the first).

\begin{theorem}\label{thm:limit pc for fp} 
	\begin{enumerate}[label=\emph{(\alph*)}, ref=(\alph*), leftmargin=*]	
	\item \label{it:pc for fp is larger than pc} For all $d\ge 2$ and all $B\in\N^{\ge 2}$, $p_c(B,d)\ge p_c(\cL^d)$.
	\item\label{it:pc for fp converges as B goes to infinity} For all $d\ge 2$, $\lim_{B\to\infty} p_c(B,d)=p_c(\cL^d)$.
\end{enumerate}
\end{theorem}
\noindent This is contained in Theorem~2 of \cite{FG92}.  \ref{it:pc for fp converges as B goes to infinity} was first proved for $d=2$ in \cite{CC89}.

\medskip

Since we want to compare the two models, let us first re-express our construction of $\supp (\nu)$ in the $B$-ary CSTRM model in a way that resembles the definition of the set $A_\infty$ in fractal percolation. 
Consider a $B$-ary CSTRM model.
If $m\in\Z_+$ and $x\in G_m$, define
\begin{equation}\label{Nmxdef}
	N_m^x:=|J_m^x|,\text{ where we recall that }J_m^x=\{f\in K^0_m:S^f=x\}.
\end{equation}
Hence $N_m^x$ is the number of particles at $x$ in generation $m$ (so in particular, $J_0^0=\{0\}$ and so $N_0^0=1$).
Note that $N_m^x$ is clearly $\cG_m$-measurable, where $\cG_m$ is defined in \eqref{newGdef}.  
For $m\in\Z_+$ our usual approximating measures are carried on $G_m$ and are characterized by 
\[
\forall x\in G_m, \quad \nu_m(\{x\})=\mu^{-m}\sum_{f\in K_m^0}1(S^f=x)=\mu^{-m} N_m^x,
\]
and therefore
\begin{equation}\label{numdef2}
	\nu_m=\mu^{-m}\sum_{x\in G_m} N_m^x \delta_{x}= \mu^{-m}\sum_{\substack{x\in G_m \\ N_m^x>0}} N_m^x \delta_{x}.
\end{equation}
One easily checks that $\{S^f:f\in K^0_m\}=\{x \in G_m: N_m^x>0\}$
and so by Lemma~\ref{Cantor},
\begin{equation}\label{suppNalpha}
	\supp(\nu)=\bigcap_{m=1}^\infty\bigcup_{f\in K^0_m} \overline{C_m(S^f)}
	=\bigcap_{m=1}^\infty\bigcup_{\substack{x\in G_m \\ N_m^x>0}}\overline{C_m(x)}\quad\text{a.s.}
\end{equation}
This closely resembles the definition of $A_\infty$ in \eqref{eq:Ainfty as intersection}. 
This way of viewing the construction of $\supp(\nu)$ makes it analogous to a version of fractal percolation where, instead of having any cube $\overline{C_m(x)}$ either present or absent at time $m$, it can be present with some multiplicity, corresponding to the number of particles lying at $x$ in the process. 
The evolution in time of the number of particles at each site is described in the following lemma.
\begin{lemma}\label{lem:evolution number of particles}
The law of the process $((N_m^x, \ x\in G_m))_{m\geq 0}$ can be described as follows:
\begin{itemize}
	\item $N_0^0=1$
	\item For any $m\geq 1$, conditionally on $((N_k^y, \ y\in G_k))_{0 \leq k \leq m-1}$, the random vector $(N_m^x, \ x\in G_m)$ is distributed in such a way that for all $x=.x_1\dots x_{m}\in G_{m}$ 
		\begin{align}\label{eq:distribution number of particles}
		 N_m^x=\sum_{i=1}^{N_{m-1}^{.x_1\dots x_{m-1}}}\sum_{j=1}^{Z_i^{.x_1\dots x_{m-1},m-1}} 1(u_{i,j}^{.x_1\dots x_{m-1},m-1}= x_m),
		\end{align} 
	where the $(Z_i^{y,m-1} : \ y\in G_{m-1},\ i,m\in\N)$ are i.i.d.\ with the same law as $Z$ and the $(u_{i,j}^{y,m-1} :\ y\in G_{m-1},\ i,j,m\in\N)$ are i.i.d\ uniformly chosen over $\{0,\dots,B-1\}^d$. 
\end{itemize}  
\end{lemma}
\begin{proof}
This follows easily from the equality \eqref{eq:evolution Jmx}. We leave the details to the reader.
\end{proof}
To achieve the coupling mentioned above we will assume that for some $c\in(B^{-d},\infty)$,
\begin{equation}\label{Poisson}
	\text{the branching distribution, $\cL(Z)$, is Poisson with mean $\mu=cB^d$.}
\end{equation}
For $Z$ as in the above we will write $Z\sim \mathrm{Poi}(\mu)$.
A reason to assume the branching has a Poisson law is that the description of the distributions arising in Lemma~\ref{lem:evolution number of particles} becomes particularly simple: in \eqref{eq:distribution number of particles}, if $Z\sim \mathrm{Poi}(\mu)$ then, conditional on $N_{m-1}^{.x_1\dots x_{m-1}}$,  the summing and thinning properties of the Poisson distribution ensure that the $(N_m^{x}  : x=.x_1\dots x_{m}\in G_m)$ are independent with respective distribution $\mathrm{Poi}(c N_{m-1}^{.x_1\dots x_{m-1}})$.
Recall from Section~\ref{sec:Bary} that $\mu=B^{\frac{2}{\beta}}$, so we are (for now) replacing the parameter $\beta$ by 
\begin{equation}\label{cbeta}
c=B^{\frac{2}{\beta}-d}\in(B^{-d},\infty).
\end{equation}
Note that $c=1$ corresponds to the critical case $d=\frac{2}{\beta}$ in Section~\ref{sec:Bary}, which will be of particular interest to us. 

\begin{lemma}\label{monotone}
If $B^{-d}<c_1<c_2$, then we can construct $B$-ary CSTRMs $\nu_i$, $i=1,2$ corresponding to $\mu=\mu_i=c_iB^d$, so that $\supp(\nu_1)\subset \supp(\nu_2)$.
\end{lemma}
\begin{proof} We may build two collections of i.i.d.\ Poisson ($\mu_i$) r.v.'s $\{Z^f_i:f \in F\}$ ($i=1,2$) which are coupled so that $Z^f_1\le Z^f_2$ for all $f\in F$.  For $i=1,2$ let $\nu_i$ be the $B$-ary CSTRM constructed in \eqref{nudef} and Section~\ref{sec:Bary} using $\{Z^f_i:f\in F\}$ and the same i.i.d.\ collection $\{X^f:f\in F\}$.  
The conclusion is now clear from Lemma~\ref{Cantor}. (Alternatively, Theorem~\ref{nuprops} (d), using the simpler notation in \eqref{BaryS}, shows $\nu_1\le \nu_2$ a.s.)
\end{proof} 

\begin{proposition}\label{prop:Anucoupling} If $B\in \N^{\ge 2}$ and $c>B^{-d}$ there is a $B$-ary CSTRM, $\nu$, satisfying \eqref{Poisson} and an fp-Cantor set, $A_\infty$, with parameters $(B,1-e^{-c})$, both defined on the same probability space such that $A_\infty\subset\supp(\nu)$ a.s.
\end{proposition}
\begin{proof}
Assume $B$, $c$ and $\nu$ are as in the statement of the Proposition. 
Let us work on a probability space where the following collection of random variables are defined 
\begin{itemize}
	\item $(Z_i^{x,m} \ : \ i\geq 1, x \in G_m, m\geq 0)$ i.i.d.\ with the same law as $Z$ in \eqref{Poisson},
	\item $(u_{i,j}^{x,m}\ : \ i,j\geq 1, x \in G_m, m\geq 0)$ i.i.d.\ with distribution $\mathrm{Uniform}(\{0,\dots, B-1\}^d)$,
\end{itemize}
all jointly independent.
From those random variables we define two processes $(\widetilde{N}_m^x, \ x\in G_m)_{m\geq 0}$ and $(\widetilde{I}_m^x, \ x\in G_m)_{m\geq 0}$ inductively on $m$ as follows:
\begin{itemize}
	\item $\widetilde{N}_0^0=\widetilde{I}_0^0=1$,
	\item For all $m\geq 1$, for all $x=.x_1\dots x_{m}\in G_{m}$,
		\begin{align*}
		\widetilde{N}_m^x:=\sum_{i=1}^{\widetilde{N}_{m-1}^{.x_1\dots x_{m-1}}}\sum_{j=1}^{Z_i^{.x_1\dots x_{m-1},m-1}} 1(u_{i,j}^{.x_1\dots x_{m-1},m-1}= x_m)
	\end{align*} 
and
	\begin{align*}
		\widetilde{I}_m^x:=\min \left(1,\sum_{i=1}^{\widetilde{I}_{m-1}^{.x_1\dots x_{m-1}}}\sum_{j=1}^{Z_i^{.x_1\dots x_{m-1},m-1}} 1(u_{i,j}^{.x_1\dots x_{m-1},m-1}= x_m)\right).
	\end{align*} 
\end{itemize}
Next define 
\begin{align*}
	B_\infty := \bigcap_{m\geq 1} \bigcup_{\substack{x\in G_m\\\widetilde{N}_m^x>0}}\overline{C_m(x)} \qquad \text{and} \qquad  A_\infty := \bigcap_{m\geq 1} \bigcup_{\substack{x\in G_m\\\widetilde{I}_m^x>0}}\overline{C_m(x)}.
\end{align*}
We need to check the three properties below to get our result:
\begin{enumerate}[label=(\alph*), ref=(\alph*)]
	\item\label{it:Binfty is a strm} $B_\infty$ has the law of $\supp(\nu)$ where $\nu$ is a $B$-ary CSTRM with Poisson reproduction with mean $cB^d$,
	\item\label{it:Ainfty is a fp cantor set} $A_\infty$ has the law of an fp-Cantor set with parameters $(B,1-e^{-c})$,
	\item \label{it:Ainfty subset Binfty}$A_\infty\subset B_\infty$.  
\end{enumerate}
The first assertion follows from Lemma~\ref{lem:evolution number of particles} and \eqref{suppNalpha}. 
Indeed, from Theorem~\ref{nuprops}\ref{it:nuprops:d}, \eqref{numdef2}, and Lemma~\ref{lem:evolution number of particles} we can even define $\tilde{\nu}$ as the a.s. weak limit of $\mu^{-m}\sum_{\substack{x\in G_m, \widetilde N_m^x>0}} \widetilde N_m^x \delta_{x}$ as $m\rightarrow \infty$, so that $\tilde{\nu}$ has the same distribution as $\nu$ and $B_\infty=\supp(\tilde{\nu})$ a.s.\ by \eqref{suppNalpha}.

For \ref{it:Ainfty subset Binfty} a straightforward induction on $m$ shows that $\widetilde{I}_m^x\leq \widetilde{N}_m^x$ for all $x\in G_m$ and $m\geq 0$, which implies the result.

It remains to show \ref{it:Ainfty is a fp cantor set}. 
For that, for any $m\geq 1$ and $x=.x_1\dots x_{m}\in G_m$, we introduce
\begin{align*}
	\tilde{e}_{x,m}:= 1\left(\sum_{j=1}^{Z_1^{.x_1\dots x_{m-1}}} 1(u_{1,j}^{.x_1\dots x_{m-1},m-1}= x_m)>0\right).
\end{align*} 
Now, using the thinning property of Poisson random variables, we get that \[\sum_{j=1}^{Z_1^{.x_1\dots x_{m-1}}} 1(u_{1,j}^{.x_1\dots x_{m-1},m-1}= x_m)\] is $\mathrm{Poi}(c)$ and that all those random variables for different values of $(x,m)$ are independent. 
This implies that $(\tilde{e}_{x,m} \ : \ x\in G_m, m\geq 1)$ are i.i.d.\ Bernoulli with parameter $1-e^{-c}$. 
In addition, we can check by induction that for any  $m\geq 1$ and $x=.x_1\dots x_{m}\in G_m$, 
\begin{align*}
	\widetilde{I}_m^x = \prod_{k=1}^{m}\tilde{e}_{.x_1\dots x_k,k}. 
\end{align*}
This and \eqref{Iande} show that the family of random variables $(\widetilde{I}_m^x : x\in G_m, m\geq 1)$ has the same law as  $(I_m^x : x\in G_m, m\geq 1)$ in the definition of fractal percolation, and hence $A_\infty$ is as in \ref{it:Ainfty is a fp cantor set}. 
\end{proof}

\begin{remark}\label{rem:fpisastrm}
Let $p>B^{-d}$ and set $L=\{0,\dots,B-1\}^d$. 
It is not hard to show that the fp-Cantor set, $A_\infty$, with parameters $(B,p)$, is equal in law to $\supp(\nu)$ for a STRM $\nu$ with $\cL(Z)$ chosen to be $\mathrm{Binomial}(B^d,p)$, $\beta>0$ given by $B^{-1}=\mu^{-\beta/2}$ (here $\mu=E(Z)=B^dp>1$), and for each $k\in \N$, $\vec X_k=(X_{j,k}, j\le k)\in L^k$, with law $Q_k$, is a uniformly chosen vector of $k$ distinct sites in $L$.  
More precisely, for $x\in L^k$ with distinct coordinates,
\[Q_k(\{x\})=[B^d(B^d-1)\times\dots\times (B^d-(k-1))]^{-1}.\]
The lack of independence of the coordinates of $\vec X_k$ means that $\nu$ is not a $B$-ary CSTRM, as defined in Section~\ref{sec:Bary}.  Note also that each $X_{j,k}$ is uniformly distributed over $L$ and  the required exchangeable  property of each $Q_k$ holds, and also that \ref{TC} is true for any $\zeta>0$. 
\SUBMIT{We leave this as an exercise using the notation and arguments above. It is not used in the proofs below.}
\ARXIV{A proof of this claim can be found in Section~\ref{app:subsec:fpisastrm} of the Appendix.}
\end{remark}

\begin{remark}\label{rem:weak disconnectedness result also holds for fp}
A subset of the arguments used to prove Theorem~\ref{thm:weak disconnectedness} will show the same conclusion holds for an fp-Cantor set, $A_\infty$, with parameters $(B,p)$, where we assume $p\in (B^{-d},1)$ to ensure non-triviality. We may again work under $\Q$ where properties \eqref{qdist1} and \eqref{qdist2} still hold (they do for any STRM).  Now the process $(|J_m^{S_m}|)_{m\geq 0}$ is identically $1$ as we have single occupancy in our approximating measures, so we can establish Lemma~\ref{lem:sequence of events happens io} directly as before without Lemma~\ref{lem:brimm}. Theorem~\ref{thm:weak disconnectedness} follows as before.

For  this random Cantor set there are non-trivial connected components  with positive probability whenever $p$ is sufficiently close to $1$ (Theorem~1 of \cite{CCD88}), and so (Remark~\ref{rem:fpisastrm}) we have an example of a STRM where a.a.\ points in its  support are totally disconnected (by the extended Theorem~\ref{thm:weak disconnectedness}) but the support has non-trivial connected components with positive probability. Recall from Remark~\ref{rem:ccomp} that the same two properties hold for a range of $B$-ary CSTRMs, {\it assuming \eqref{pccond}}.  Together these results indicate that such a situation is perhaps more generic than one might think.
\end{remark}
\begin{theorem} \label{thm:perc1} Assume $d\ge 2$ and $c_0>0$ satisfies $1-e^{-c_0}>p_c(\cL^d)$. There is a $B_0=B_0(c_0,d)\in \N^{\ge 2}$ such that if $B\ge B_0$, $c\ge c_0$ and $\nu$ is the $B$-ary CSTRM associated with the Poisson$\,(cB^d)$ distribution, then
\begin{equation}\label{percandnottd}
\P(\supp(\nu) \text{ percolates})>0\text{ and }\P(\supp(\nu)\text{ not TD}\,|\, \nu\neq 0)=1.
\end{equation}
In particular, the above holds if $1-e^{-c_0}>p_c(\Z^d)$.
\end{theorem}
\begin{proof} By Theorem~\ref{thm:limit pc for fp}\ref{it:pc for fp converges as B goes to infinity} there is a $B_0\in \N^{\ge 2}$ such that for $B\ge B_0$ and $c\ge c_0$,
\begin{equation*}
p_c(B,d)<1-e^{-c_0}\le 1-e^{-c}\text{ and }B^{-d}\le B_0^{-d}<c_0\le c.
\end{equation*}
If $\nu$ is as in the Theorem, then the above inequalities and Proposition~\ref{prop:Anucoupling} imply
\[\P(\supp(\nu)\text{ percolates})\ge \P(A_\infty(B,1-e^{-c}))\text{ percolates})>0.\]
Lemma~\ref{lem:01} now implies the last equality in \eqref{percandnottd}. The final assertion of the Theorem holds because (see \cite{FG92}) $p_c(\Z^d)\ge p_c(\cL^d)$.
\end{proof}
It is easy to prove Theorem~\ref{thm:perc2} by reinterpreting the above using our original parameter $\beta=\frac{2\log B}{\log\mu}$ instead of $c$.
\begin{proof}[Proof of Theorem~\ref{thm:perc2}.] If $d\ge 3$, then by \cite{CP85} $p_c(\Z^d)<1/2$ (the reference handles $d=3$ but it then follows trivially for $d>3$). Therefore we may take $c_0=\log 2<1$ in Theorem~\ref{thm:perc1} and let $B_0$ be as in that result. Recalling the connection between $c$ and $\beta$ in \eqref{cbeta}, we see that the condition $c\ge c_0$ in Theorem~\ref{thm:perc1} is equivalent to $0<\beta\le \frac{2}{d}+\veps(d,B)$, where 
\[\veps(d,B)=\frac{2}{d-(|\log(\log 2)|/\log B)}-\frac{2}{d}>0.\]
\ref{it:perc2:a} is now immediate from Theorem~\ref{thm:perc1}.\\
\noindent\ref{it:perc2:b} follows just as above but now use \eqref{pccond} to find $0<c_0<1$ so that $p_c(\Z^2)<1-e^{-c_0}$ and take 
\[\veps(2,B)=\frac{2}{2-(|\log(c_0)|/\log B)}-1>0.\]
\end{proof}

\begin{remark}\label{critpercdef}
If we fix $d,B\in\N^{\ge 2}$ and consider the $B$-ary CSTRM, $\nu_\beta$, associated with the Poisson distribution with mean $\mu=B^{2/\beta}$ it follows from \eqref{cbeta} and Lemma~\ref{monotone} that $\P(\supp(\nu_\beta)\text{ percolates})$ is non-increasing in $\beta>0$.  
		Therefore if 
		$\beta_c(d,B)=\inf\{\beta>0:\P(\supp(\nu_\beta)\text{ percolates})=0\}$, we see that $\supp(\nu_\beta)$ percolates w.p.p. for $0<\beta<\beta_c$, and fails to percolate a.s.\ for $\beta>\beta_c$.  
		Theorem~\ref{thm:totdisc} and Theorem~\ref{thm:perc2} imply that 
\begin{equation*}\text{for }B\ge B_0(d), \quad  \beta_c(d,B)\in (2/d,4/(d+1)],\ \text{for }d\ge 2,
\end{equation*}
		where for $d=2$, the lower bound requires \eqref{pccond}. Moreover  percolation fails at the above upper bound.
\end{remark}		
\noindent{\bf Acknowledgement.}
	E.P.\ thanks Bal\'asz R\'ath for pointing out the potential use of fractal percolation in the study of the original disconnectness problem for super-Brownian motion. 
	We thank John Wierman for discussions on site percolation in the plane.  Thanks also go to two anonymous referees for a careful reading of the manuscript and for suggesting a number of clarifications.
	E.P.'s research was supported by an NSERC Canada Discovery Grant.
	D.S.'s research at UBC was supported by an NSERC Canada Discovery Grant and his
	work at Université Paris Nanterre has been conducted within the FP2M federation (CNRS FR 2036).

\bibliographystyle{plain}
\def\cprime{$'$}

\newpage
\appendix
\ARXIV{
\section{Tables of notation}
\begin{center}
\begin{longtable}{|l|l|l|}\caption{Table of notation for Section~\ref{sec:STRM}}
	\\
\hline 
\endfirsthead
\hline 
\endhead

\hline \multicolumn{3}{|r|}{{Continued on next page}} \\ \hline
\endfoot

\hline 
\endlastfoot
		$F_n$  & set of finite sequences of integers of length $n$, i.e. $\N^{\{1,\dots,n\}}$& \\
		$F_0$  & by convention $F_0=\{0\}$ & \\
		$F$  &  set of finite sequences of integers $F=\cup_{n=0}^\infty F_n$& \\
		$I$  &  set of infinite sequences of integers $I=\N^\N$& \\
		$i|m$  & sequence $i$ restricted to its first $m$ terms $(i_1,\dots, i_m)$& \\
		$\kappa(i,j)$  & generation of the most recent common ancestor of $i$ and $j$& \\
		$Z^f$  & number of children of individual $f\in F$& \\
		$K^f_n$  &  descendants of individual $f$ at generation $n$& \eqref{eq:def K^f_n}\\
		$K_\infty^0$  & set of all the descendants of the root $K_\infty^0:=\cup_{n=1}^\infty K^0_n$& \\
		$K$  & set of infinite lines of descents $K=\{i\in I: i_{j+1}\le Z^{i|j}\text{ for all }j\in\Z_+\}$& \eqref{Kdef}\\
		$D(f)$  & set of possible descendants of $f$, i.e. $D(f)=\{g\in F:g|m=f\}$& \eqref{eq:def D(f)}\\
		$\overline D(f)$  & set of possible infinite line of descent from $f$ i.e.  $\overline D(f)=\{i\in I:i|m=f\}$& \eqref{eq:def D(f)} \\
		$W^f$  & asymptotic mass of descendants of $f$ i.e. $W^f=\lim_{n\to\infty}\mu^{-n}|K_n^f|$& \eqref{Wfprop}\\
		$Y$ & random measure supported on $K$ so that $Y(\overline D(f))=W^f$& \eqref{Ydef1}\\
		\hline
		$p=(p_k)_{k\geq 0}$  & reproduction measure & \\
		$Z$ & random variable with law given by $p$ & \\
		$\mu>1$  & expectation of the reproduction measure & \eqref{pk}\\
		$\beta>0$  & parameter of our model& \\
		$\rho>0$  & displacement exponent, obtained as $\rho =\mu^{-\beta/2}$& \\
		\hline
		$\left(X_{m,k}\right)_{\substack{m\le k \\ 
				m,k\in \N }}$ &   displacement vector & \\
		$Q^s$  & distribution of the displacement vector& \\
		$Q_k$ & symmetric displacement law given $k$ children & \\
		$X_{m,k}^f$ & spatial displacement of the $m$th child of $f$, given that $f$ has $k$ children& \\
		$S^f$  & spatial position of individual $f\in K_\infty^0$& \eqref{Sfdefn}\\
		$S^i$  & limiting position for the line of descent $i\in K$& \eqref{Sidefn}\\
		
		\hline
		$\nu$  & push-forward of the measure $Y$ through the map $i \mapsto S^i$ & \eqref{nudef} \\
		$\nu_n$  & an adapted approximation of $\nu$ & \\
		$\overline{\nu}_n$  & a multiple of $\nu_n$ with integer values & \\
		$\tilde \nu_n$  & another, non-adapted, approximation of $\nu$ & \\
		$\bar{\nu}$ & a generic integer-valued finite random measure on $\R^d$ & \\
		\hline
		$\xi$ & generic Borel map $\R^d \rightarrow (0,1]$ &\\
		$G\xi$ & p.g.f.\ of a random measure $\bar{\nu}$ applied to $\xi$ &  \\
		$G_n\xi$ & p.g.f.\ of the random measure $\bar{\nu}_n$ applied to $\xi$& \\
		\hline
\end{longtable}
\end{center}
\begin{table}[htbp]\caption{Table of notation for Section~\ref{sec:SBM}}
	\begin{tabular}{|l| l |l |}
		\hline
		$C$ & space of continuous $\R^d$-valued paths on the unit interval& \\
		$\pi_t$ & projection map, for $y\in C$ we have $\pi_t(y)=y(t)$& \\
		$y^t$ & stopped path $s\mapsto y(s\wedge t)$& \\
		$C^t$ & set of paths stopped at $t$ & \\
		$y/s/w$ & concatenation of paths $y$ and $w$, joined at time $s$& \eqref{eq:definition concatenation paths}\\
		$\cC_t$ & canonical filtration on $C$ & \\
		$T_{s,t}$ & semigroup of the inhomogeneous process $(B^t, \ t\in[0,1])$& \eqref{eq:def semi group HBM}\\
		$(H_t,t\in[0,1])$ & historical super-Brownian motion& \\
		$\Q_m$ & probability for $(H_t,t\in[0,1])$ started from measure $m$& \\
		$\P^*$ & cluster law for SBM & \\
		$(H_t^*,t\in[0,1])$  & historical branching Brownian motion with rate function $\lambda$ & \eqref{Hrep1}\\
		$\lambda$ & continuous function from $[0,1)$ to $\R_+$, later taken to be $s\mapsto\frac{1}{1-s}$ & \\
		$(X^*_t, t \in [0,1])$ & ordinary BBM associated to $H^*_t$ & \\
		$Q^*_{s,m}$ & measure for historical BBM started at $m$ at time $s$ & \\
		$Q^*$ & shorthand for for $Q^*_{0,\delta_0}$ & \\
		$P^*_{s,\eta}$ & probability for ordinary BBM $X^*$ started at time $s$ from $\eta$ & \\
		$M_F(E)$ &space of finite measures on a metric space $E$  & \\
		$N_F(E)$ & subspace of $M_F(E)$ consisting of $\Z_+$-valued finite measures& \\
		$\theta$ & generic Borel map $\theta: C \rightarrow (0,1]$ & \\
		$\overline{G}_{s,t}\theta(m)$ & p.g.f. of $H_t^*$ under $Q^*_{s,m}$  & \eqref{eq:def overline G st}\\
		$G_{s,t}\xi(x)$ & p.g.f. of $X_t^*$ under $P_{s,\delta_x}$ & \eqref{eq:def G st xi}\\
		$N_t$ & the number of points in the BBM at time $t$ & \\
		\hline
	\end{tabular}
	\label{ton:introduction}
\end{table}
\begin{table}[htbp]\caption{Table of notation for Section~\ref{sec:Bary}}
	\centering 
	
	\begin{tabular}{|l| l |l |}
		\hline
		$d$ & dimension, assumed larger than or equal to $2$ &  \\		
		$B$ &   integer larger than $2$& \\
		$\mu$ &   expected value of the reproduction distribution& \\
		$\beta$ &   exponent $\beta =\frac{2\log B}{\log \mu}$ & \\
		$\rho$ &  rate of decrease of spatial displacements, $\rho=\mu^{-\beta/2}=B^{-1}$& \\
		\hline
		$X^{f|m}$ &   (unscaled) displacement of the particle $f|m$ & \\
		$G_m$ &  level-$m$ grid approximation of $[0,1)^d$, i.e. $G_m = [0,1)^d\cap B^{-m}\Z^d$& \\
		$C_m(x)$ &  $d$-dimensional cube of edge length $B^{-m}$ with ``lower left corner" $x\in G_m$& \\
		$x^{(j)}$ &  $j$th coordinate of $x\in\R^d$ & \\
		$.x_1\dots x_m$ &  point $x\in G_m$ such that for $j=1,\dots,d$, $x^{(j)}=\sum_{\ell=1}^mx_\ell^{(j)}B^{-\ell}$. & \\
		$J_m^x$ &  set of individuals at position $x$ at time $w$& \eqref{Jdefn}\\
		$\Lambda_m$ &  set of level-$m$ closed cubes $\{\overline{C_m(x)}:x\in G_m\}$ & \\
		\hline
	\end{tabular}
	\label{ton:Bary strm}
\end{table}
\newpage
\section{Appendix}
\subsection{Proof of Proposition~\ref{bpshbm}}\label{app:subsec:bpshbm}
We follow the reasoning in Section~3 of \cite{DP91}. 
If $H_1$ is a random measure in $M_F(C)$ define an increasing collection of $\sigma$-fields for $s\in [0,1]$ by
\[\cG^*_s=\sigma(\{1(H_1(A)>0),A\in\cC_s\}).\]
In practice we consider $H_1$ under either $\P^*$ or $\Q_{\delta_0}$. 
If $s\in[0,1)$ we let $H_1^{*,s}$ denote a version of the conditional expectations $\P^*(H_1(\cdot)/((1-s)\gamma/2)\, | \, \cG^{*}_s)$ which is a random measure on $C$.  
For $s<1$ under $\Q_{\delta_0}$ (and hence also under $\P^*$) $r_sH_1=\sum_{i=1}^M e_i\delta_{y_i}$ for some finite collection $\{y_1,\dots,y_M\}$ in $C^s$ (see Theorem~III.1.1 of \cite{Per02} or Proposition~3.5(a) of \cite{DP91}).   
It helps to think of $\cG^{*}_s$ as the $\sigma$-field generated by the set of atoms $\{y_1,\dots,y_M\}$. 

\begin{proposition}\label{thm3.9} Under $\P^*$ there is a c\`adl\`ag version of $(r_sH_1^{*,s},\,0\le s<1)$ which is equal in law to the historical branching Brownian motion $(H^*_s,\, 0\le s<1)$ under $Q^*$. 
\end{proposition}

This is Theorem~3.9(b) in \cite{DP91}. The existence of a c\`adl\`ag version is implicit in the proof. 

The following is a minor modification of Theorem~3.10 of \cite{DP91}.
\begin{proposition}\label{thm3.10}
If $\veps_n\downarrow 0$ where $\veps_n\in(0,1)$, then 
\[\lim_{n\to\infty}\frac{\veps_n\gamma}{2}r_{1-\veps_n}H_1^{*,1-\veps_n}=H_1\quad\P^*-\text{a.s.}\]
\end{proposition}
\begin{proof}
For $s<1$ define $H^s_1$ as $H^{*,s}_1$ but now using $\Q_{\delta_0}$ in place of $\P^*$, that is, $H^s_1$ is a version of $\Q_{\delta_0}(H_1(\cdot)/((1-s)\gamma/2) \, | \, \cG^{*}_s)$ which is a random measure.  Theorem~3.10 of \cite{DP91} implies
\begin{equation}\label{thm310conv}
\lim_{n\to\infty}\frac{\veps_n\gamma}{2}H_1^{1-\veps_n}=H_1\quad \Q_{\delta_0}-\text{a.s.}
\end{equation}
Assume $\phi:C\to\R$ is bounded and continuous. If $\tau_nw=w^{1-\veps_n}$, then
\begin{equation} \label{Hunt}\Bigl|\frac{\veps_n\gamma}{2}(r_{1-\veps_n}H_1^{1-\veps_n})(\phi)-\frac{\veps_n\gamma}{2} H_1^{1-\veps_n}(\phi)\Bigr|
\le \Q_{\delta_0}(H_1(|(\phi\circ\tau_n)-\phi|) \, | \, \cG^*_{1-\veps_n}).
\end{equation}
Note that by dominated convergence we have $H_1(|(\phi\circ\tau_n)-\phi|)\to 0$ a.s.\ as $n\to\infty$,  and also that $H_1(|(\phi\circ\tau_n)-\phi|)\le 2\Vert\phi\Vert_\infty H_1(1)\in L^1$.  
Therefore by Hunt's Lemma (see Theorem~45 in Ch. V  of \cite{DM82}) we have that the right side of \ref{Hunt} approaches $0$ a.s. 
By considering a countable convergence determining set of $\phi$'s (e.g., use Theorem~4.5 in Ch. 3 of \cite{EK}) we can conclude from \eqref{thm310conv} and \eqref{Hunt} that
\begin{equation}\label{gettingr}
\lim_{n\to\infty}\frac{\veps_n\gamma}{2}r_{1-\veps_n}H_1^{1-\veps_n}=H_1\quad \text{a.s.}
\end{equation}
Let $\phi$ be as above. 
By the cluster decomposition \eqref{Poisdec} the above implies that
\begin{equation}\label{clusterconv}
\Q_{\delta_0}\Bigl(\sum_{i=1}^N r_{1-\veps_n}H^i_1(\phi)\, \Bigl| \,\cG^*_{1-\veps_n}\Bigr)\to \sum_{i=1}^N H^i_1(\phi)\text{ as }n\to\infty, \ \Q_{\delta_0}-\text{a.s.,}
\end{equation}
where $N$ and $H^i_1$ are as in \eqref{Poisdec}.
Proposition~3.5(b) of \cite{DP91} with $t=1$ and $\veps\uparrow 1$ in that result shows that $N$ is $\cG^*_{1-\veps_n}$-measurable. 
If we apply \eqref{clusterconv} on the event $\{N=1\}$ we get
\begin{equation}\label{clusterconvergence}
\Q_{\delta_0}(r_{1-\veps_n}H_1(\phi)1(N=1)\, | \, \cG^*_{1-\veps_n})\to H_1(\phi)1(N=1)\text{ as }n\to\infty\ \ \Q_{\delta_0}-\text{a.s.}
\end{equation}
Note that $\Q_{\delta_0}(H_1\in\cdot, N=1)$ and $\P^*(H_1\in\cdot)$ are equivalent measures.  
We claim that 
\begin{equation}\label{clustereq}\Q_{\delta_0}(r_{1-\veps_n}H_1(\phi)1(N=1) \, |\, \cG^*_{1-\veps_n})=\P^*(r_{1-\veps_n}H_1(\phi) \, | \, \cG^*_{1-\veps_n})1(N=1)\ \ \Q_{\delta_0}-\text{a.s.}
\end{equation}
Note that both sides of the above are measurable functions of $H_1\in M_F(C)$.
To prove \eqref{clustereq} it suffices to show that for any $A_1,\dots,,A_k\in\cC_{1-\veps_n}$ 
we have 
\begin{align}\label{condequality}
\int_{\{H_1(A_j)>0, \text{ for }j=1,\dots,k\}}r_{1-\veps_n}&H_1(\phi)1(N=1)\,\dd \Q_{\delta_0}\\
\nonumber&=\int \int_{\{H_1(A_j)>0\text{ for }j=1,\dots,k\}}r_{1-\veps_n}H_1(\phi)\,\dd \P^*1(N=1)\dd \Q_{\delta_0}.
\end{align}
Given the above, a monotone class theorem then gives \eqref{clustereq}. 
By the cluster decomposition \eqref{Poisdec} the left-hand side of \eqref{condequality} equals
\begin{align*}\int_{\{H^1_1(A_j)>0\text{ for }j=1,\dots,k\}}&r_{1-\veps_n}H_1^1(\phi)1(N=1)\dd \Q_{\delta_0}\\
&=\Q_{\delta_0}(N=1)\int_{\{H_1(A_j)>0\text{ for }j=1,\dots,k\}}r_{1-\veps_n}H_1(\phi)1(N=1)\dd \P^*,
\end{align*}
which is the right-hand side of \eqref{condequality}. Now use the equality \eqref{clustereq} in \eqref{clusterconvergence} and the aforementioned equivalence of measures to conclude that
\[\lim_{n\to\infty}\P^*(r_{1-\veps_n}H_1(\phi) \, | \, \cG^*_{1-\veps_n})= H_1(\phi), \ \P^*-\text{a.s.}\]
By considering a countable convergence determining set of $\phi$'s we conclude that 
\[\lim_{n\to\infty}\frac{\veps_n\gamma}{2}r_{1-\veps_n}H^{*,1-\veps_n}_{1}=H_1, \ \P^*-\text{a.s.,}\]
and the proof is complete.
\end{proof}

\begin{proof}[Proof of Proposition~\ref{bpshbm}]
 Propositions~\ref{thm3.9} and \ref{thm3.10} give Proposition~\ref{bpshbm}.
\end{proof}

\subsection{Proof of Proposition~\ref{HBBMmean}.}\label{app:subsec:HBBMmean}
\begin{proof}[Proof of Proposition~\ref{HBBMmean}]
We give a standard construction of $H^*_t$ as an explicit tree-indexed system of BBM's (e.g. see Ch. 8 of \cite{W86} or Ch. 1 of \cite{JB}).  Let $I=\cup_{n=0}^\infty\{0\}\times\{0,1\}^n$, write $|\beta|=n$ if $\beta\in\{0\}\times\{0,1\}^n$, and let $\pi\beta$ denote the parent of $\beta\in I\setminus\{0\}$.  Consider an i.i.d.\ collection $\{W^\beta_\cdot:\beta\in I\}$ of standard $d$-dimensional Brownian motions starting at $0$, and an independent i.i.d.\ collection 
$\{(\tau^\beta_s)_{s\in[0,1)}:\beta\in I\}$ of Poisson processes with rate $\frac{ds}{1-s}$. We inductively (on $|\beta|$) define the death time $T^\beta$ of particle $\beta$ by setting $T^{\pi 0}=0$ (even though $\pi 0$ is not defined) and
for $|\beta|>0$,
\[T^\beta=\inf\{t\ge T^{\pi\beta}:\Delta \tau^\beta_t=1\}.\]
We also call $T^{\pi\beta}$ the birth time of particle $\beta$ and write $\beta\sim t$ iff $T^{\pi\beta}\le t<T^\beta$, and  in this case say $\beta$ is alive at time $t$. Again by induction on $|\beta|$ we define an $\R^d$-valued process, $B^\beta$, stopped at $T^\beta$, by $B^0(t)=W^0(t\wedge T^0)$ and for $|\beta|>0$,
\[B^\beta(t)=\begin{cases} B^{\pi\beta}(t),\ &\text{ if }t<T^{\pi\beta},\\
B^{\pi\beta}(T^{\pi\beta})+\int_0^t1(T^{\pi\beta}\le u<T^\beta)\,\dd W^\beta_u,\ &\text{ if }t\ge T^{\pi\beta}.
\end{cases}\]
If $\cF_t=\sigma(\{\tau^{\beta'}:\beta'\in I\})\vee \sigma(\{W^{\beta'}(s):s\le t, \beta'\in I\})$, then it is easy to check that $B^\beta$ is a standard $d$-dimensional $\cF_t$-Brownian motion stopped at the $\cF_0$-measurable time $T^\beta$ and starting at $0$.  
Then $H^*_t=\sum_{\beta\sim t}\delta_{B^\beta(\cdot\wedge t)}$ defines a historical branching Brownian motion with law $Q^*_{0,\delta_0}$. This should be clear but can also be shown, for example, by verifying \eqref{brnle} through direct calculation. 
Therefore if $k\in \N$ and $t,\phi$ are as in the Proposition, then
\begin{align*}
Q^*(H^*_t(\phi)1(H^*_t(1)=k))&=\E\Bigl(\E\Bigl(\sum_{\beta\sim t}\phi(B^\beta(\cdot\wedge t))|\cF_0\Bigr)1(|\{\beta:\beta\sim t\}|=k)\Bigr)\\
&=\E\Bigl(\sum_{\beta\sim t}\E\Bigl(\phi(B^\beta(\cdot\wedge t))|\cF_0\Bigr)1(|\{\beta:\beta\sim t\}|=k)\Bigr)\\
&=\E\Bigl(\sum_{\beta\sim t}E_0(\phi(B(\cdot\wedge t\wedge T^\beta)))1(|\{\beta:\beta\sim t\}|=k)\Bigr)\\
&=kE_0(\phi(B^t))Q^*(H^*_t(1)=k).
\end{align*}
In the last line we use the fact that $t<T^\beta$ on $\{\beta\sim t\}$. This gives \eqref{condhmm}, and \eqref{condmm} is then immediate.
\end{proof}.
\subsection{Proof of Proposition~\ref{Qk}}\label{app:subsec:Qk}
We first restate the proposition.  
\begin{proposition}
	Let $\bar\nu$ be a random measure in $N_F(\R^d)$ such that $\P(\bar\nu(1)=k)>0$ for all $k\in\N$.  For each $k\in\N$ there is a unique symmetric probability, $\tilde Q_k$, on $(\R^d)^k$ (symmetry in the $k$ $\R^d$-valued components) such that
	\begin{equation}\label{QKcharapp}
		\int\prod_{i=1}^k\xi(x_i)\dd \tilde Q_k(x_1,\dots,x_k)=\E(\exp\{\bar\nu(\log\xi)\} \, | \, \bar\nu(1)=k)\ \text{for all Borel $\xi:\R^d\to(0,1]$}).
	\end{equation}
	Moreover for all $i\le k$, 
	\begin{equation}\label{Qkmarapp}
		\int 1(x_i\in\cdot)\dd \tilde Q_k(x_1,\dots,x_k)=\frac{1}{k}\E(\bar\nu(\cdot) \, | \,\bar\nu(1)=k).
	\end{equation}
\end{proposition}

\begin{proof} Fix $k\in\N$ and work under $\P(\, \cdot\, | \,\nu(1)=k)$.  
Let $\cS=\{x\in\R^d:\bar\nu(\{x\})\ge 1\}$ be the support of $\bar\nu$, order the (at most $k$) points in $\cS$ lexicographically, and repeat each $x\in\cS$ $\bar\nu(\{x\})$ times to construct random variables $\hat X_1,\dots,\hat X_k$.
In words $\hat X^1$ is ``the lowest of the left-most points" in $\cS$ and one can easily check measurability.  
$\hat X_2$ is defined as $\hat X_1$ but for $\bar\nu-\delta_{\hat X_1}$. Proceeding inductively we see these are all r.v.'s.  
Let $\pi$ be an independent random permutation chosen uniformly from $S_k$ and set $X_i=\hat X_{\pi i}$ for $i=1,\dots,k$. 
Then $\{X_i:i\le k\}$ are exchangeable and $\bar\nu=\sum_{i=1}^k\delta_{X_i}$.
Let $\tilde Q_k$ be the law of $(X_1,\dots,X_k)$.  
Then $\tilde Q_k$ is symmetric and for $\xi$ as above, 
	\begin{align*}\int\prod_{i=1}^j\xi(x_i)\dd \tilde Q_k(x_1,\dots,x_k)&=\E\Bigl(\exp\Bigl\{\sum_{i=1}^k\log\xi(X_i)\Bigr\}\, \Bigl| \, \bar\nu(1)=k\Bigr)\\
		&=\E(\exp\{\bar\nu(\log\xi)\}\, | \, \bar\nu(1)=k).
	\end{align*}
	
For uniqueness define an equivalence relation $\sim$ on $(\R^d)^k$ by $(x_1,\dots,x_k)\sim(y_1,\dots,y_k)$ iff for some $\pi\in S_k$, $y_i=x_{\pi i}$ for all $i\le k$ (write $y=\pi x$).  
Now apply Stone-Weierstrass on the locally compact space $(\R^d)^k/\sim$ (see Ex XI.6.6 on p. 255 of \cite{Dug}) to the algebra of linear combinations of $\prod_{i=1}\xi(x_i)$, where $\xi:\R^d\to(0,1]$ are continuous functions vanishing at infinite (i.e., is in $C_0(\R^d,(0,1])$). 
This shows that \eqref{QKcharapp} uniquely determines $\int\phi\, \dd \tilde Q_k$ for all symmetric functions $\phi(x_1,\dots x_k)\in C_0((\R^d)^k,\R)$. If $\psi\in C_0((\R^d)^k,\R)$, then $\phi(x):=\frac{1}{k!}\sum_{\pi\in S_k}\psi(\pi x)$ is symmetric in $C_0((\R^d)^k,\R)$.  
The symmetry of $\tilde Q_k$ shows $\int \psi\,\dd \tilde Q_k=\int \phi\,\dd \tilde Q_k$. 
The uniqueness result follows.
	
Turning to \eqref{Qkmarapp}, let $i\le k$ and $\phi:\R^d\to\R$ be bounded Borel. Then
	\begin{align*}
		\int\phi(x_i)\,\dd \tilde Q_k=\int \frac{1}{k}\sum_{j=1}^k\phi(x_j)\,\dd \tilde Q_k=\E\Bigl(\frac{1}{k}\sum_{j=1}^k\phi(X_j)\, \Bigl| \, \bar\nu(1)=k\Bigr)=\frac{1}{k}\E(\bar\nu(\phi) \, | \, \bar\nu(1)=k).
	\end{align*}
\end{proof}
\subsection{Proof of Lemma~\ref{Yfequiv}}\label{app:subsec:Yfequiv}
We first restate the Lemma.
\begin{lemma}\label{Yfequivapp}For any $f\in F_m$, $Y^f=1(f\in K^0_m)\mu^{-m}Y_{[f]}$.
\end{lemma}
\begin{proof}
Assume $n\in\Z_+$ and $g\in F_n$.  Then
\[\overline D(f)\cap\{i\in I:T^mi\in \overline D(g)\}=\overline D(f\vee g),
\]
and so 
\begin{equation}\label{Yf1}
Y^f(\overline D(g))=Y(\overline D(f\vee g))=1(f\vee g\in K^0_{m+n})W^{f\vee g}.
\end{equation}
For $N\ge n$ we introduce the $f$-shifted analogues of $K^g_N$ and $W^g$, that is,  
\begin{equation}\label{Knshift}
K^g_{[f],N}=\{h\in F_N:h|n=g, h_{j+1}\le Z^{h|j}_{[f]}=Z^{f\vee (h|j)}\text{ for }n\le j<N\},
\end{equation}
and
\begin{equation*}
W^g_{[f]}=\lim_{N\to\infty} \mu^{-N}|K^g_{[f],N}|.
\end{equation*}
We have 
\begin{align}\label{fgK}
\nonumber f\vee g\in K^0_{m+n}&\iff f\in K^0_m\text{ and }g_{j-m+1}\le Z^{(f\vee g)|j}\text{ for }j=m,\dots,m+n-1\\
 &\iff f\in K^0_m\text{ and }g_{j+1}\le Z^{f\vee(g|j)}\text{ for }j=0,\dots,n-1\\
\nonumber&\iff f\in K^0_m\text{ and }g\in K^0_{[f],n}.
\end{align}
In addition, 
\begin{align}\label{Wshift2}
\nonumber W^{f\vee g}&=\lim_{N\to\infty}\mu^{-N}|K^{f\vee g}_N|\\
\nonumber &=\lim_{N\to\infty}\mu^{-N} |\{h\in F_N:h|(m+n)=f\vee g, h_{j+1}\le Z^{h|j}\text{ for }m+n\le j<N\}\\
\nonumber &=\lim_{N\to\infty}\mu^{-N} |\{h\in F_N:h|m=f, (T^mh)|n=g, \\
\nonumber &\phantom{=\lim_{N\to\infty}\mu^{-N} |\{h\in F_N:h_1=}(T^mh)_{j+1}\le Z^{f\vee (T^mh|j))}\text{ for }n\le j<N-m\}|\\
\nonumber &=\mu^{-m}\lim_{N\to\infty}\mu^{-(N-m)}|\{h'\in F_{N-m}:h'|n=g,\\
\nonumber &\phantom{=\lim_{n\to\infty}\mu^{-n} |\{h\in F_n:h_1=\ }h'_{j+1}\le Z^{f\vee (h'|j)}\text{ for }n\le j<N-m\}|\\
\nonumber &=\mu^{-m}\lim_{N\to\infty}\mu^{-(N-m)}|K^g_{[f],N-m}|\quad\text{(by \eqref{Knshift})}\\
&=\mu^{-m} W^g_{[f]}.
\end{align}
Clearly \eqref{Yf1}, \eqref{fgK}, \eqref{Wshift2}, and the definition of $Y$ from \eqref{Ydef1} imply
\[Y^f(\overline D(g))=1(f\in K^0_m)1(g\in K^0_{[f],n})\mu^{-m}W_{[f]}^g=1(f\in K^0_m)\mu^{-m}Y_{[f]}(\overline D(g)).\]
The result now follows, just as in the uniqueness proof of Proposition~\ref{Y}.
\end{proof}

\subsection{Proof of Remark~\ref{rem:fpisastrm}.}\label{app:subsec:fpisastrm} 
We repeat the statement of the Remark for convenience.
\begin{remark}\label{rem:fpisastrmapp}
	Let $p>B^{-d}$ and set $L=\{0,\dots,B-1\}^d$. 
	It is not hard to show that the fp-Cantor set, $A_\infty$, with parameters $(B,p)$, is equal in law to $\supp(\nu)$ for a STRM $\nu$ with $\cL(Z)$ chosen to be $\mathrm{Binomial}(B^d,p)$, $\beta>0$ given by that $B^{-1}=\mu^{-\beta/2}$ (here $\mu=E(Z)=B^dp>1$), and for each $k\in \N$, $\vec X_k=(X_{j,k}, j\le k)\in L^k$, with law $Q_k$, is a uniformly chosen vector of $k$ distinct sites in $L$.  
	More precisely, for $x\in L^k$ with distinct coordinates,
	\[Q_k(\{x\})=[B^d(B^d-1)\times\dots\times (B^d-(k-1))]^{-1}.\]
\end{remark}
\begin{proof} If $Z$ and $\vec{X}=\{\vec X_k:k\in\N\}$ are independent and  are distributed as above (the joint law of $\{\vec X_k\}$ is irrelevant), then a moment's thought, or an elementary calculation, shows that $(e_x)_{x\in L}$ defined so that 
\begin{align}\label{Bernoulliconst}
\forall x\in L,\qquad e_x:=1(x\in\{X_{j,Z}:j\le Z\})=\sum_{j=1}^Z1(X_{j,Z}=x),
\end{align}
is an independent collection of $\mathrm{Bernoulli}(p)$ r.v.'s.

Let $\nu$ be the STRM described above, as constructed in Section~\ref{sec:STRM}. Although it is not a $B$-ary CSTRM, $\nu_m$ is still carried by $G_m$ and so we can still use the notation $N^x_m$ defined in \eqref{Nmxdef} and the representation for $\nu_m$ in \eqref{numdef2}.  
Also the boundedness of $Z$ and uniform boundedness of the coordinates of $\vec X=(X_{j,k}, j\le k)$ allow the proof of Lemma~\ref{Cantor} to go through unchanged and so \eqref{suppNalpha} remains valid.
In addition the natural analogue of \eqref{eq:evolution Jmx} holds, namely
\begin{equation}\label{Nmfp}
N^{.x_1\dots x_m}_m=|J^{.x_1\dots x_m}_m|=\sum_{f\in J_{m-1}^{.x_1\dots x_{m-1}}}\sum_{\ell=1}^{Z^f}1(X_{\ell,Z^f}^f=x_m),
\end{equation}
where $.x_1\dots x_{m-1}=0$ if $m=1$. 
A simple induction on $m$ now shows that $N_m^x=0\text{ or }1$ for all $x\in G_m$ and $m\in\Z_+$. As in Lemma~\ref{lem:evolution number of particles}, \eqref{Nmfp} allows us to obtain the analogue of \eqref{eq:distribution number of particles},
\begin{equation}\label{fpinddef}N^{.x_1\dots x_m}_m=\sum_{i=1}^{N_{m-1}^{.x_1\dots x_{m-1}}}\sum_{j=1}^{Z^{.x_1\dots x_{m-1},m-1}} 1(X_{j,Z^{.x_1\dots x_{m-1},m-1}}^{.x_1\dots x_{m-1},m-1}= x_m),
\end{equation} 
where the $(Z^{y,m-1} : \ y\in G_{m-1},\ m\in\N)$ are i.i.d.\ with the same law as $Z$, the $(X_{j,k}^{y,m-1} :\ j\le k\in\N)_{y\in G_{m-1},m\in\N}$ are i.i.d.\ copies of $\vec{X}$, and the two collections are independent. The fact that $N^y_{m-1}=0\text{ or }1$ means we do not need multiple i.i.d.\ copies of  $(Z^{y,m-1},y\in G_{m-1}, m\in\N)$ or $(\vec{X}^{y,m-1},y\in G_{m-1}, m\in\N)$ as before. Define 
\[e_{.x_1\dots x_m,m}=\sum_{j=1}^{Z^{.x_1\dots x_{m-1},m-1}} 1(X_{j,Z^{.x_1\dots x_{m-1},m-1}}^{.x_1\dots x_{m-1},m-1}= x_m),\]
so that by \eqref{fpinddef} and induction,
\[N^{.x_1\dots x_m}_m=\prod_{k=1}^me_{.x_1\dots x_k,k}.\]
Clearly \eqref{Bernoulliconst} and the independence properties of $\{Z^{y,m-1}, \ \vec{X}^{y,m-1} \ : \  y\in G_{m-1},\,m\in\N\}$ imply that $\{e_{x,m}:x\in G_m,m\in\N\}$ are i.i.d.\ Bernoulli r.v.'s with parameter $p$.  
Let $A_\infty$ be the associated fp-Cantor set with parameters $(B,p)$.  
By \eqref{eq:Ainfty as intersection} (note that $N^x_m$ is the $I^x_m$ appearing there) and \eqref{suppNalpha} we have w.p.$1$,
\[A_\infty=\bigcap_{m=1}^\infty\bigcup_{\substack{x\in G_m \\ N_m^x>0}}\overline{C_m(x)}=\supp(\nu),\]
and the derivation is complete.
\end{proof}}

\end{document}